\documentclass[10pt]{amsart}
\usepackage{amssymb}
\usepackage[varbb, varg, smallerops]{newtxmath}
\usepackage[scaled=0.87]{inconsolata} 
\usepackage[T1]{fontenc}

\usepackage{amsthm}
\usepackage{mathrsfs}  
\usepackage{amssymb}
\usepackage{mathtools}
\usepackage{amsmath}
\usepackage{yfonts}
\usepackage{amsfonts}
\usepackage{aliascnt}
\usepackage{graphicx}
\usepackage{tikz}
\usetikzlibrary{shapes.geometric}

\usepackage[a4paper,margin=1.1in]{geometry} %MARGINI MIGLIORI
\usepackage{tikz}
\usepackage{tikz-cd}
\usetikzlibrary{arrows}
\tikzcdset{arrow style=tikz,
diagrams={>={Stealth[round,length=4pt,width=6pt,inset=3pt]}}}
\usepackage{graphicx}
\usepackage[all,cmtip]{xy}
\makeatletter
\newcommand*{\tightdisplaymath}{\abovedisplayskip\z@\belowdisplayskip\z@}

\usepackage{calligra}

\usepackage[colorlinks,allcolors=red]{hyperref}

\makeatletter
\newcommand\RSloop{\@ifnextchar\bgroup\RSloopa\RSloopb}
\makeatother
\newcommand\RSloopa[1]{\bgroup\RSloop#1\relax\egroup\RSloop}
\newcommand\RSloopb[1]
  {\ifx\relax#1
   \else
     \ifcsname RS:#1\endcsname
       \csname RS:#1\endcsname
     \else
       \GenericError{(RS)}{RS Error: operator #1 undefined}{}{}
     \fi
   \expandafter\RSloop
   \fi
  }
\newcommand\X{0}
\newcommand\RS[1]
  {\begin{tikzpicture}
     [every node/.style=
       {circle,draw,fill,minimum size=1.5pt,inner sep=0pt,outer sep=0pt},
      line cap=round
     ]
   \coordinate(\X) at (0,0);
   \RSloop{#1}\relax
   \end{tikzpicture}
  }

\newcommand\RSdef[1]{\expandafter\def\csname RS:#1\endcsname}
\newlength\RSu
\RSdef{i}{\draw (\X) -- +(90:\RSu) node{};}
\RSdef{g}{\draw (\X) -- +(30:\RSu) node{};} %NUOVO MIO
\RSdef{l}{\draw (\X) -- +(135:\RSu) node{};}
\RSdef{r}{\draw (\X) -- +(45:\RSu) node{};}
\RSdef{I}{\draw (\X) -- +(90:\RSu) coordinate(\X I);\edef\X{\X I}}
\RSdef{G}{\draw (\X) -- +(25:\RSu) coordinate(\X G);\edef\X{\X G}} %NUOVO MIO
\RSdef{H}{\draw (\X) -- +(115:\RSu) coordinate(\X H);\edef\X{\X H}} %NUOVO MIO
\RSdef{L}{\draw (\X) -- +(135:\RSu) coordinate(\X L);\edef\X{\X L}}
\RSdef{R}{\draw (\X) -- +(45:\RSu) coordinate(\X R);\edef\X{\X R}}
\RSdef{D}{\draw (\X) -- +(35:\RSu) coordinate(\X D);\edef\X{\X D}} %NUOVO MIO
\RSdef{A}{\draw (\X) -- +(65:\RSu) coordinate(\X A);\edef\X{\X A}} %NUOVO MIO
\RSdef{B}{\draw (\X) -- +(115:\RSu) coordinate(\X B);\edef\X{\X B}} %NUOVO MIO
\RSdef{C}{\draw (\X) -- +(145:\RSu) coordinate(\X C);\edef\X{\X C}} %NUOVO MIO
\RSdef{d}{\draw (\X) -- +(35:\RSu) node{};} %NUOVO MIO
\RSdef{a}{\draw (\X) -- +(65:\RSu) node{};} %NUOVO MIO
\RSdef{b}{\draw (\X) -- +(115:\RSu) node{};} %NUOVO MIO
\RSdef{c}{\draw (\X) -- +(145:\RSu) node{};} %NUOVO MIO

\newcommand{\MDu}{\RS{{R{r}}{L{l}}}}
\newcommand{\MTr}{\RS{{R{r}}{I{i}}{L{l}}}}
\newcommand{\MQu}{\RS{{D{D{d}}}{A{a}}{C{C{c}}}{B{b}}}}
\newcommand{\MCS}{\RS{{R{r}}{L{l}r}}} %COMPOSIZIONE SINISTRA
\newcommand{\MCD}{\RS{{R{l}{r}}{L{l}}}} %COMPOSIZIONE DESTRA
\newcommand{\Q}{\mathcal{Q}} %Quivers
\newcommand{\dQ}{d_{\Q}}
\newcommand{\Ain}{\mbox{A}_{\infty}} %Quivers
\newcommand{\Id}{\mbox{Id}}

\newcommand{\T}{\mathfrak{t}}

\newcommand{\TT}{\textswab{F}}
\newcommand{\TTn}{\tiny\textswab{F}}
\newcommand{\mTT}{m_{\TTn(\Q)}}

\newcommand\restr[2]{{% we make the whole thing an ordinary symbol
  \left.\kern-\nulldelimiterspace % automatically resize the bar with \right
  #1 % the function
  \littletaller % pretend it's a little taller at normal size
  \right|_{#2} % this is the delimiter
  }}

\newcommand{\littletaller}{\mathchoice{\vphantom{\big|}}{}{}{}}

\newcommand{\aCat}{\mbox{A$_{\infty}$Cat}}

\newcommand{\DgQui}{\mbox{DGQuiv}}
\newcommand{\DgCat}{\mbox{DGCat}}

\newcommand{\HomA}{\mbox{Hom}_{\mathscr{A}}}

\newcommand{\A}{\mathscr{A}}

\newcommand{\HomB}{\mbox{Hom}_{\mathscr{B}}}
\newcommand{\B}{\mathscr{B}}

\newcommand{\Ho}{\mbox{Ho}}
\newcommand{\Fun}{\mbox{Fun}_{\infty}}

\newcommand{\G}{\mathscr{G}}
\newcommand{\F}{\mathscr{F}}
\newcommand{\C}{\mathscr{C}}
\newcommand{\Hom}{\mbox{Hom}}
\newcommand{\dego}{\mbox{deg}}
\newcommand{\sF}{\mathsf{F}}
\newcommand{\sG}{\mathsf{G}}
\newcommand{\Na}{\mbox{N}_{\tiny\Ain}}
\newcommand{\Nd}{\mbox{N}_{\tiny\mbox{DG}}}
\newcommand{\tA}{|\tilde{\A}_{n-1}|}
\newcommand{\Qs}{\Q_{\tilde{S}_n}}

\newcommand{\sem}{\tiny\mbox{sf}}

\newcommand{\vers}{The Homotopy Theory of $\Ain$categories}
\title[\vers]{The Homotopy Theory of $\Ain$categories} 

\author{Mattia Ornaghi}  

\usepackage{hyperref}
\hypersetup{colorlinks=false}

\address{\parbox{0.9\textwidth}{Universit\`a degli Studi di Milano\\
Dipartimento di Matematica\\
Via Cesare Saldini 50, 20133 Milano, Italy}}
\email{mattia12.ornaghi@gmail.com}

\subjclass[2020]{14F08, 18E35, 18G70}

\thanks{The author was supported by the research project FARE 2018 HighCaSt (grant number R18YA3ESPJ) and by ERC Advanced
Grant - CUP: G43C23001750006 - HE-ERC23ANEEM-01}

\theoremstyle{definition}
\newtheorem{defn}{Definition}[section]
\newtheorem{cons}{Construction}[section]
\newtheorem{defnthm}{Definition-Theorem}[section]
\newtheorem{thm}{Theorem}[section]

\newtheorem{lem}[thm]{Lemma}

\newtheorem{prp}[thm]{Proposition}
\newtheorem{cor}[thm]{Corollary}

\newtheorem*{namedthmA}{Theorem A}
\newtheorem*{namedthmB}{Theorem B}
\newtheorem*{namedthmBB}{Theorem B$'$}
\newtheorem*{namedthmC}{Theorem C}

\newtheorem*{namedlemD}{Lemma D}
\theoremstyle{remark}
\newtheorem{rem}{Remark}[section]
\newtheorem{exmp}{Example}[section]
\newtheorem{notatu}{Notation}[section]

\begin{document}

\date{\today}

\maketitle

\begin{abstract}
In this paper we describe the homotopy category of the $\Ain$categories.\ 
To do that we introduce the notion of semi-free $\Ain$category, 
%inspired by the notion of relatively free $\Ain$categories of Lyubashenko-Manzyuk, 
which plays the role of standard cofibration.\ 
Moreover, we define the non unital $\Ain$ (resp. DG)categories with cofibrant morphisms 
and we prove that any non unital $\Ain$ (resp. DG)category has a resolution of this kind.  
\end{abstract}

\setcounter{tocdepth}{1}

\tableofcontents

\section*{Introduction and Statement of Results}

The history of $\Ain$categories has roots in the foundational work of Jim Stasheff, who, in the 1960s, made significant contributions to understanding higher homotopy operations.\ His work focused on providing a combinatorial description of the Stasheff polytopes, revealing the algebraic structure underlying higher associativity relations.\ Stasheff's ideas served as a precursor to the development of $\Ain$structures \cite[II]{Sta}.\ 

In the early 1990s, it was introduced the concept of $\Ain$algebra and it was extended to the categories thanks to the work of many people:\ Kontsevich \cite{Kon}, Soibelman \cite{KS3}, Fukaya \cite{Fuk}, Oh, Ono, Ohta \cite{FOOO1}, Seidel \cite{Sei}, etc.\
In fact, $\Ain$categories became fundamental in the context of Homological Mirror Symmetry, since, the Fukaya category of a symplectic manifold, comes naturally with a structure of $\Ain$category.\ Or, more precisely, a structure of $\Ain$pre-category.\

Nowadays $\Ain$categories have applications in many other fields of mathematics and physics: String Theory, Moduli Spaces, Deformation Theory, Algebraic Topology, Categorical Homotopy Theory, Homological Algebra, etc.\

In a few words, an $\Ain$category is a non associative $R$-linear DG category, whose cohomology is a $R$-linear (graded) category.\
Let us denote by $\aCat$ the category of $\Ain$categories with $\Ain$functors, linear over a commutative ring $R$.\
In this framework, we have a very natural notion of quasi-equivalence, which is a functor between $\Ain$categories inducing an equivalence between the underlying cohomology categories.\
As in the case of DG categories, what we actually care is not the (1-)category of $\Ain$categories, rather its homotopy category $\mbox{Ho}(\aCat)$.\ Namely, the localization of the category of $\Ain$categories with respect to the quasi-equivalences.\

In general, the fundamental tool for investigating the homotopy category of a category is, of course, the model structure.\
Model structures were introduced by Quillen, in the 1960s, as a way to define the hom-spaces of homotopy categories in terms of cofibrant-fibrant resolutions and homotopy relations.\

For this reason, one would provide a model structure on the category $\aCat$.\ Unfortunately the nature of the $\Ain$functors makes this undertaking very hard (see \cite{COS1}).\
In 2003, Lefevre-Hasegawa managed to equip the category of one object $\Ain$categories (namely the $\Ain$algebras), linear over a field, with a model structure without limits \cite{LH}.\ 
See also the more recent work \cite{Val}.\

The aim of this present work is to give a very explicit description of $\mbox{Ho}(\aCat)$, the homotopy category of $\aCat$ (linear over a commutative ring $R$) in terms of semi-free resolutions and weakly equivalences, avoiding the model structures.\\
In particular, our first main result is the following:

\begin{namedthmA}[\ref{artoo}]
Given two strictly unital $\Ain$categories $\A,\B$ we have a natural bijection:
\begin{align}\label{frutott}
\mbox{Ho}(\aCat)(\A,\B)\simeq^{\dagger_1}\mbox{Ho}(\aCat)(\A^{\tiny\mbox{sf}},\B)\simeq^{\dagger_2}\aCat^{\tiny\mbox{(cu)}}(\A^{\tiny\mbox{sf}},\B)/_{\approx}.
\end{align}
Where ${\mbox{(cu})}$ can be omitted and denotes the set of cohomological unital $\Ain$functors.
\end{namedthmA}

Let us explain formula (\ref{frutott}) in the details: $\A^{\tiny\mbox{sf}}$ denotes the semi-free resolution of $\A$.\\
The notion of semi-free $\Ain$category is new in the literature (see Definition \ref{semifrollo}).\ We introduce it in section \ref{semifreeee} and it comes, quite naturally, from the definition of relatively free $\Ain$category, due to Lyubashenko-Manzyuk \cite[5.1 Definition]{LM2}, and from the definition of semi-free DG category, due to Drinfeld \cite[B.4]{Dri}.\ We prove the following:

\begin{namedthmB}[\ref{semifree}+\ref{h-proj}+Remark \ref{functororne}]
Every strictly unital $\Ain$category $\A$ admits a semi-free resolution (it can be made functorial in $\aCat_{\tiny\mbox{strict}}$) which has cofibrant morphisms (see Theorem \ref{cofmor}).\\
In particular there exists a semi-free $\Ain$category $\A^{\tiny\mbox{sf}}$ with a (surjective on the morphisms) strict quasi-equivalence $\Psi:\A^{\tiny\mbox{sf}}\to\A$.
\end{namedthmB}

%This construction can be made functorial in $\aCat_{\tiny\mbox{strict}}$, see Remark \ref{functororne}.\
Here it is crucial the construction of free $\Ain$categories and of quotients by $\Ain$ideals.\ 
The first is due to Kontsevich \cite{Kon}, the latter is due to Lyubashenko and Manzyuk \cite{LM2}.\ 
We dedicate the whole section \ref{Liberone} to these constructions.\
Clearly Theorem above provides the isomorphism $\dagger_1$ in (\ref{frutott}).\\

It remains to explain the natural isomorphism $(\dagger_2)$.\ The relation $\approx$, is an equivalence relation on the set of $\Ain$functors with same source and target.\
We write $\F\approx\G$ if they are weakly equivalent (see Definition \ref{def1}).\ 

To understand this equivalence relation it is necessary the following fact:\
%Fixing two $\Ain$categories, there are two different equivalence relations on the hom-space cf. 
%We that there are two 
Fixing two $\Ain$categories $\A$ and $\B$, the $\Ain$functors $\A\to\B$ form an $\Ain$category.\
In this category, the objects are the $\Ain$functors, and the morphisms are the prenatural transformations (cf. Definition \ref{trantran}).\
Two $\Ain$functors are weakly equivalent if there are two prenatural transformations making them quasi-isomorphic.\ 

In a few words, a prenatural transformation $T:\F\Rightarrow \G:\A\to\B$, can be considered as a 2-morphism:
\begin{align*}
\xymatrix@R=1em{
&  \ar@{=>}[dd]_{T}&\\
\A\ar@/^2pc/[rr]^-{\F}\ar@/_2pc/[rr]_-{\G}&&\C\\
&&
}
\end{align*}
We prove an important relation between the weakly equivalences and the homotopy classes: 
\begin{namedthmC}[\ref{mspaces}]
Given two $\Ain$categories $\A$ and $\B$ if $\F\approx\G$ then $[\F]=[\G]$ in $\mbox{Ho}(\aCat)$.
\end{namedthmC}
To prove Theorem \ref{mspaces} we use a slightly modification of the $\Ain$nerve.\ 
In this case the $\Ain$nerve is not only an $\infty$-category but it is enriched over $\Ain$categories (see Lemma \ref{LEMMOMBA}).\   
\\
We recall that, there is another equivalence relation between two $\Ain$functors: the homotopy one.\ 
We say that $\F$ and $\G$ are homotopic if $\F-\G=\mathfrak{M}^1(H)$, where $H$ is a prenatural transformation.\ 
We denote it by $\F\sim\G$.\ In Lemma \ref{sim} we prove that $\sim$ implies $\approx$.\ 
This is fundamental to pass from cohomological to strictly unital $\Ain$functors.\\
\\
Let us compare our Theorem \ref{artoo} with the \emph{fundamental theorem about model categories}\footnote{quoting Hovey  \cite[pp.13 Theorem 1.2.10]{Hov}} \cite[Theorem 1.2.10]{Hov}.\\
Suppose that $\mathcal{M}$ is a model category whose objects are fibrant\footnote{For example the category DGCat with the model structure of \cite{Tab1}}.\ 
Fixing two objects $X,Y\in\mathcal{M}$, we have natural isomorphisms
\begin{align*}
\mbox{Ho}(\mathcal{M})(X,Y)\simeq\mbox{Ho}(\mathcal{M})(X^{\tiny\mbox{cof}},Y)\simeq\mathcal{M}(X^{\tiny\mbox{cof}},Y)/_{\sim_{l,r}}.
\end{align*}
Here $X^{\tiny\mbox{cof}}$ denotes the cofibrant replacement and $\sim_{l,r}$ is the homotopy relation.\\
So, we can see that semi-free resolutions and $\approx$ play the role, respectively, of cofibrant replacements and homotopy relation in the framework of $\Ain$categories.\
This is very similar to the case of DG categories, where the semi-free resolutions are the standard cofibrant objects (cf. Corollary \ref{cor1}).\  
On the other hand, the interesting thing is that, the relation $\approx$, naturally encodes the existence of a path object (see Remark \ref{pathobj}).\\

Note that in Theorem B we mentioned $\aCat_{\tiny\mbox{strict}}$.\ This is the category of $\Ain$categories with strict $\Ain$functors.\ It will play a significant role in this paper.\
There are, indeed, two notions of functors between $\Ain$categories: the $\Ain$functors (Definition \ref{af}) and the strict $\Ain$functors (Definition \ref{astricts}).\
The first notion is very general, and it is designed to preserve the operations up to homotopy.\ 
The second one, is much more narrow, and it is basically a morphism of the underlying graded $R$-quivers.\
For example, considering two DG categories as $\Ain$categories, the strict $\Ain$functors between them are exactly the DG functors.\

We point out that, the $\Ain$functors, are the right notion of functors between $\Ain$categories (see subsection \ref{HPT}).\
This is because, if $R$ is a field, they are flexible enough to make an $\Ain$category equivalent to its cohomology.\  
In a few words, if $\A$ is an $\Ain$category, we can equip the cohomology category $H(\A)$ with a minimal (i.e. with trivial differentials) $\Ain$structure which is equivalent to $\A$, via $\Ain$functors.\ 
This is no longer true over a commutative ring.\ 
One has to take the antiminimal model and equip the cohomology with a derived $\Ain$structure \cite{Sag}, but this is not the topic of this paper.\ 

The downside of $\Ain$functors is that they make very difficult to work with the category $\aCat$.\ 
For example $\aCat$ has no equalizers (sf. \cite{COS1}).\ 

On the other hand, the category of $\Ain$categories equipped with strict $\Ain$functors, is much better from this point of view.\
In subsection \ref{cocompletone}, we prove that $\aCat_{\tiny\mbox{strict}}$ is complete and cocomplete, and we give a very explicit description of the (small) limits and colimits.\
This is fundamental to prove the existence of a semifree resolution.\
\\
Moreover, we proved that, in $\aCat_{\tiny\mbox{strict}}$, our semi-free $\Ain$resolutions have a \emph{lifting property} \`a la Drinfeld.\ 

\begin{namedlemD}[\ref{lemlift}]
Let $\A^{\tiny\mbox{sf}}$ be a semi-free $\Ain$category and $\B$, $\C$ two $\Ain$categories.\\
Given $\mathsf{F}:\A^{\tiny\mbox{sf}}\to\B$, a strict $\Ain$functor, and $\mathsf{G}:\C\to\B$ a strict quasi-equivalence, surjective on the morphisms, there exists a (non unique) strict $\Ain$functor $\tilde{\mathsf{F}}:\A^{\tiny\mbox{sf}}\to\C$ making the diagram
\[
\xymatrix{
&{\C}\ar@{->>}[d]_{\simeq}^{\mathsf{G}}\\
\A^{\tiny\mbox{sf}}\ar@{-->}[ur]^{\tilde{\sF}}\ar[r]^{\sF}&\B
}
\]
commutative in $\aCat_{\tiny\mbox{strict}}$.
\end{namedlemD}
It is not difficult to see that, in general, given an $\Ain$functor is not always possible to find a (weakly) equivalent strict $\Ain$functor.\ 
One has to consider the category $\aCat/\approx$, whose objects are the $\Ain$categories and the morphisms given by $\aCat/{\approx}(\A,\B):=\aCat(\A,\B)/\approx$ (note that $\approx$ is compatible with the composition).\ 
%The main result of \cite{LM1} is that, if $\A$ is a free $\Ain$category, every $\Ain$functor, with source $\A$, is weakly equivalent to a strict one.\\
%We should say that Lemma \ref{lemlift} is far from being true, if $\sF$ is a non strict $\Ain$functor.\
%So our strategy is to \emph{strictify} the $\Ain$functors with source a semi-free $\Ain$category.\ We prove   
%\begin{namedthm}[\ref{cofibrant}]
%Let $\A^{\tiny\mbox{sf}}$ be a semi-free $\Ain$category.\ Given an $\Ain$functor $\F:\A^{\tiny\mbox{sf}}\to\B$ there exists a strict $\Ain$functor $\sF$ such that $\mathsf{F}\approx \F$.
%\end{namedthm}
%The proof of Theorem \ref{cofibrant} is very long and technical.\ 
%But, in a few words, the idea is that we can get the homotopy $T:\F\Rightarrow\sF$, gluing the homotopies on the free generators.\
%By Theorem \ref{cofibrant} and \ref{lemlift}, 
We have the following weak lifting property: 
%\begin{namedthmDD}[Lemma \ref{dodi}]
Given two $\Ain$functors $\F:\A\to\B$ and $\G:\C\to\B$, such that $\G$ is a quasi-equivalence.\ 
There exists a (non unique) $\Ain$functor $\tilde{\F}:\A^{\tiny\mbox{sf}}\to\C$ making the diagram
\[
\xymatrix{
&{\C}\ar[d]_{\simeq}^{\G}\\
\A^{\tiny\mbox{sf}}\ar@{-->}[ur]^{\tilde{\F}}\ar[r]^{\F}&\B
}
\]
commutative in $\aCat/\approx$.\ We call it "weak", since $\G\cdot\tilde{\F}\approx\F$ (see Lemma \ref{lordod})\\
Moreover the semi-free construction defines a functor:
\begin{align*}
(\mbox{-})^{\sem}_{\approx}:\aCat/\approx &\to \aCat_{\sem}/\approx\\
\A&\mapsto \A^{\sem}\\
%\F&\mapsto \F^{\sem}_{\approx}.
\end{align*}
Where $\aCat_{\sem}$ denotes the full subcategory of $\aCat$ whose objects are the semi-free $\Ain$categories (see Theorem \ref{funtocat}).\
Note that Theorem A does not implies the equivalence between $\Ho(\aCat)$ and $\aCat/\approx$ (this is indeed false, even if $R$ is a field see \cite[Remark 4.14]{COS2}).\\

To conclude, we remark that Theorem B has to do with strictly unital $\Ain$categories.\ 
In subsection \ref{cofibrenzio} we introduce the notion of categories/quivers with cofibrant morphisms.\ 
This notion in some sense "replace" the one of semi-free $\Ain$categories in the cases of non unital, cohomological unital or unital $\Ain$categories.\
In a few words, given $\C$ a DG quiver, (non unital) DG category, (non unital) $\Ain$category, we say that $\C$ has cofibrant morphisms if, for every $x,y\in\C$ 
the DG $R$-module $\Hom_{\C}(x,y)$ is cofibrant (according to Remark \ref{cofibrelli}).\ We prove the following result:
 
\begin{namedthmBB}[\ref{cofmor}+ Remark \ref{functororne}]
Given a $\star$ $\Ain$category (resp. DG category) $\A$ we can find a $\star$ $\Ain$category $\A^{\tiny\mbox{cm}}$ (resp. DG category) with cofibrant morphisms 
(which has the same objects of $\A$), and a (surjective on the morphisms) strict $\Ain$functor $\Psi_{\A}:\A^{\tiny\mbox{cm}}\to\A$ which is a quasi equivalence.\ 
Moreover the construction $\A\mapsto \A^{\tiny\mbox{cm}}$ can be made functorial in $\aCat_{\tiny\mbox{strict}}^{\star}$ and make the diagram
\label{cofmor}
\begin{align}
\xymatrix{
\A^{\tiny\mbox{cm}}\ar[d]^{\Psi_{\A}}\ar[r]^{\sF^{\tiny\mbox{cm}}}&\ar[d]^{\Psi_{\B}}\B^{\tiny\mbox{cm}}\\
\A\ar[r]^{\sF}&\B\\
}
\end{align}
commutative.\
Here $\star\in\mathcal{f} \mbox{non unital, cohomological unital, unital}\mathcal{g}$.
\end{namedthmBB}
 
As before, considering the category $\aCat^{\star}/\approx$, we have a functor
\begin{align*}
(\mbox{-})^{\tiny\mbox{cm}}_{\approx}:\aCat^{\star}/\approx &\to \aCat^{\star}_{\tiny\mbox{cm}}/\approx\\
\A&\mapsto \A^{\tiny\mbox{cm}}.
\end{align*}
%It is important to note that the $\aCat$ in the image of $(\mbox{-})^{\tiny\mbox{cm}}_{\approx}$ have a lifting property 

%which are the left inverses of inclusions.
 
\subsection*{Related work and further applications:}

This article could be considered as a side/complementary project of \cite{COS2} by the author joint with Canonaco and Stellari.\ 
The aim of that paper was twofold.\ First: we proved that the homotopy categories of $\DgCat$ and $\aCat$ 
(with any possible notion of unit) are equivalent, Second the description of the internal homs of $\Ho(\DgCat)$ in terms of $\Ain$functors.\ 
There was a technical point to prove these results, namely \cite[Proposition 4.10]{COS2}, i.e. the isomorphism of the sets 
$$\Ho(\DgCat)(\A,\B)\simeq\aCat^{\tiny\mbox{u}}/\approx(\A,\B),$$
where $\A$ is a h-projective DG category.\ 
In order to do that we used the model structure on $\DgCat$ (see \cite{Tab1}).\ 
In particular the existence of a cofibrant replacement of DGCat and the existence of the path object (see \cite[Lemma 4.3-4.8-4.9]{COS2}).\
This current paper started from the very natural question if one could prove the same isomorphism avoiding the model structure of $\DgCat$.\ 
The answer is not only positive but our Theorem A is a genuine improvement, see also Lemma \ref{rorororoor}.\\
We give now a few of applications of our results: \\
\begin{itemize}
\item[1.] Our construction $\A^{\tiny\mbox{cm}}$ is used in \cite[\S 3.3]{COS2} to prove that the category of (homotopy theory of) 
cohomological unital $\Ain$categories is equivalent to the (homotopy theory of) unital ones.
%(cf. \cite[]{COS2})\cite{COS2} the passage from cohomological unital to unital.\\
\item[2.] In \cite[\S 1.2]{KS3} Kontsevich and Soibelman described the Hochschild Cohomology of an $\Ain$cat $\A$ (over a field) 
as
\begin{align*}
\mathbb{HC}^n(\A):=H^n\big( \mbox{Ext}(\Id_{\A},\Id_{\A}) \big)
\end{align*}
So, over a commutative ring, we can describe the Hochschild Cohomology of $\A$ as:
\begin{align*}
\mathbb{HC}^n(\A):=H^n\big( \Fun(\A^{\tiny\mbox{sf}},\A)(\Psi_{\A},\Psi_{\A}) \big).
\end{align*}
Note that this is not an immediate consequence, we need to prove that $\Ho(\aCat)$ has a closed symmetric monoidal structure, 
this will be done in \cite{Orn2}.\
\item[3.]  The construction of Theorem \ref{cofmor} can be extended to $\Ain$precategories.\ 
Together with Canonaco and Stellari, we will use these resolutions to approach Kontsevich-Soibelman Conjectures 4 and 5 \cite[Conjectures 4 and 5]{KS4}.
\end{itemize}

We conclude by saying that in this paper we do not treat the various $\infty$-version on this topic.\ 
Namely, one can consider as \emph{homotopy theory} of $\aCat$ the simplicial localization (see \cite{DK}) of $\aCat$ instead of the usual localization.\ 
The two localizations are related as follows: denoting by $L(\aCat)$ the simplicial localization, we have
$$\pi_0(L(\aCat))\simeq\Ho(\aCat).$$
Roughly speaking $L(\aCat)$ is a simplicial enrichment of $\Ho(\aCat)$.\ 
It means that, fixed two $\Ain$categories $\A$ and $\B$, $L(\aCat)(\A,\B)$ is a simplicial set.\\
In the case of DG categories, it makes sense to consider $L(\DgCat)$ instead of $\Ho(\DgCat)$ (see \cite[pp. 617]{Toe}).\
This is because fixed two DG categories $\A$ and $\B$, it is difficult to say what is $\DgCat(\A,\B)$ (see section \ref{lonzolenzo}).\
On the other hand, the category $\aCat$ has a natural enrichment on $\aCat$, 
namely 
$$\aCat(\A,\B)$$
is an $\Ain$category, fixed two $\Ain$categories $\A$ and $\B$ and so $\aCat(\A,\B)/\approx$.\
It relies on the globular nature of $\aCat$ (see \cite{Lyu}).

\subsection*{Notation} 

Throughout the paper $R$ denotes a commutative ring.\ We assume that all our categories are small, in an appropriate universe $\mathbb{U}$.\
We denote by $\Hom_{\C}(x,y)$ the \emph{hom-space between $x$ and $y$}.\ 
Sometimes we simply denote it by $\Hom(x,y)$, omitting the category $\C$, if it is clear from the context.\ 
If $\C$ is the category of (DG, $\Ain$, etc) categories we denote the hom-space by $\C(X,Y)$.\
Moreover, if $\C$ is a 2-category, we denote by $F\Rightarrow G$ a 2-arrow between two 1-arrows $F$ and $G$.\ % and, as usual, $:x\to y$ a 1-modphism.
Given a graded complex $M$, and an integer $n$, we denote by $M[n]$ the graded complex shifted by $n$.\
Given two DG $R$-modules $M$ and $N$, $\mbox{Hom}(M,N)$ denotes the DG module of morphisms (in the DG category of DG $R$-modules).\ 
In particular, $h^N$ denotes the contravariant hom-functor, i.e. $h^N(M):=\mbox{Hom}(M,N)$.

\subsection*{Acknowledgements}

I am very grateful to Francesco Genovese for all the stimulating discussions about DG categories and higher categories.\ 
I thank Alberto Canonaco and Paolo Stellari for all the valuable conversations I benefitted during the writing of this paper.\ 
I want to express my gratitude to Volodymyr Lyubashenko and Paul Seidel for the useful clarifications about their results.\
%Finally, I thank Francesco de Vecchi for the aesthetic advice regarding graph theory and trees.
%Moreover Francesco de Vecchi for the useful suggestions about the trees and .

\newpage

\section{Background}

\subsection{Glossary of Graph Theory}\label{gloglossary}

We assume that the reader is familiar with the concepts of \emph{Planar Rooted Trees} and the operation of \emph{grafting}.\ A precise treatment of these notions can be found in \cite[Appendix C.2.]{LV}.\
Roughly speaking a \emph{planar rooted tree} is a graph such that each vertex has $n$-inputs and  $1$-output.\
The upper inputs are called \emph{leaves} and the lower output is called \emph{root}, the intermediate vertices are called \emph{nodes}.\ Note that the root counts as a \emph{trivial node}.\ See the picture: 
\begin{figure}[htbp]
\centering
\includegraphics[width=1\textwidth, height=.12\textheight, keepaspectratio]{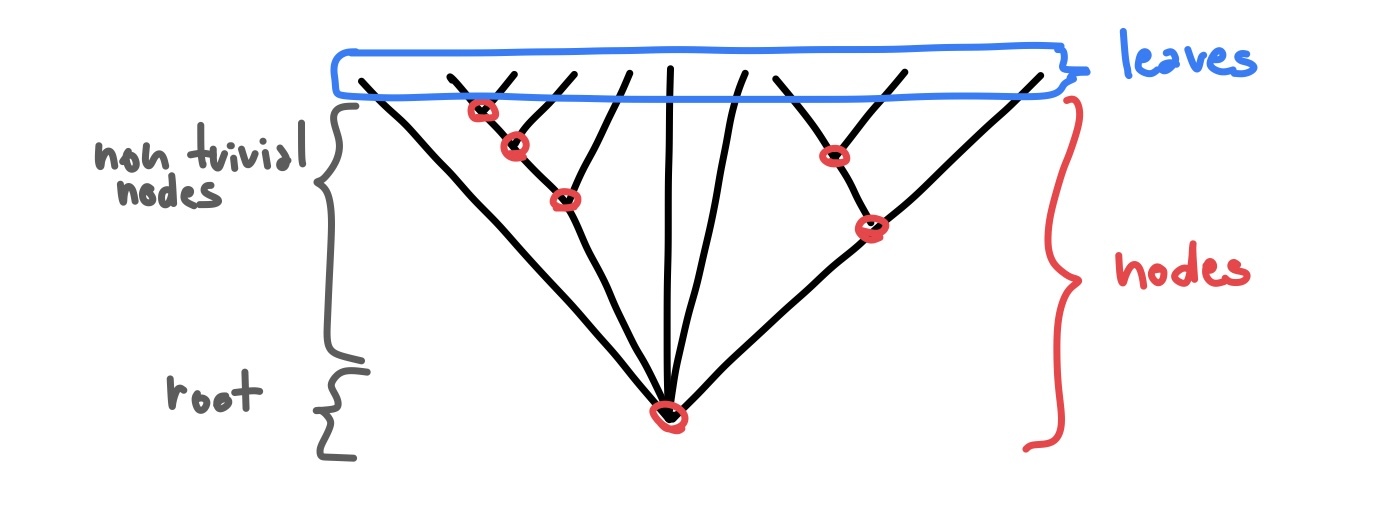}
\end{figure}
\\
Let us draw some examples of planar rooted trees:
\begin{figure}[htbp]
\centering
\includegraphics[width=.6\textwidth, height=.3\textheight, keepaspectratio]{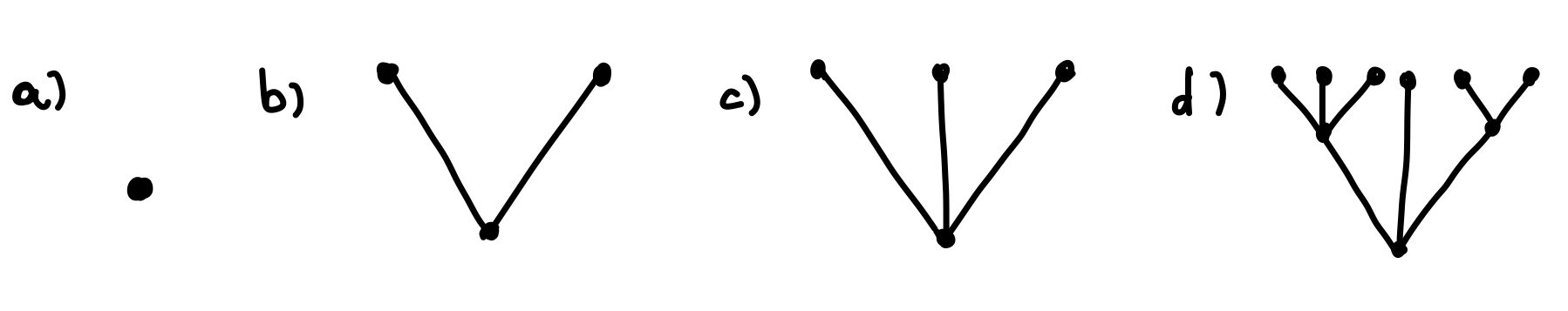}
%\caption{This is an image}
%\label{fig}
\end{figure}
\\
%\begin{figure}[htbp]
%\centerline{\includegraphics[width=.6\textwidth, height=.3\textheight, keepaspectratio]{trreeeez.jpg}}
%\end{figure}\\
The planar rooted trees a), b) and c) have respectively one, two and three leaves and one trivial note (the root).\
d) is the planar rooted tree with six leaves and three nodes (two non trivial).\
The planar rooted trees with $n$-vertices and no (non trivial) nodes are denoted by $\mathfrak{T}_n$.\
So a) is $\mathfrak{T}_1$, b) is $\mathfrak{T}_2$ and c) is $\mathfrak{T}_3$.\ %on the other hand $d$ has two nodes.\ 
On the other hand, we denote by $PT_n$ the set of planar rooted trees with $n$-inputs.\ So, $d\in PT_6$ and cleary $\mathfrak{T}_n\in PT_n$.\ 
More precisely we denote by $PT^l_n$ the set of set of planar rooted trees with $n$-inputs and $(n-l)$-nodes.\ So a) $\in PT^0_1$, b) $\in PT^1_2$, c) $\in PT^2_3$ and d) $\in PT^{3}_6$.\ We have: 
\begin{align*}
PT_n=\displaystyle\bigcup^{n-1}_{l=1} PT^l_n.
\end{align*}

Now we introduce the \emph{grafting operation}:\
We consider $n$ planar rooted trees $\mathfrak{t}_j\in PT_{m_j}$ where $n$ is an integer $n>1$ .\
The \emph{grafting} of $\mathfrak{t}_j$ is the tree $\mathfrak{t}_1\vee ...\vee \mathfrak{t}_n \in PT_{m_1+...+m_n}$ obtained by joining the roots of the $\mathfrak{t}_j$'s to the leaves of the tree $\mathfrak{T}_n$.\ See the following picture:
\begin{figure}[htbp]
\centerline{\includegraphics[width=.6\textwidth, height=.3\textheight, keepaspectratio]{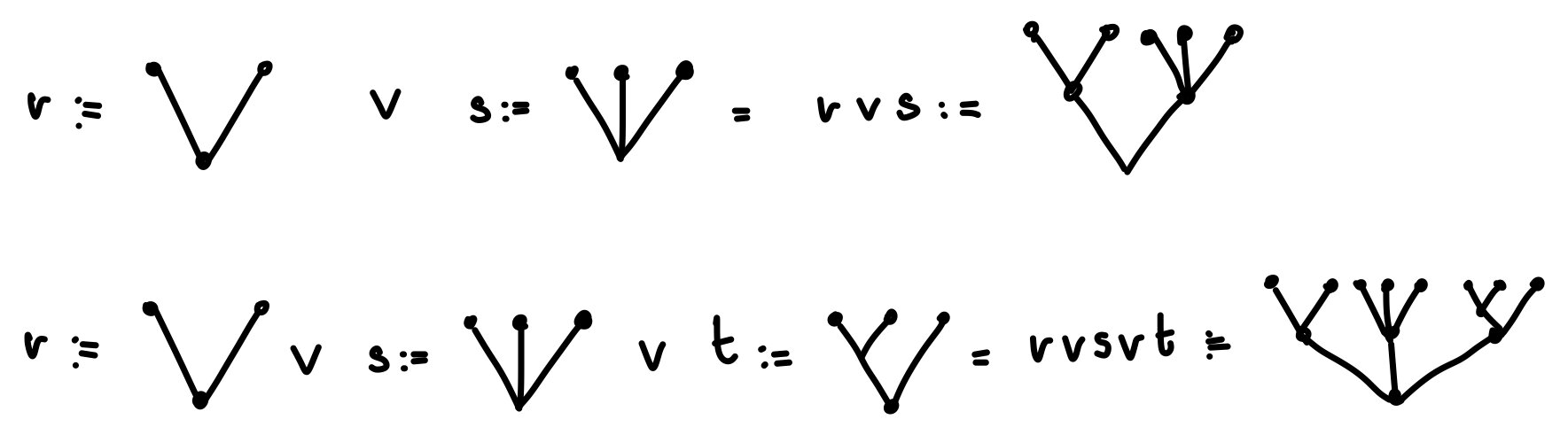}}
%\caption{This is an image}
\label{fig}
\end{figure}
\\
We note that every planar rooted tree is obtained by grafting $\mathfrak{T}_1$.\\
It is clear that $\vee$ is not associative.\ For example $(\mathfrak{t}_1\vee\mathfrak{t}_2)\vee\mathfrak{t}_3\not=\mathfrak{t}_1\vee\mathfrak{t}_2\vee\mathfrak{t}_3\not=\mathfrak{t}_1\vee(\mathfrak{t}_2\vee\mathfrak{t}_3)$.\\

From now on, we refer to the \emph{planar rooted trees} simply as \emph{trees}.

\subsection{$\Ain$categories and DG quivers}
In this subsection we give the definitions of DG quivers, DG categories, non unital $\Ain$categories and their corresponding notions of functors.\ 
In particular it will be useful the notion of underlying DG quiver of an $\Ain$category.\
We suppose that the readers are familiar with (graded $R$-linear) quivers (if not, the Definition can be found in \cite[Definition 1.11]{COS1}).
\begin{defn}[$\Ain$category]\label{acat}
A \emph{non unital $\Ain$category} $\A$ is a graded $R$-linear quiver equipped with $R$-linear morphisms
\begin{align*}
m_{\A}^n:\Hom_{\A}(x_{n-1},x_n)\otimes ...\otimes\Hom_{\A}(x_0,x_1)\to&\Hom_{\A}(x_0,x_n)[2-n],
\end{align*}
for every positive integer $n$, and sequence $x_0,...,x_n$ of $n+1$-objects, such that:
\begin{align}\label{grango}
\displaystyle\sum^n_{m=1}\sum^{n-m}_{d=0}(-1)^{\dagger_d} m_{\A}^{n-m+1}\big(f_n,...,f_{d+m+1},m_{\A}^{m}(f_{d+m},...,f_{d+1}),f_d,...,f_1\big)&=0,
\end{align}
where $\dagger_d=\dego(f_d)+...+\dego(f_1)-d$.
\end{defn}
We can denote by $\big(\A, (m_{\A}^1,m_{\A}^2,...)\big)$ the graded $R$-linear quiver $\A$ and its $\Ain$structure.
\begin{defn}[DG quiver]\label{dgqui}
A \emph{DG quiver} $\Q$ is a graded $R$-linear quiver equipped with a $R$-linear morphism 
\begin{align*}
d_{\Q}:\Hom(x,y)&\to\Hom(x,y)[1]
\end{align*}
for every two objects $x$ and $y$, such that
\begin{align}
d_{\Q}\big(d_{\Q}(f)\big)=0.
\end{align}
Shortly a DG quiver is a non unital $\Ain$category such that $m_{\Q}^1=d_{\Q}$ and $m_{\Q}^{n\ge2}=0$.
\end{defn}
As before we can denote a DG quiver $\Q$ by $\big(\Q, d_{\Q}\big)$.\\
\\
Given a non unital $\Ain$category $\A$ we can associate a DG quiver $|\A|$ as follows:
\begin{align}\label{forgm}
\A\to & |\A|\\
\big(\A, (m_{\A}^1,m_{\A}^2,...)\big) \mapsto& \big(\A,m_{\A}^1\big).
\end{align}
Since $\A$ is a non unital $\Ain$category then $m_{\A}^1(m_{\A}^1(f))=0$, this make $\big(\A,m_{\A}^1\big)$ a DG quiver.\\ 
We call $|\A|$ the \emph{underlying} DG quiver of $\A$.

\begin{defn}[Subquiver]\label{subsquinzi}
Let $\Q$ be a DG quiver.\ A DG subquiver $\Q'$ of $\Q$ is a DG quiver such that 
\begin{itemize}
\item[1.] $\mbox{Ob}(\Q')=\mbox{Ob}(\Q)$.
\item[2.] Given two objects $x,y\in\Q'$ we have an inclusion of graded $R$-modules: 
$$\Hom_{\Q'}(x,y)\subset\Hom_{\Q}(x,y).$$
\item[3.] Given two objects $x,y\in\Q'$ and a morphism $f\in\Hom_{\Q'}(x,y)$ we have:
$$d_{\Q'}(f)=d_{\Q}(f).$$
\end{itemize}
\end{defn}

\begin{defn}\label{sumqui}
Given two DG quivers $\Q_1$, $\Q_2$ with the same objects we can define the DG quiver $\Q_1+\Q_2$ as follows:
\begin{itemize}
\item[1.] The objects of $\Q_1+\Q_2$ are the same objects of $\Q_1$ (or $\Q_2$).
\item[2.] Fixed two objects $x,y\in\Q_1+\Q_2$, the hom spaces is given by:
\begin{align*}
\Hom_{\Q_1+\Q_2}(x,y)&:=\Hom_{\Q_1}(x,y)\oplus\Hom_{\Q_2}(x,y).
\end{align*}
\end{itemize}
\end{defn}

\begin{defn}[DG category]\label{dgcat}
A \emph{DG category} $\C$ is a graded $R$-linear category equipped with a $R$-linear morphism 
\begin{align*}
d_{\C}:\Hom(x,y)&\to\Hom(x,y)[1]
\end{align*}
for every two objects $x$ and $y$, such that
\begin{align}
d_{\C}\big(d_{\C}(f)\big)=0.
\end{align}
Moreover $d_{\C}$ satisfies the Leibniz rule, i.e. 
\begin{align}
d_{\C}(g\cdot f)&=g\cdot d_{\C}(f) + (-1)^{\tiny\dego(g)} d_{\C}(g)\cdot f,
\end{align}
for every $f\in\Hom(x,y)$, $g\in\Hom(y,z)$ and triple of objects $x,y,z\in\mathcal{C}$.
\end{defn}

We can see a DG category $\C$ as a non unital $\Ain$category such that $m_{\C}^1=d_{\C}$, $m_{\C}^2(g,f)=g\cdot f$ and $m_{\C}^{n\ge 3}=0$.

\begin{rem}\label{dgnu}
By definition, every object of a DG category has a unit which is a closed degree zero morphism satisfying Definition \ref{su}.\ 
In subsection \ref{unit} we will define different notions of unit for $\Ain$categories.\
On the other hand we can define a \emph{non unital DG category} to be a non unital $\Ain$category such that $m^{n\ge 3}=0$.
\end{rem}

\begin{rem}
If $\A$ is an $\Ain$category (resp. DG category) with one object it is called $\Ain$algebras (resp. DG algebras).
\end{rem}

\subsection*{$\Ain$functors and forgetful functor}

Let $\A$ and $\B$ be two non unital $\Ain$categories.
\begin{defn}[$\Ain$functors]\label{af}
A \emph{non unital $\Ain$functor} between $\A$ and $\B$ is given by the following data:
\begin{itemize}
\item[1.] A map of sets $\F^0:\mbox{Ob}(\A)\to\mbox{Ob}(\B)$.
\item[2.] A collection of $R$-linear morphisms 
\begin{align*}
\F^n:\Hom_{\A}(x_{n-1},x_n)\otimes...\otimes\Hom_{\A}(x_0,x_1)\to&\Hom_{\A}(x_0,x_n)[1-n],
\end{align*}
for every positive integer $n\ge 1$ and sequence of objects $x_0,...,x_n\in\A$, such that 
\begin{align}
\displaystyle\sum_{r\ge1} &\displaystyle\sum_{s_1+...+s_r=n}m^{r}_{\B}\big(\F^{s_r}(f_n,...,f_{n-s_r+1}),...,\F^{s_1}(f_{s_1},...,f_1)\big)=\\
&=\displaystyle\sum_{m=1}^n\displaystyle\sum_{k=0}^{n-m}(-1)^{\dagger_k} \F^{n-m+1}\big(f_n,...,f_{k+m+1},m^{m}_{\A}(f_{k+m},...,f_{k+1}),f_k,...,f_1\big).
\end{align}
Where $\dagger_k$ is the same of Definition \ref{acat}.
\end{itemize}
\end{defn}

\begin{defn}[Strict functor]\label{astricts}
A \emph{strict non unital $\Ain$functor} $\F$ between $\Ain$categories is a non unital $\Ain$functor such that 
\begin{align*}
\F^{n}(f_n,...,f_1)&=0 
\end{align*}
for every positive integer $n\ge 2$ and every sequence of morphisms $f_n,...,f_1\in\A$.
\end{defn}

\begin{rem}\label{SFE}
If $\F:\A\to\B$ is a strict non unital $\Ain$functor then 
\begin{align*}
m^n_{\B}\big(\F^{1}(f_n),...,\F^{1}(f_1)\big)&=\F^{1}\big( m^n_{\A}(f_n,...,f_1) \big).
\end{align*}
\end{rem}

We recall from \cite[Definition 1.2.4]{Orn1}, \cite[(1.7)]{Sei}.
\begin{defn}[Composition of (non unital) $\Ain$functors]\label{composizioncella}
Let $\F:\A\to\B$ and $\G:\B\to\C$ two (non unital) $\Ain$functors we have a (non unital) $\Ain$functor $\G\cdot\F:\A\to\C$ defined as follows:
\begin{itemize}
\item[1.] $(\G\cdot\F)^0:=(\G)^0\cdot(\F)^0$ (composition of maps of sets).
\item[2.] For every $n>0$, and every sequence of morphisms $f_n,...,f_1\in\A$ we have:
\begin{align}
(\G\cdot\F)^n(f_n,...,f_1):=\displaystyle\sum_{r\ge1} &\displaystyle\sum_{s_1+...+s_r=n}
\G^{r}\big(\F^{s_r}(f_n,...,f_{n-s_r+1}),...,\F^{s_1}(f_{s_1},...,f_1)\big).
\end{align}
\end{itemize} 
\end{defn}
It is not difficult to prove that if $\F$ and $\G$ are strict (resp. cohomological unital, unital, strictly unital, augmented strictly unital in sense of subsection \ref{unit}) then even $\G\cdot\F$ is so.

\begin{defn}[DG functors]\label{dgf}
A \emph{DG functor} is a strict $\Ain$functor between two DG categories preserving the units.\ Namely, denoting by $1_x$ the unit of $x$ we have $\F^1(1_x)=1_{\F^0(x)}$ for every object $x\in\A$. 
\end{defn}

\begin{defn}[Functor between DG quivers]\label{dgqf}
A functor between two DG quivers $\Q_1$ and $\Q_2$ is given by the following data:
\begin{itemize}
\item[1.] A map of sets $\sF^0:\mbox{Ob}(\Q_1)\to\mbox{Ob}(\Q_2)$.
\item[2.] A map of graded $R$-modules $\sF^1:\Hom_{\Q_1}(x,y)\to\Hom_{\Q_2}\big(\sF^0(x),\sF^0(y)\big)$ such that:
\begin{align*}
d_{\Q_1}(\sF^1(f))&=\sF^1(d_{\Q_2}(f)).
\end{align*}
\end{itemize}
\end{defn}
In analogy to the case of categories we can denote an $\Ain$functor by
\begin{align*}
\F:=\big(\F^0,\F^1,...,\F^n,...\big)
\end{align*}
Given an $\Ain$functor $\F$, we can consider $\F^1:=\big(\F^0,\F^1,0,0,...\big)$.\\
We point out that $\F^1$ is not an $\Ain$functor but it is a functor between the underlying DG quivers.\\ 
\\
So we can extend the map $|\mbox{-}|$, defined in (\ref{forgm}), to a forgetful functor:
\begin{align}\label{foffofo}
\aCat \to& \DgQui \\
\A\mapsto&|\A |\\
\big(\F^0,\F^1,...,\F^n,...\big)\mapsto& \F^1.
\end{align}

We conclude this subsection with a definition.

\begin{defn}\label{equiv}
We say that two non unital $\Ain$categories $\A$ and $\B$ are \emph{equivalent} if there exist two non unital $\Ain$functors $\F:\A\to\B$ and $\G:\B\to\A$ such that $\F\cdot\G=\Id_{\B}$ and $\G\cdot\F=\Id_{\A}$.\ We say that $\F$ and $\G$ are \emph{equivalences}.
\end{defn}

In Definition \ref{equiv}, we denote by $\Id$ the (strict) $\Ain$functor given by:
$$\Id_{\A}=(\Id_{\tiny\mbox{Ob}(\A)},\Id_{\tiny\Hom_{\A}},0,...).$$

\subsection{The cohomology category and quasi-equivalences}\label{queequi}
Let $\A$ be a non unital $\Ain$category.\
It is clear that (non unital) $\Ain$categories are not even close to the concept of categories (or enriched categories) in the usual sense.\
But we can associate to $\A$ the \emph{cohomology category} $H(\A)$ which is a graded $R$-linear non unital category.\
The category $H(\A)$ is defined as follows:
\begin{itemize}
\item[1.] $H(\A)$ has the same objects of $\A$. 
\item[2.] The hom-spaces are the cohomology classes 
$$\Hom_{H(\A)}(x,y):=\bigoplus_{n\in\mathbb{Z}}H^n\big(\Hom_{\A}(x,y)\big),$$
\item[3.] The product is 
\begin{align}\label{ciciop}
g\cdot f&:=[m_2^{\A}(g,f)].
\end{align}
for every $f,g\in\A$.
\end{itemize}
A non unital $\Ain$functor $\F:\A\to\B$ is a \emph{quasi-equivalence} if it induces an equivalence of non unital graded categories 
$$[\F]:H(\A)\to H(\B).$$
This is equivalent to say that 
\begin{itemize}
\item[1.] For every $x,y\in\A$, the morphism of chain complexes $$\F^1:\Hom_{\A}(x,y)\to \Hom_{\B}(\F^0(x),\F^0(y))$$
is a quasi-isomorphism of chain complexes.
\item[2.] $H^0(\F)$ is essentially surjective.
\end{itemize}
The same definition holds %to define the notion of \emph{quasi-equivalence} 
in the framework of DG categories.

\begin{rem}
An equivalence in the sense of Definition \ref{equiv} is a quasi-equivalence (clearly the opposite implication is false).\ 
A quasi-equivalence between $\Ain$algebras is called \emph{quasi-isomorphism}.\ In this case condition 2. above is superfluous.\
\end{rem}

\subsection{The notion of unit in $\Ain$categories and functors}\label{unit}
The definition of $\Ain$category does not include a natural notion of unit.\ 
But in most of the examples of $\Ain$categories (especially the ones coming from Homological Mirror Symmetry) we have a morphism which can be considered as a unit (at least in cohomology).\
In a few words, the interesting $\Ain$categories are the ones whose cohomology category is unital.\\
In this section we recall some notions of unit in this framework.\ We are particularly interested in the strictly unital $\Ain$categories.\

\begin{defn}[Strictly unital $\Ain$category]\label{su}
An $\Ain$category $\A$ is \emph{strictly unital} if, given an object $x\in\A$, there is a closed (degree zero) morphism $1_x:x\to x$ such that:
\begin{itemize}
\item[1.] $m^2_{\A}(1_x,f)=m^2_{\A}(f,1_x)=f$, for every $f\in\A$.
\item[2.] $m^n_{\A}(f_n,...,1_x,...,f_1)=0$, for every $f_i\in\A$ and $n>2$. 
\end{itemize}
\end{defn}
\begin{defn}[Nice unit]\label{snu}
We say that a strictly unital $\Ain$category $\A$ has a \emph{nice} unit if, for every $x\in\A$ the unit $R$ is such that $R\cdot 1_x \cong R$ and the following short exact sequence 
\begin{align}
0\to R\cdot 1_x \to \Hom_{\A}(x,x) \to \Hom_{\A}(x,x)/ R\cdot 1_x\to 0
\end{align}
splits in the category of graded $R$-modules.
\end{defn}
Given two strictly unital $\Ain$categories $\A$ and $\B$.
\begin{defn}[Strictly unital $\Ain$functor]\label{suF}
An $\Ain$functor $\F:\A\to\B$ is \emph{strictly unital} if
\begin{itemize}
\item[1.] $\F^1(1_x)=1_{\F^0(x)}$, for every $x\in\A$.
\item[2.] $\F^n(f_n,...,1_x,...,f_1)=0$, for every $f_i\in\A$ and $n>1$. 
\end{itemize}
\end{defn}

\begin{defn}[Discrete category and quiver]\label{disco}
Let $\A$ be a DG category.\\
The \emph{discrete category} of $\A$, denoted by $\mbox{disc}(\A)$, has the same objects of $\A$ and the morphisms are defined as follows
\begin{align}
\Hom_{\tiny\mbox{disc}(\A)}(x,y)=
&\begin{cases}
   0        & \text{if } x \not=y \\
   1_x        & \text{if } x=y.
  \end{cases}
\end{align}
The differential is zero and the composition is the one making $\mbox{disc}(\A)$ a DG category.\\
We denote by $\mathsf{I}_{\A}$ the DG quiver $|\mbox{disc}(\A)|$.\ We call $\mathsf{I}_{\A}$ the \emph{discrete} DG quiver of $\A$.
\end{defn}

Clearly Definition \ref{disco} makes sense even if $\A$ is a DG quiver or an $\Ain$category.

\begin{defn}[Augmented unital $\Ain$category]\label{sAu}
A strictly unital $\Ain$category $\A$ is \emph{augmented} if it has a strict and strictly unital $\Ain$functor 
$\epsilon_{\A}:\A\to\mbox{disc}(\A)$ which is the identity on objects. 
\end{defn}
Given a non unital $\Ain$category $\A$ we can define the augmented strictly unital $\Ain$category $\A_{+}$, which has the same objects of $\A$.\ The augmented $\Ain$categories, together with the strictly unital $\Ain$functors preserving $\epsilon$, form a category denoted by $\aCat^{\tiny\mbox{a}}$.

On the other hand, given an augmented strictly unital $\Ain$category $\A$, we can take its reduction $\overline{\A}$.\ Such a category is a non unital $\Ain$category with the same objects of $\A$ and the hom-spaces $\overline{\A}(x,y)$ defined to be the Kernel (in the category of DG modules) of $\epsilon_{\A}:\A(x,y)\to \mbox{disc}(\A)(x,y)$.\\ 
Augmentation and reduction give rise to two functors 
$$\mbox{(-)}_{+}:\aCat^{\tiny\mbox{nu}}\to \aCat^{\tiny\mbox{a}}$$ 
and 
$$\overline{\mbox{(-)}}:\aCat^{\tiny\mbox{a}}\to \aCat^{\tiny\mbox{nu}},$$
which are quasi-inverse equivalences of categories (cf [COS; pp. 7]).\ 
Here $\aCat^{\tiny\mbox{nu}}$ denotes the category of non unital $\Ain$categories together with non unital $\Ain$functors (see Section \ref{hcat}).

\begin{defn}[Cohomological unital $\Ain$categories and functors]\label{uniff}
An $\Ain$category $\mathscr{A}$ is \emph{cohomological unital} if its cohomology category $H^0(\mathscr{A})$ is a category (i.e. it is unital).\ An $\Ain$functor between two cohomological unital $\Ain$categories is said \emph{cohomological} if it induces a (unital) functor between their cohomology categories.
\end{defn}

Cohomological unital (resp. strictly unital)  $\Ain$categories together with cohomological unital (resp. strictly unital) $\Ain$functors form a category.

\begin{rem}
By condition 2. of Definition \ref{su} we have that a strict unit $1_x$ (if it exists) is unique.\ On the other hand a cohomological unit $1_x$ it is unique up to $m^1$.
\end{rem}

The next definition is due to Lyubashenko \cite[Definition 7.3.]{Lyu}.

\begin{defn}[Unital $\Ain$categories]\label{LMuni}
A cohomological unital {$\Ain$}category $\A$ is \emph{unital} if, for every object $x\in\A$, there exists $1_x$, a representative of the cohomology unit of $x$, such that $m^2_{\A}(1_x,\mbox{-})$ and $m^2_{\A}(\mbox{-},1_x)$ induce two morphisms of chain complexes invertible in homotopy.
\end{defn}
In other words, for every object $y\in\A$ there exists a morphism of graded $R$-complexes: 
{$$\mathcal{H}:\Hom_{\A}(x,y)\to\Hom_{\A}(x,y)[-1]$$} 
such that 
{$$\Id_{\tiny\Hom_{\A}(x,y)}-m^2_{\A}(-,1_x)=m^1_{\A}\cdot\mathcal{H}+\mathcal{H}\cdot m^1_{\A},$$}
and a morphism of graded $R$-complexes {$\mathcal{H'}:\HomA(y,x)\to\HomA(y,x)[-1]$} such that 
{$$\Id_{\tiny\Hom_{\A}(y,x)}-m^2_{\A}(1_x,-)=m^1_{\A}\cdot\mathcal{H}'+\mathcal{H}'\cdot m^1_{\A}.$$}

We have the following inclusions of categories:
\[
\mbox{strictly unital $\Ain$cats}\subset\mbox{unital $\Ain$cats}\subset \mbox{cohomological unital $\Ain$cats}.
\]
Here the category of unital $\Ain$category is a full subcategory of the cohomological unital $\Ain$cats.\ See \cite[pp 23]{Sei} and \cite[Definition 8.1. and Proposition 8.2.]{Lyu}.\ %We conclude this section saying that there is also a notion of \emph{weak unit}, which is due to Kontsevich and Soibelman [KS3, \S 4.2].\ These notions are equivalent (see [LM3]).\ In the upcoming paper [COS2], we fully investigate the relations between the categories of $\Ain$-categories with these different notions of unit.

\begin{rem}
We have other notions of unit in the framework of $\Ain$categories.\ We mention the \emph{weakly unital} $\Ain$categories due to Kontsevich and Soibelman \cite[Definition 4.2.3]{KS3} and homotopy unital $\Ain$categories (\cite[\S (2a)]{Sei}).\ According to \cite[3.6 Proposition-3.7 Theorem]{LM3} all of them are (in some sense) equivalent.
\end{rem}

\subsection{Categories and Homotopy Categories}\label{hcat}
We conclude this section by recalling all the categories involved in this paper.\\

Denoting by $\DgQui$, $\DgCat^{\tiny\mbox{nu}}$, $\aCat^{\tiny\mbox{nu}}$, $\aCat^{\tiny\mbox{nu}}{\tiny\mbox{strict}}$ respectively the categories of DG quivers (Definition \ref{dgqui}), non unital DG Categories (Remark \ref{dgnu}) together with non unital strict $\Ain$functors (Definition \ref{astricts}), non unital $\Ain$categories (Definition \ref{acat}) with non unital $\Ain$functors (Definition \ref{af}) and with non unital strict $\Ain$functors (Definition \ref{astricts}).\ We have the inclusions:
\[
\xymatrix{
{\DgQui}\subset \DgCat^{\tiny\mbox{nu}} \subset \aCat^{\tiny\mbox{nu}}_{\tiny\mbox{strict}} \subset \aCat^{\tiny\mbox{nu}}.
}
\]
On the other hand we can define a \emph{unital DG category} to be a unital $\Ain$category (Definition \ref{LMuni}) such that $m^{n\ge3}=0$.\ Unital DG categories together with cohomological unital strict $\Ain$functors (Definition \ref{uniff} + \ref{astricts}) form a category $\DgCat^{\tiny\mbox{u}}$.\ 
Denoting by $\aCat^{\tiny\mbox{u}}$ and $\aCat^{\tiny\mbox{u}}_{\tiny\mbox{strict}}$ the category of unital $\Ain$categories equipped, respectively, with cohomological $\Ain$functors and with strict cohomological $\Ain$functors, we have the inclusions of categories:
\begin{align}\label{inc1}
\xymatrix{
\DgCat^{\tiny\mbox{u}}\subset\aCat^{\tiny\mbox{u}}_{\tiny\mbox{strict}} \subset \aCat^{\tiny\mbox{u}} \subset \aCat^{\tiny\mbox{cu}}.
}
\end{align}
To conclude we denote by $\DgCat$, $\aCat_{\tiny\mbox{strict}}$ and $\aCat$ the categories of DG categories (Definition \ref{dgcat}), strictly unital $\Ain$categories (Definition \ref{su}) equipped with strict strictly unital $\Ain$functors (Definition \ref{astricts} + Definition \ref{suF}) and strictly unital $\Ain$functors (Definition \ref{suF}).\ Then we have the inclusions:
\begin{align}\label{inc2}
\xymatrix{
\DgCat \subset \aCat_{\tiny\mbox{strict}} \subset \aCat.
}
\end{align}
It is easy to see that none of the inclusions above is an equivalence of (1-)categories.\\
On the other hand we have a notion of quasi-equivalence in both the categories $\aCat$ and $\DgCat$ (see Subsection \ref{queequi}).\\
We denote by $\mbox{Hqe}$ and $\Ho(\aCat)$ respectively the localization of the category of DG categories and the category of the strictly unital $\Ain$categories, 
with respect to the classes of quasi-equivalences.\\
The categories $\mbox{Hqe}$ and $\Ho(\aCat)$ are called the \emph{homotopy categories} of DG categories and $\Ain$categories.\\
With \emph{Homotopy theory} we mean the Gabriel-Zisman localization of 1-categories see \cite[\S1.1.1]{GZ}.\ 
For more example of homotopy theories see \cite[\S2.1]{Toe1}.\\
In \cite[Theorem B]{COS2} we proved that the functors $U$ and $i$ (see subsection \ref{hyhyhhyhy}) induce an equivalences 
\begin{align*}
\mbox{Hqe}\simeq\Ho(\aCat)\simeq\Ho(\aCat^{\tiny\mbox{u}})\simeq\Ho(\aCat^{\tiny\mbox{cu}}).
\end{align*}
Note that these categories are equivalent at the $(\infty,1)$-level, see \cite[Theorem A]{COS2}.\\

\newpage

\section{The $\Ain$category of $\Ain$functors}\label{lonzolenzo}

As we saw in the previous section the definition of $\Ain$category is pretty technical and complicated.\ 
The great advantage of this framework respect to the DG one is in the extreme flexibility of the (pre)natural transformations (cf. Definition \ref{trantran}).\
Let us start with a example.

\begin{defn}[Natural transformations of DG categories]\label{NatDg}
Given $\A$, $\B$ two DG categories and $F,G:\A\to\B$ two DG functors, a \emph{natural transformation} $\eta$ is a morphism $\eta_x:F(x)\to G(x)$ in $\B$, for every $x\in\A$, such that the following diagram
\begin{align*}
\xymatrix{
F(x)\ar[r]^{\eta_x}\ar[d]^{F(f)}& G(x)\ar[d]^{G(f)}\\
F(y)\ar[r]^{\eta_y}&G(y).
}
\end{align*}
is commutative in $\B$ for every $f\in\A$.
\end{defn}
The set of natural transformations between two DG categories forms a DG category, denoted by $\mbox{Nat}_{\tiny\mbox{dg}}(\A,\B)$, whose objects are the DG functors and the morphisms are natural transformations.\ The differential and composition are defined respectively as follows:
\begin{align*}
d_{\tiny\mbox{Nat}_{\tiny\mbox{dg}}(\A,\B)}(\eta):=d_{\B}\cdot \eta - \eta\cdot d_{\A}
\end{align*}
\begin{align*}
(\eta \cdot \psi)_x:=\eta_x \cdot \psi_x.
\end{align*}

\begin{thm}
The category of natural transformations provides an Internal Hom for the symmetric monoidal category ($\DgCat$,$\otimes$,$R$).\ Namely, given three DG categories we have a natural bijection of sets:
\begin{align*}
\DgCat(\A\otimes\B,\C)\simeq\DgCat(\A,\mbox{Nat}_{\tiny\mbox{dg}}(\B,\C)).
\end{align*}
\end{thm}
\begin{proof}
Left as an exercise.
\end{proof}
On the other hand, even the homotopy category of DG categories is a closed symmetric monoidal category.\ This is a remarkable result by To\"en \cite[\S6]{Toe} (see also \cite[Theorem 1.1]{CS}).\\
Clearly in this case the DG category of natural transformations $\mbox{Nat}_{\tiny\mbox{dg}}(\A,\B)$ is not sufficient to describe the Internal Hom of $\mbox{Hqe}$.\ 
But we can describe it via (natural transformations of) $\Ain$functors.\ This was a brilliant intuition by Kontsevich and it was proven in \cite[Theorem B]{COS1} (over a field) and \cite{COS2} (over a commutative ring).\
This is an example of how the language of $\Ain$categories is the right framework for studying the homotopy category.

\subsection{Category of functors between $\Ain$categories}\label{S1}
In this subsection we introduce the notion of 1-transformation and prenatural transformation.\ Morally they are a kind of generalization of Definition \ref{NatDg} in the context of $\Ain$functors (see remark \ref{NATRAN}).

\begin{defn}[1-transformation DG quiver and $\Ain$category]
Let $\Q$ be a quiver, and $\A$ a non unital $\Ain$category and $\sF,\sG:\Q\to|\A|$ are morphisms of DG quivers (Definition \ref{dgqf}).\\
A \emph{1-transformation} is the datum $T:=(T^0,T^1)$.\ 
Where $T^0(x):\sF(x)^0\to\sG^0(x)$, for every $x\in\A$ and 
\begin{align*}
T^1:\Hom_{\Q}(x,y)\to \Hom_{\A}(\sF^0(x),\sG^0(y))
\end{align*}
is a morphism of differential graded complexes, for every $f:x\to y\in \A$.
%$r_1(f):\sF(x)\to \sG(y)$
\end{defn}

Fixed $\Q$ a DG quiver and $\A$ a non unital $\Ain$category, we can define
\begin{align}\label{func}
\mbox{Fun}_1(\Q,\A)
\end{align}
a non unital $\Ain$category such that:
\begin{itemize}
\item[1.] the objects are the functors of DG quivers with source $\Q$ and target $|\A|$,
\item[2.] the morphisms are the 1-transformations $(T^0,T^1)$.
\item[3.] Fixed two functors $\sF,\sG:\Q\to|\A|$ and $T=(T^0,T^1)$ the differential is given
\begin{align*}
\mathfrak{M}^1(T^0,T^1)(x,f):=\left( m_{\A}^1\big(T^0(x)\big), T^1\big( d_{\Q}(f)\big)  +m_{\A}^2\big(T^0(y) , \sF^1(f) \big) + m_{\A}^2\big(\sG^1(f), T^0(x)\big)+ m_{\A}^1(T^1(f)) \right).
\end{align*}
such that $x\in\Q$ and $f:x\to y\in\Q$.
\end{itemize}
The higher compositions are inherited by the ones of $\A$.\\ 
Moreover if $\A$ is a unital (resp. strictly unital) $\Ain$category then even $\mbox{Fun}_1(\Q,\A)$ is so.\\
A 1-transformation which is closed with respect to $\mathfrak{M}_1$ is called \emph{natural} 1-transformation.\\
A detailed description of 1-transformations (and more in general about \emph{n-transformations}) can be found in \cite[\S1 and \S 2.7]{LM1} and \cite[Section 5]{Lyu}.\\
\\
Let $\A$ and $\B$ be two non unital $\Ain$categories and two non unital $\Ain$functors $\F,\G:\A\to\B$.
\begin{defn}[Prenatural transformation between $\Ain$categories]\label{trantran}
A \emph{prenatural transformation} $$T:=(T^0,T^1,...)$$ 
of degree $g$ is given by the datum: 
\begin{itemize}
\item[1.] For every object $x\in\A$ a morphism $T^0(x):\F(x)\to\G(x)$ in $\B$.
\item[2.] for every $n>0$ and sequence of objects $x_0,...,x_n\in\A$.
\begin{align*}
T^n:\Hom(x_{n-1},x_n)\otimes...\otimes\Hom(x_0,x_1)\to&\Hom\big(\F^0( x_0),\G^0( x_n)\big)[g-n].
\end{align*}
\end{itemize}
\end{defn}

\begin{rem}\label{pretra}
To understand what is a prenatural trasformation it is useful the interpretation of (non unital) $\Ain$categories and (non unital) $\Ain$functors via bar construction (see \cite[\S2]{Orn1}).\\
Denoting by $\textbf{B}_{\infty}$ the bar construction and by $\mathsf{Q}$ a graded-quiver (not a DG quiver) we have:
\begin{itemize}
\item[1.] $\Ain$structures on $\mathsf{Q}$ $\leftrightarrow$ DG structures on $\textbf{B}_{\infty}(\mathsf{Q})$.
\item[2.] $\Ain$functors $\A_1\to\A_2$ $\leftrightarrow$ DG cofunctors: $\textbf{B}_{\infty}(\A_1)\to\textbf{B}_{\infty}(\A_2)$. 
\end{itemize}
In this setting the prenatural transformations 
$$T:\F\Rightarrow\G$$ 
correspond to the $(\textbf{B}_{\infty}(\F),\textbf{B}_{\infty}(\G))$-coderivations.\
Let us recall that $\textbf{B}_{\infty}(\mathsf{Q})$ is a non unital DG Cocategory.\\
The associated unital DG cocategory is denoted by $\textbf{B}_{\infty}(\mathsf{Q})_{+}$.\ 
In a few words if $\mathsf{Q}=V$ a graded vector space then $\textbf{B}_{\infty}(\mathsf{Q}):=\overline{\mbox{T}}^c(V[1]):=\displaystyle\bigoplus_{n\ge1} \big(V[1]\big)^{\otimes n}$ and $\textbf{B}_{\infty}(\mathsf{Q})_{+}:=\displaystyle\bigoplus_{n\ge0} \big(V[1]\big)^{\otimes n}$.
\end{rem}

Fixed two non unital $\Ain$categories $\A$ and $\B$, we denote by
\begin{align}\label{func}
\Fun(\A,\B)
\end{align}
the set of non unital functors from $\A$ to $\B$.\ We can make $\Fun(\A,\B)$ a non unital $\Ain$category:
\begin{itemize}
\item[1.] the objects are the non unital $\Ain$functors with source $\A$ and target $\B$,
\item[2.] the morphisms are the prenatural transformations.
\item[3.] Given a prenatural transformation $T:\F\Rightarrow\G:\A\to\B$, the differential is given by the formula:
\begin{align}\label{lososot}
\mathfrak{M}^1(T)_n(f_n,...,f_1)&:= \displaystyle\sum_{1\le i\le r}\sum_{s_1+...+s_r=n}(-1)^{\dagger} m^r_{\B}\big(\F^{s_r}(f_n,...,f_{n-s_r+1}),...,\F^{s_{j+1}}(f_{...},...,f_{...}), \\
&T^{s_j}(f_{...},...,f_{...}),\G^{s_j-1}(f_{...},...,f_{...}),...,\G^{s_1}(f_{s_1},...,f_1) \big)\\
&+ \displaystyle\sum_{r,i}(-1)^{\ddagger}T^{d-m+1}\big(f_d,...,f_{n+m+1},m^{m}_{\A}(f_{n+m},...,f_{n+1})f_n,...,f_1\big).
\end{align}
%Here $\dagger=\big(\mbox{deg}(T)-1\big)\big(\mbox{deg}(f_1)-1 \big)$
The precise formula (with the right signs) can be found in \cite[Definition 1.3.2]{Orn1}.\ 
In the first summation we could have none term of the form $\F^{...}(...)$ or $\G^{...}(...)$ but we may have $T^0(...)$. 
\item[4.] the composition $\mathfrak{M}^2$ is described in \cite[(1.9)]{Sei}, the higher multilinear maps are described in \cite[Section 5]{Lyu}.
\end{itemize}
We define a \emph{natural transformation} to be a prenatural transformation $T$, such that $\mathfrak{M}^1(T)=0$.\\
As in the case of non unital $\Ain$categories and $\Ain$functors, let 
$$T=(T^0,T^1,T^2,...)$$
be a natural transformation between two non unital $\Ain$functors $\F,\G:\A\to \B$, then 
\begin{align}\label{zorgonzolo}
(T^0,T^1)
\end{align}
is a natural 1-transformation between $|\A|$ and $\B$.\\

Suppose that $\A$ is a strictly unital $\Ain$category.\ 
A prenatural transformation $T\in\Fun(\A,\B)$ is \emph{strictly unital} if $T^n(f_n,...f_1)=0$ whenever $n\ge 1$ and there exists $j$ such that $f_j=1_x$ for.\\
We define the subcategory $\Fun^u(\A,\B)$, with the same objects of $\Fun(\A,\B)$, and whose morphisms are the strictly unital prenatural transformations (see \cite[Remark 1.24]{COS1} and \cite[Definition 1.17]{COS2}).

\begin{rem}\label{NATRAN}
The natural transformations are the closed coderivations in the interpretation of Remark \ref{pretra}.\ 
If $T$ is a natural transformation then $[T]$ is a natural transformation $[\F]\to[\G]:[\A]\to[\B]$ between $R$-linear categories.
\end{rem}

\begin{rem}
To provide an $\Ain$structure to the category of functors, we just care about the structure of the target category $\B$.\ It means that if $\B$ is non-unital (resp. cohomologically unital, unital, strictly unital, DG) then $\Fun(\A,\B)$ is non-unital (resp. cohomologically unital, unital, strictly unital, DG).
\end{rem}

\subsection{Vertical and horizontal composition}
We already said that the category $\aCat^{\tiny\mbox{nu}}$ is a 1-category.\ The horizontal composition of two $\Ain$functors is defined explicitly in Definition \ref{composizioncella}.\ 
On the other hand, fixed two non unital $\Ain$categories we have a vertical composition $\mathfrak{M}^2$ according to the previous section.\\% (the explicit definition can be found here [Orn; Definition 1.3.3] or here [Sei; (1.9)]).\\
\\
We fix a non unital $\Ain$functor $\F:\A\to\B$ between two $\Ain$categories.\\
Given a non unital $\Ain$category $\C$, we can consider the non unital $\Ain$category $\mbox{Fun}_{\infty}(\B,\C)$.\\ 
The functor $\F$ induces a non unital strict $\Ain$functor:
\begin{align}\label{PUSSHH}
\Fun(\B,\C)&\to\Fun(\A,\C).\\
\G&\mapsto \F^*(\G),\\
T& \mapsto \F^{*}(T)
\end{align}
Where $\F^*(\G):=\G\cdot\F$ and $\F^{*}(T)$ is defined in
\cite[Proposition-Definition 8.41]{Fuk} or \cite[(1e)]{Sei}.\\
Taking two non unital $\Ain$functors $\G,\G':\B\to\C$ and a prenatural transformation $T:\G\Rightarrow\G'$, we have:
\[\xymatrix@R=1em{
&\ar@{=>}[dd]^{T}&\\
\B\ar@/^2pc/[rr]^{\G}\ar@/_2pc/[rr]_{\G'}&&\C\\
&&
}
\xymatrix@R=1em{
&&\\
&\mapsto&\\
&&
}
\xymatrix@R=1em{
&\ar@{=>}[dd]^{\F^*(T)}&\\
\A\ar@/^2pc/[rr]^{\F^{*}(\G)}\ar@/_2pc/[rr]_{\F^*(\G')}&&\C\\
&&
}
\]

\subsection{Homotopic and weakly equivalent $\Ain$functors}

We fix a non unital $\Ain$category $\A$ and a cohomological unital $\Ain$category $\B$.\	
In this case the cohomological category% It makes sense to consider 
\begin{align}\label{Hfun}
H(\Fun(\A,\B))
\end{align}
is a (graded) 1-category (i.e. it is unital, see \cite[Proposition 7.7]{Lyu}).

We need to define some relations between functors with same source and target.\\
\\
The first Definition can be found in \cite[Definition 1.26. (ii)]{COS1} or \cite[(1h)]{Sei}.\\
We say that two non unital $\Ain$functors $\mathscr{F}$ and $\mathscr{G}$ are \emph{homotopic} if $\mathscr{F}^0=\mathscr{G}^0$ and there exists a prenatural transformation $H$ of degree $0$ such that $H^0(x)=0$ (for every $x$) and $\mathscr{F}-\mathscr{G}=\mathfrak{M}^1(H)$.\\
If $\F$ and $\G$ are homotopic then they induce the same functor in cohomology.\

\begin{defn}[Weakly equivalent \mbox{\cite[Definition 1.26. (i)]{COS1}}]\label{def1}
Two non unital functors {$\F$, $\G\in\Fun(\A,\B)$} are \emph{weakly equivalent} if they are isomorphic (as objects) in {$H(\Fun(\A,\B))$}.\\
%\end{defn}
%\begin{defn}[Homotopy equivalent \mbox{[Fuk, Definition 8.1 + Definition 6.22]}]\label{defi2}
%Two $\Ain$-functor $\mathscr{F},\mathscr{G}:\A\to\B$ are {homotopy equivalent} if they are homotopy equivalent\footnote{see [Fuk, Definition 6.22.]} as object in {$\Fun(\A,\B)$}.\\ 
%Namely $\mathscr{F}$ and $\mathscr{G}$ are {homotopy equivalent} if 
Namely there exist two natural transformations $T:\F\Rightarrow\G$, $S:\G\Rightarrow\F$ and two prenatural trasformations $H:\F\Rightarrow\F$, $H':\G\Rightarrow\G$ such that: 
\begin{align}\label{appross}
\mathfrak{M}^2(S,T)=\Id_{\F}+\mathfrak{M}^1(H)\mbox{ and } \mathfrak{M}^2(T,S)=\Id_{\G}+\mathfrak{M}^1(H')
\end{align}
\end{defn} 
Here $\Id_{\F}$ and $\Id_{\G}$ are the prenatural transformations such that $(\Id_{\F})^{n>0}=(\Id_{\F})^{n>0}=0$ and 
$(\Id_{\F})^0(x)$ and $(\Id_{\G})^0(x)$ are (one of) the units of $\F^0(x)$ and $\G^0(x)$ respectively.\

\begin{rem}\label{defi2}
Definition \ref{def1} correspond to the notion of \emph{homotopy equivalent} non unital $\Ain$functors {(see \cite[Definition 8.1 + Definition 6.22]{Fuk}}).
\end{rem}
We can enrich the category $\aCat$ with a structure of $\omega$-globular set (cf. \cite[6.4 Definition]{Lyu}).\ In this case $H$ and $H'$ play the role of 3-morphisms between the 2-morphisms $\mathfrak{M}^2(S,T)$, {$\Id_{\F}$} and $\mathfrak{M}^2(T,S)$, $\Id_{\G}$ respectively.\\
In a few words, $\F$ and $\G$ are weakly equivalent if there exist two natural transformations $S:\F\Rightarrow\G$ and $T:\G\Rightarrow\F$ such that $\mathfrak{M}^2(S,T)$ is homotopic to $\Id_{\F}$ and $\mathfrak{M}^2(T,S)$ is homotopic to $\Id_{\G}$.\\
\\
Following \cite{COS1} we use the following notation:
\begin{itemize}
\item[1.] $\mathscr{F}\sim\mathscr{G}$ if they are homotopic, 
\item[2.] $\mathscr{F}\approx\mathscr{G}$ if they are weakly equivalent (or homotopy equivalent). 
\end{itemize}
It is clear that $\sim$ and $\approx$ are equivalence relations.\\
Moreover the relation $\approx$ is compatible with the compositions, so we can form the category $\aCat/\approx$, see \cite[pp.13]{COS1} and \cite[pp.12]{COS2}.

\begin{defn}\label{isoph}
We say that two cohomological unital $\Ain$categories $\A$ and $\B$ are \emph{weakly equivalent} if there exist two $\Ain$functors $\F:\A\to\B$ and $\B\to\A$ such that $\G\cdot\F\approx\Id_{\A}$ and $\F\cdot\G\approx\Id_{\B}$. 
\end{defn}

\begin{rem}
Let $\mathsf{C}$ be a non unital cocomplete DG cocategory (see \cite[Definition 1.14]{COS1}) and $\mathsf{Q}$ be a graded quiver.\
Given a morphism $\mathsf{F}:\mathsf{C}\to \mathsf{Q}$ of graded quivers (considering the underline graded quiver of $\mathsf{C}$).\\
We can extend uniquely\footnote{This is via the adjunction $\mbox{coCat}^{\tiny\mbox{n}}\leftrightarrows\mbox{GrQu}$ (see \cite[Proposition 1.17]{COS1}.} $\mathsf{F}$ to an functor of cocategories $\textbf{B}_{\infty}(\mathsf{F}):\mathsf{C}\to \textbf{B}_{\infty}(\mathsf{Q})$.\\
Let $\A_1$, $\A_2$ be two non unital $\Ain$categories and $\B$ a (at least) cohomological unital $\Ain$category, an $\Ain$bifunctor $\mathsf{F}$ from $\A_1$ and $\A_2$ to $\B$ is a morphism of graded quivers 
$$\mathsf{F}:\overline{\textbf{B}_{\infty}(\A_1)_{+}\otimes\textbf{B}_{\infty}(\A_1)_{+}}\to\B$$ 
such that 
$$\textbf{B}_{\infty}(F):\overline{\textbf{B}_{\infty}(\A_1)_{+}\otimes\textbf{B}_{\infty}(\A_1)_{+}}\to \textbf{B}_{\infty}(\B)$$
is a DG cofunctor (see \cite[1.5.3]{Orn1}).\\
As in the case of $\Ain$functors, the $\Ain$bifunctors form a non unital $\Ain$category \cite{BLM}.\ We denote by $\Fun(\A_1,\A_2,\B)$ the $\Ain$category of $\Ain$bifunctors from $\A_1$, $\A_2$ to $\B$.\ By definition the non unital $\Ain$category $\Fun(\A_1,\A_2,\B)$ is isomorphic to $\Fun(\A_2,\A_1,\B)$.\ Moreover we have an isomorphism of non unital $\Ain$categories
\begin{align}\label{evalu}
\Fun(\A_1,\A_2,\B)\cong\Fun(\A_1,\Fun(\A_2,\B)).
\end{align}
The isomorphism (\ref{evalu}) on the objects is simply the evaluation.\\
Given a non unital $\Ain$bifunctor $T$
\begin{align*}
T:\overline{\textbf{B}_{\infty}(\A_1)_+\otimes \textbf{B}_{\infty}(\A_2)_+}\to \B 
\end{align*}
Its evaluation $\mbox{ev}(T):\A_1\to\Fun(\A_2,\B)$ is given by
\begin{align*}
\big(\mbox{ev}(T)^n(f_n,...,f_1)\big)^m(g_m,...,g_1):=T\big((f_n[1]\otimes ...\otimes f_1[1])\otimes(g_m[1]\otimes...\otimes g_1[1])\big)
\end{align*}
for any sequences of morphisms $f_n,...,f_1\in\A_1$ and $g_m,...,g_1\in\A_2$.\\
In the same vein we can define the $\Ain$multifunctors.\ Which are $\Ain$functors with source n non unital $\Ain$categories $\A_1$,...,$\A_n$.\ 
Denoting by $\Fun(\A_1,...,\A_n,\B)$ the $\Ain$category of $n$-$\Ain$functors inherits the structure of $\B$.\ If $\B$ is unital, strictly unital, dg then $\Fun(\A_1,...,\A_n,\B)$ is so.\
This provides a structure of \emph{closed multicategory} on the category of $\Ain$categories to overcome the lack of a tensor product of $\Ain$categories.\
Note that the relation $\approx$ makes sense for $\Fun(\A_1,...,\A_n,\B)$, so we can take a quotient (multi)category $\aCat/\approx$, see \cite[Definition 1.26]{COS1}. 
\end{rem}

\subsection{Homological Perturbation Theory}\label{HPT}

The %(probably) 
most important property of the $\Ain$categories is the following:
\begin{thm}[Whitehead Theorem for $\Ain$categories]\label{White}
Let $R$ be a field.\ Every quasi-equivalence between two ($R$-linear) $\Ain$categories has an inverse "up to homotopy".
\end{thm}
Here, by "up to homotopy", we mean that if $\F:\A\to\B$ is a quasi-equivalence, then there exists an $\Ain$functor $\G:\B\to\A$ such that: $\F\cdot\G\approx\Id$ and $\G\cdot\F\approx\Id$ (see Definition \ref{def1}).\ We intentionally stated this Theorem \ref{White} imprecisely.\
One can ask the "right" hypothesis under which it works (e.g. if $\A$, $\B$ or $\F$ has to be unital, strictly unital, etc.).\
Since it was proven many times e.g. \cite[Theorem 4.2.45]{FOOO1}, \cite[Corollaire 1.3.1.3 (b)]{LH}, \cite{Kad}, \cite{KS}, \cite[Corollary 1.14]{Sei}, with different hypotheses (for DG algebras, homotopy unital $\Ain$categories, over a field of characteristic zero, etc..), we suggest to have a look directly to the references mentionated.\ We point out that all of them works only over a field (or assuming that all the cohomologies involved are h-projective) since they use the \emph{minimal model} of an $\Ain$category.\ 

The name \emph{Whitehead Theorem} was given by Fukaya, Oh, Ono and Otha since it is an algebraic counterpart of Whitehead's theorem for CW-complexes or $\infty$-groupoids.\\

Theorem \ref{White} can be considered an application of a more general Homological Perturbation Theory (shortly HPT) which works over a commutative ring (\cite[\S 1.4.2]{LH}, \cite[(1i)]{Sei}).\ HPT is a method to produce an $\Ain$functor (with inverse up to homotopy) starting from a various number of hypotheses.\ We recall this Theorem which provides a good replacement of Theorem \ref{White} if $R$ is a commutative ring. 

\begin{thm}[\mbox{\cite[Theorem 8.8]{Lyu} + \cite{COS2}}]\label{DC}
Let $\F$ be a (a priori) non unital $\Ain$functor $\F:\A\to\B$ where $\A$ is (a priori) a non unital $\Ain$category and $\B$ a unital $\Ain$category.\
If the followings are satisfied:
\begin{itemize}
\item[1.] there is a mapping of sets {$h:\mbox{Ob}(\B)\to\mbox{Ob}(\A)$}.
\item[2.] For every two objects $a$ and $a'$ in $\A$, the map of chain complexes:
{$$\mathscr{F}^1:\HomA(a,a')\to\HomB(\mathscr{F}^0(a),\mathscr{F}^0(a'))$$}
has a homotopy inverse $\mathscr{G}^1$: 
{$$\mathscr{G}^1:\HomB(\mathscr{F}^0(a),\mathscr{F}^0(a'))\to\HomA(a,a').$$}
\end{itemize}
Then $\A$ and $\F$ are unital and it exists a unital $\Ain$functor $\mathscr{G}:\B\to\A$ such that $\mathscr{F}$ and $\mathscr{G}$ are quasi-inverse to each other.\ Namely {$\mathscr{G}\cdot\mathscr{F}\approx\Id_{\mathscr{A}}$} and {$\mathscr{F}\cdot\mathscr{G}\approx\Id_{\mathscr{B}}$}.\\
Moreover, if $\A$, $\B$ and $\F$ are strictly unital, in particular $\B$ with nice unit\footnote{Note that \cite[8.8 Theorem]{Lyu} is stated differently.\ V. Lyubashenko informed us by email that a strictly unital version of this Theorem it is not immediate.\ This is why we require the extra hypotheses on $\B$ of being nice unital \cite[Lemma 3.1]{COS2}.} (see Definition \ref{snu}), then $\G$ is a strictly unital $\Ain$functor.
\end{thm}

It follows immediately

\begin{cor}
If $\F:\A\to\B$ (in $\aCat$) is a quasi-equivalence, $\A$ and $\B$ are h-projective $\Ain$categories (see Section \ref{hprojec}), in particular $\B$ with nice unit, then $\F$ has an inverse up to $\approx$. 
\end{cor}

\begin{proof}
Since every quasi-isomorphism between h-projective DG modules has an inverse in homotopy it follows directly from Theorem \ref{DC}.
\end{proof}

Theorem \ref{DC} basically says: if $\F$ is an $\Ain$functor, we can find a (weak) inverse starting if $\F_1$ has an homotopy inverse of chain complexes.\
Actually we have something similar for natural transformation:

\begin{prp}\label{Invertitransf}
Let $T$ be a natural transformation in $\Fun(\A,\B)(\F,\G)$.\\ 
We recall that, for any $x\in\A$ we have $T^0(x):\F^0(x)\to\G^0(x)\in\B$.\\ 
If, for every $x\in\A$ and $b\in\B$, the morphisms of complexes
\begin{align}
m^2_{\B}(T^0(x),\mbox{-}):\Hom_{\B}\big(b,\F^0(x)\big)\to\Hom_{\B}\big(b,\G^0(x)\big)
\end{align}
and
\begin{align}
m^2_{\B}(\mbox{-},T^0(x)):\Hom_{\B}\big(\F^0(x),b\big)\to\Hom_{\B}\big(\G^0(x),b\big)
\end{align}
are homotopy invertible then the natural transformation $T$ is invertible as natural transformation.\\ 
It means that it exists a natural transformation $S:\G\Rightarrow\F$ such that $\F\approx\G$.
\end{prp}
\begin{proof}
See \cite[\S 7.13]{Lyu}.
\end{proof}

We conclude this section with a Lemma which is a (kind of) nice application of Theorem \ref{DC}.\
We extend \cite[Lemma 2.5]{Sei} over a commutative ring $R$ characterizing the relation between homotopic and weakly equivalent functors.\ This will be useful in the last section.

\begin{lem}\label{sim}
Let $\mathscr{F}$ and $\mathscr{G}$ be two cohomological unital functors between strictly unital $\Ain$categories, if $\mathscr{F}\sim\mathscr{G}$ then $\mathscr{F}\approx\mathscr{G}$ and $[\F]=[\G]$ in $\mbox{Ho}(\aCat)$.
\end{lem}
\begin{proof}
We define the DG algebra $I:=R\cdot u_0 \oplus R\cdot u_1\oplus R\cdot h$ where $\mbox{deg}(u_0)=\mbox{deg}(u_1)=0$ and $\mbox{deg}(h)=1$.\ The differential is given by $d(u_0)=h=-d(u_1)$ and the multiplication is given as follows:
$$\mbox{$u_0\cdot u_0=u_0$, \qquad $u_1\cdot u_1=u_1$, \qquad $u_0\cdot u_1=u_1\cdot u_0=0$, \qquad $u_1\cdot h=h\cdot u_0=h$, \qquad $h\cdot u_1=u_0\cdot h=0$.}$$ 
We note that $u_0+u_1$ is the unit of $I$.\ We have three DG functors:
\begin{align*}
i:R&\to I\\
1&\mapsto u_0+u_1.
\end{align*}
And for $j=0,1$ the projection:
\begin{align*}
f_j:I&\to R\\
r_0\cdot u_0 +r_1\cdot u_1 +r_2\cdot h&\mapsto r_j.
\end{align*}
Such that $i$ and $f_j$ are homotopy inverses.\\
We recall that, if $\F\sim \G$, there exists a prenatural transformation $H$ such that $\F-\G=\mathfrak{M}_1(H)$.\\
Following \cite[Remark 1.11]{Sei} or \cite[Remark 1.4]{Orn1} via $H$, we can define an $\Ain$functor $\mathscr{H}$ fitting the commutative diagram in $\aCat$:
\begin{align}\label{inver}
\xymatrix{
&&\B&&\\
\A\ar@/^/[urr]^{\F}\ar@/_/[drr]_{\G}\ar[rr]^{\mathscr{H}}&&I\ar[u]^{f_0\otimes\tiny\Id_{\B}}\ar[d]_{f_1\otimes\tiny\Id_{\B}}\otimes\B&&\ar[ll]_{i\otimes\tiny\mbox{id}_{\B}}\B\ar@{=}[ull]\ar@{=}[dll]\\
&&\B&&
}
\end{align}
By (\ref{inver}) we have that $[\F]=[\G]$ in $\mbox{Ho}(\aCat)$.\ Moreover since $f_j\otimes\Id_{\B}$ and $i\otimes \Id_{\B}$ are homotopy inverses we can apply Theorem \ref{DC} getting $\F\approx\G$.
\end{proof}

\subsection{The simplicial and globular nature of the $\Ain$categories}\label{highercats}
In the previous section we proved that, if $\F\sim\G$ then $\F\approx\G$.\ 
To understand the homotopy category of $\Ain$categories it will be crucial Theorem \ref{mspaces} which says that, if $\F\approx\G$, then $[\F]=[\G]$ in $\mbox{Ho}(\aCat)$.\ This is a very straightforward consequence of a much more general result (Lemma \ref{LEMMOMBA}).\
In order to prove it we use the fact that $\Ain$categories are "enriched" in $\Ain$categories.\ 
In this subsection we consider strictly unital $\Ain$categories, strictly unital  $\Ain$functors and strictly unital prenatural transformations.\\

Let $\A$ be a strictly unital $\Ain$category.\\
In this subsection we consider only strictly unital $\Ain$categories and strictly unital $\Ain$functors.\\
\\
We can associate to $\A$ an $\infty$-category $\Na(\A)$ (see \cite{Orn2}).\\
Explicitly we have
\begin{align*}
\Na(\A)_n:=\Hom_{\tiny\Ain}([n],\A).
\end{align*}
$\Hom_{\tiny\Ain}([n],\A)$ denotes the set of strictly unital $\Ain$functor, which is contained in $\mbox{Ob}\big(\mbox{Fun}_{\infty}([n],\A)\big)$.\\ 
Here $[n]$ denotes the object of the simplex category seen as an $\Ain$category.
\\
\\
We recall by \cite{Orn2} that an object in $\Na(\A)_1$ is given by:
\begin{itemize}
\item[1.] two objects of $\A$, $x_0:=\F^0(0)$ and $x_1:=\F^0(1)$,
\item[2.] one closed morphisms of degree zero $f_{01}:=\F^1(j_{01})$.
\end{itemize}
We recall that, if $\A$ is a DG category we have $\Na(\A)_1=\Nd(\A)_1$, where $\Nd$ is the DG nerve defined in \cite[\S1.3.1]{Lu}.\\
Now we consider $\Hom_{\tiny\Ain}([n],\A)$ with its natural $\Ain$structure given by $\mbox{Fun}_{\infty}([1],\A)$.\\

We recall that $[0]$ is the $\Ain$category with one object and whose morphisms is the strict identity:
\begin{align*}
\Hom_{\tiny\Ain}([0],\A)\simeq\A.
\end{align*}
We have two $\Ain$functors 
\begin{align*}
i_0,i_1:[0]&\to[1]
\end{align*}
such that $i_0(0)=0$ and $i_1(0)=1$.\\
We can define two $\Ain$functors 
\begin{align*}
s:=\Hom_{\tiny\Ain}(i_0,\A):\Na(\A)_1\to\A\\
t:=\Hom_{\tiny\Ain}(i_1,\A):\Na(\A)_1\to\A
\end{align*}
On the other hand we can take the DG category $\overline{[1]}$ which has two objects and two morphisms closed of degree zero:
\[
\xymatrix{
0\ar@/^/[r]^{j_{01}}&1\ar@/^/[l]^{j_{10}}
}
\]
such that $j_{01}\cdot j_{10}=\mbox{Id}$ and $j_{10}\cdot j_{01}=\mbox{Id}$.\\
We have
\begin{align*}
[1]\subset \overline{[1]}.
\end{align*}
We have 
\begin{align*}
\Hom_{\tiny\Ain}(\overline{[1]},\A) \subset \Na(\A)_1.
\end{align*}
From now on let me denote by ${\Na^{\star}(\A)}_1$ the $\Ain$category $\Hom_{\tiny\Ain}(\overline{[1]},\A)$.\\
Since $\A$ is an $\Ain$category then ${\Na^{\star}(\A)}_1$ has an $\Ain$structure.\\ 
The objects in ${\Na^{\star}(\A)}_1$ are uniquely determined by 
\begin{itemize}
\item[1.] two objects of $\A$, $x_0:=\F^0(0)$ and $x_1:=\F^0(1)$,
\item[2.] two closed morphisms $f_{01}:=\F^1(j_{01})$ and $f_{10}:=\F^1(j_{10})$ of degree zero,
\item[3.] two morphisms of degree one $f_{010}:=\F^2(j_{01},j_{10})$ and $f_{101}:=\F^2(j_{10},j_{01})$.\\ 
We have
\begin{align*}
m^1_{\A}(f_{010}):=m^2_{\A}(f_{10}, f_{01})+\mbox{Id}_{x_0},
\end{align*}
and 
\begin{align*}
m^1_{\A}(f_{101}):=m^2_{\A}(f_{01}, f_{10})+\mbox{Id}_{x_1}.
\end{align*}
\end{itemize}
It is easy to see that $[0]$ and $\overline{[1]}$ are equivalent as $R$-linear category, we have: 
\begin{align*}
i_0,i_1:[0]&\to\overline{[1]}
\end{align*}
such that $i_0(0)=0$ and $i_1(0)=1$, and 
\begin{align*}
j:\overline{[1]}&\to[0].
\end{align*}
We can define the following three quasi-equivalences of $\Ain$categories:
\begin{align*}
s:=\Hom_{\tiny\Ain}(i_0,\A):{\Na^{\star}(\A)}_1\to\A,\\
t:=\Hom_{\tiny\Ain}(i_1,\A):{\Na^{\star}(\A)}_1\to\A,\\
i:=\Hom_{\tiny\Ain}(j,\A):\A\to {\Na^{\star}(\A)}_1,
\end{align*}
such that $t\cdot i=s\cdot i=\mbox{Id}$.

\begin{lem}\label{LEMMOMBA}
Let $\A$, $\B$ be two strictly unital $\Ain$categories.\\
For every positive integer $n$, we have an isomorphism of $\Ain$categories:
\begin{align*}
\Na\big(\Fun(\A,\B)\big)_n&\to \Fun\big(\A,\Na(\B)_n \big).
\end{align*}
Moreover 
\begin{align*}
\Na^{\star}\big(\Fun(\A,\B)\big)_1&\to \Fun\big(\A,\Na^{\star}(\B)_1 \big).
\end{align*}
\end{lem}

\begin{proof}
Here we are considering the $\Ain$nerve with his $\Ain$enrichement, so we can write $\Hom_{\tiny\aCat}=\Fun$.\ We have:
\begin{align*}
\Na\big(\Fun(\A,\B)\big)_n:=\Fun([n],\Fun(\A,\B))&\simeq \Fun([n],\A,\B)\\
&\simeq\Fun(\A,[n],\B)\\
&\simeq \Fun\big(\A,\Fun([n],\B)\big)\\
&=:\Fun\big(\A,\Na(\B)_n \big).
%\Na(\Fun(\F,\G)_1)&\to\mathcal{f}\mbox{$\psi:\A\to\Na(\B)_1$ such that $s\psi=\F$ and $t\psi=\G$}\mathcal{g}
\end{align*}
The same proof holds for $\Na^{\star}$ taking $\overline{[1]}$ instead of $[n]$.
\end{proof}

A direct consequence of Lemma \ref{LEMMOMBA} is the following:
\begin{cor}\label{Corollaroni}
Denoting by $\mbox{Nat}$ the sets of closed degree zero prenatural transformations.\ In other words $\mbox{Nat}=\mbox{Ob}(H^0(\Fun(\F,\G)))$.\ We have a bijection of sets:
\begin{align*}
\mbox{Nat}(\F,\G)&\leftrightarrow\mathcal{f}\mbox{$\psi:\A\to\Na(\B)_1$ such that $s\cdot\psi=\F$ and $t\cdot\psi=\G$}\mathcal{g}%\\
%H&\mapsto \psi_{H}.
\end{align*}
If $T$ is invertible in $H\big(\Fun(\A,\B)\big)$ then $\psi(T):\A\to\Na^{\star}(\B)$.
\end{cor}

\begin{thm}\label{mspaces}
Given two $\Ain$categories $\A$ and $\B$ if $\F\approx\G$ then $[\F]=[\G]$ in $\mbox{Ho}(\aCat)$.
\end{thm}

\begin{proof}
If $\F\approx\G$ then there exist two natural transformations $H:\F\to\G$ and $H':\G\to\F$ satisfying (\ref{appross}).\\
In particular it means that $m^2(H'^0(a),H^0(a))=\mbox{Id}_a+m^1(W^0(a))$ and $m^2(H^0(a),H'^0(a))=\mbox{Id}_a+m^1(W'^0(a))$.\\
By Corollary \ref{Corollaroni} $\psi({H}):\A\to\Na(\B)_1$ is contained in ${\Na^{\star}(\B)}_1$.\ We have the commutative diagram:
\[
\xymatrix{
&\B&\\
\A\ar[ur]^{\F}\ar[dr]_{\G}\ar[r]^-{\psi_H}&{\Na^{\star}(\B)}_1\ar[u]_s^{\simeq}\ar[d]^t_{\simeq}&\ar@{=}[ul]\ar@{=}[dl]\B\ar[l]_-{i}^-{\simeq}\\
&\B&
}
\]
So $\F$ and $\G$ are the inverses in $\mbox{Ho}(\aCat)$ of the roof
\[
\xymatrix{
\A\ar[r]^-{\psi_H}&{\Na^{\star}(\B)}_1&\B\ar[l]_-{i}^-{\simeq}\\
}
\]
It means that $[\F]=[\G]$ and we are done.
\end{proof}

\begin{rem}\label{pathobj}
If $\A$ is a DG category then ${\Na^{\star}(\A)}_1$ coincides with the DG category $P(\A)$ of \cite[\S3]{Tab2}, \cite[\S 2.2]{CS}, see also \cite[Lemma 4.8]{Gen}.\
We recall that $P(\A)$ is the path object in the Tabuada model structure.\ 
It has a very simply characterization in term of $\Ain$functors.\
This category is different from $\Hom_{\tiny\DgCat}(\overline{[1]},\A)$.\ 
The second one is the category of closed degree zero morphisms with inverse (not inverse in homology).
\end{rem}

\newpage

\section{Free categories, $\Ain$ideals and quotients}\label{Liberone}

This section is divided in three subsections: in the first one we define the non unital, strictly unital $\Ain$categories, and DG categories, generated by a DG quiver (the free categories).\
In the second we characterize the $\Ain$functors and natural transformations with source a free $\Ain$category.\ In the last one we characterize the $\Ain$functors and natural transformations with source a quotient of a free $\Ain$category.\ 

\subsection{Free $\Ain$categories, A$_{\infty}$ideals and quotients}
We start this subsection with the construction of free $\Ain$category due to Kontsevich and Soibelman.\ Then we introduce the notions of $\Ain$ideal and system of relations.\ Once we have the notion of $\Ain$ideal it is easy to define the quotient of an $\Ain$category.\\

Let $\Q$ be a DG quiver.
\begin{cons}\label{Liberanzo}
We can define a non unital $\Ain$category $\TT(\Q)$ as follows:
\begin{itemize}
\item[1.] the objects of $\TT(\Q)$ are the same of $\Q$.
\item[2.] Fixed two objects $x,y\in \TT(\Q)$, the hom-space is the graded $R$-module given by
\begin{align*}
\Hom_{\TT(\Q)}(x,y):=\oplus_{n\ge1}\displaystyle\bigoplus_{\substack{\mathfrak{t}\in PT_n \\ x=x_1,...,x_{n+1}=y}}\Hom_{\Q}(x_1,x_{2}) \otimes ...\otimes \Hom_{\Q}(x_{n-1},x_n).
\end{align*}
See Remark \ref{alberoni}.
\item[3.] The $\Ain$structure on $\TT(\Q)$ is defined as follows: 
\begin{itemize}
\item For $n=2$, we have
\begin{align*}
\mTT^2\big((\T_1; f_1,...,f_n),(\T_2; g_1,...,g_m)\big):=(\T_1\vee \T_2; f_1,...,f_n,g_1,...,g_m).
\end{align*}
Where $\T_1\vee \T_2$ is the planar rooted tree with $n+m$-leaves obtained by grafting $\T_1$ and $\T_2$, see subsection \ref{gloglossary}.
\item In general, for $n>2$ we have:
\begin{align*}
\mTT^n\big((\T_1;f_{1_1},...,f_{m_1}\big),...,\big(\T_n;f_{n_1},...,f_{n_m})\big):=(\T_1\vee...\vee \T_n; f_{1_1},...,f_{n_m}).
\end{align*}
\end{itemize}
We define $\mTT^1$ to be the only one making $\TT(\Q)$ an $\Ain$category (see Example \ref{diffex}).
\end{itemize}
It is clear that the hom-spaces of $\TT(\Q)$ has a natural grading induced by the hom-spaces of $\Q$.
\end{cons}

\begin{rem}\label{alberoni}
Since the hom-space of $\TT(\Q)$ is the direct sum indexed by the trees, we can write a homogeneous morphism as follows:
\begin{align}\label{enxo}
(\mathfrak{t}; f_1,...,f_n).
\end{align}
Where $\mathfrak{t}\in PT_n$ (a tree with $n$-leaves) and 
\begin{align*}
f_1,...,f_n\in \Hom_{\Q}(x_1,x_{2}) \otimes ...\otimes \Hom_{\Q}(x_n,x_{n+1})
\end{align*}
with $x_1=x$ and $x_{n+1}=y$.\ 
The writing (\ref{enxo}) means that the leaves of $\mathfrak{t}$ are labeled by the $f_i$'s.\ 
See the following picture:
\begin{figure}[htbp]
\centerline{\includegraphics[width=.5\textwidth, height=.1\textheight, keepaspectratio]{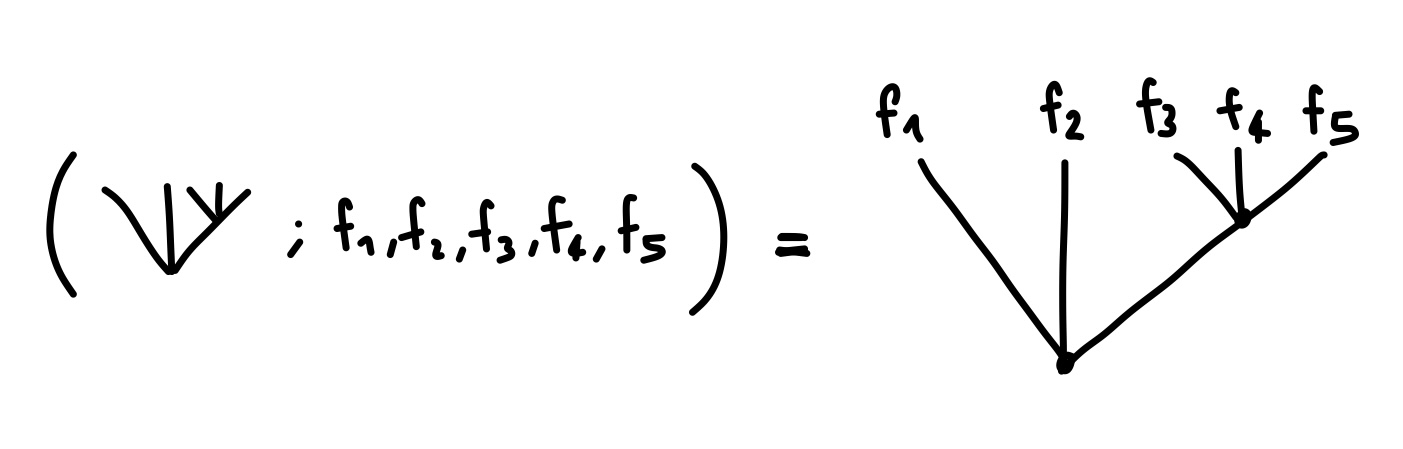}}
%\caption{This is an image from a text that uses color to teach music.}
%\label{fig}
\end{figure}
\\
Since the morphisms $f_1,...,f_n$ are taking in the tensor product, it means that the followings hold:  
\begin{align*}
r\cdot (\mathfrak{t};f_1,...,f_j,...,f_n)=(\mathfrak{t}; rf_1,...,f_j,...,f_n)=...=(\mathfrak{t}; f_1,...,rf_j,...,f_n)=...=(\mathfrak{t};f_1,...,f_j,...,rf_n)
\end{align*}
for every tree $\mathfrak{t}$, $f_1,...,f_n\in\Q$ and $r\in R$.\ And
\begin{align*}
(\mathfrak{t};f_1,...,f_j,...,f_n)+(\mathfrak{t};f_1,...,f'_j,...,f_n)=(\mathfrak{t}; f_1,...,(f_j+f'_j),...,f_n)
\end{align*}
for every tree $\mathfrak{t}$, $f_1,...,f_{j-1},f_{j+1},f_n \in\Q$ and $f_j,f'_j\in\Hom_{\Q}(x_j,x_{j+1})$.
\end{rem}

\begin{exmp}\label{genera}
Note that every tree is generated (iterating the $\mTT$) by the one of the form:
\begin{align*}
\big(\hspace{0,1cm}\bullet \hspace{0,1cm};& \hspace{0,1cm}f\hspace{0,1cm}\big)\hspace{0,3cm} \mbox{where $f\in\Hom_{\Q}(x,y)$}.
\end{align*}
For example 
\begin{align}
t_2=\big(\ \MDu\ ;\ f_1,f_2\ \big)=\mTT^2\big((\bullet ; f_1),(\bullet ; f_2)\big)
\end{align}
The homogeneous morphism 
\begin{align}
t_3=\big(\ \MCS\ ;\ f_1,f_2,f_3\ \big)=\mTT^2(t_2,t_1)
\end{align}
where $t_1=(\bullet ; f_3)$.
\end{exmp}

\begin{exmp}\label{diffex}
In Construction \ref{Liberanzo}, the differential $\mTT^1$ is implicit.\ 
Let us calculate explicitly a few of differentials unrolling equation (\ref{grango}):
\begin{itemize}
\item[$PT_1$)] Given $f\in\Hom_{\Q}(x,y)$:
\begin{align*}
&\mTT^1\big( (\bullet; f )\big)=\big( \bullet; d_{\Q}(f) \big)
\end{align*}
\item[$PT_2$)] Given $f_1,f_2\in\Hom_{\Q}(x_1,x_2)\otimes \Hom_{\Q}(x_2,x_3)$:
\begin{align*}
&\mTT^1\left(\big(\ \MDu\ ;\ f_1,f_2\ \big) \right)=\left(\big(\ \MDu \ ;\ d_{\Q}(f_1),f_2\  \big)\right) +(-1)^{\tiny\mbox{deg($f_2$)}} \left( \big(\ \MDu\ ;\ f_1,d_{\Q}(f_2)\  \big)\right),
\end{align*}
\item[$PT_3$)] Given $f_1,f_2,f_3\in\Hom_{\Q}(x_1,x_2)\otimes \Hom_{\Q}(x_2,x_3)\otimes \Hom_{\Q}(x_3,x_4)$, we have two possibilities:
\begin{itemize}
\item[] It $T\in\mathcal{f}\  \MCS\ ,\  \MCD\  \mathcal{g}$, then:
\begin{align*}
\mTT^1\left(\big(\ T\ ;\ f_1,f_2,f_3\ \big)\right)&= (-1)^{\tiny\mbox{deg($f_1$)}} \big(\ T\ ;\ \dQ (f_1),f_2,f_3\ \big) \\
&+(-1)^{\tiny\mbox{deg($f_2$)}} \big(\ T\ ;\ f_1,\dQ (f_2),f_3\ \big)\\
&+ (-1)^{\tiny\mbox{deg($f_3$)}} \big(\ T\ ;\ f_1,f_2,\dQ (f_3)\ \big).
\end{align*}
\item[] Otherwise:
\begin{align*}
\mTT^1\left(\ \big(\ \MTr\ ;\ f_1,f_2,f_3\ \big)\right)&= \big(\ \MCS\ ;\ f_1,f_2,f_3\  \big) \\
&+ \big(\ \MCD\ ;\ f_1,f_2,f_3\ \big)\\
&+ (-1)^{\tiny\mbox{deg($f_1$)}} \big(\ \MTr\ ;\ \dQ (f_1),f_2,f_3\ \big) \\
&+(-1)^{\tiny\mbox{deg($f_2$)}} \big(\ \MTr\ ;\ f_1,\dQ (f_2),f_3\ \big)\\
&+ (-1)^{\tiny\mbox{deg($f_3$)}} \big(\ \MTr\ ;\ f_1,f_2,\dQ (f_3)\ \big).
\end{align*}
\end{itemize}
\end{itemize}
\end{exmp}

\begin{notatu}\label{notello}
We denote by $PT_n(\Q)(x,y)$ the DG $R$-module 
\begin{align*}
PT_n(\Q)(x,y):=\displaystyle\bigoplus_{\substack{\mathfrak{t}\in PT_n \\ x=x_1,...,x_{n+1}=y}}\Hom_{\Q}(x_1,x_{2}) \otimes ...\otimes \Hom_{\Q}(x_{n},x_{n+1}).
\end{align*}
Note that, if $T\in PT_n(\Q)(x,y)$ then $\mTT^1(T)\subset PT_n(\Q)(x,y)$.\\
Clearly 
$$PT_n(\Q)(x,y)\subset \Hom_{\TT(\Q)}(x,y)$$ 
as DG $R$-modules.\\
We denote 
$$PT_n(\Q):=\displaystyle\bigoplus_{x,y\in\Q}PT_n(\Q)(x,y).$$
Moreover, we denote by $PT^l_n(\Q)(x,y)$ the graded $R$-module 
\begin{align*}
PT_n(\Q)(x,y):=\displaystyle\bigoplus_{\substack{\mathfrak{t}\in PT^l_n \\ x=x_1,...,x_{n+1}=y}}\Hom_{\Q}(x_1,x_{2}) \otimes ...\otimes \Hom_{\Q}(x_{n},x_{n+1}),
\end{align*}
for every $\mathfrak{t}\in PT^l_n$ and $x_2,...,x_{n}\in\Q$.\\ 
We denote 
$$PT^l_n(\Q):=\displaystyle\bigoplus_{x,y\in\Q}PT^l_n(\Q)(x,y).$$
Note that, if $T\in PT^l_n(\Q)(x,y)$ and $l>1$, then $\mTT^1(T)\not\subset PT^l_n(\Q)(x,y)$ (see Example \ref{diffex}).
\end{notatu}

We can make Construction \ref{Liberanzo} functorial.\\
Given a functor $\sF$ between two DG quivers $\sF:\Q_1\to\Q_2$ we have a strict non unital $\Ain$functor
\begin{align}\label{functfree}
\TT(\sF):\TT(\Q_1)\to&\TT(\Q_2).
\end{align}
Defined as follows:
\begin{itemize}
\item[1.] We set $\TT(\sF)^0(x)=\sF^0(x)$ for any object $x\in\Q_1$. 
\item[2.] Fixed two objects $x,y\in\Q_1$ we have:
\begin{align*}
\TT(\sF)^1:\Hom_{\TT(\Q_1)}(x,y)\to&\Hom_{\TT(\Q_2)}(\TT(\sF)^0(x),\TT(\sF)^0(y))\\
(\mathfrak{t};f_1,...,f_n)\mapsto& (\mathfrak{t};\sF^1(f_1),...,\sF^1(f_n)).
\end{align*}
\end{itemize}

\begin{defn}[Free $\Ain$category]
A \emph{free $\Ain$category} is a non unital $\Ain$category $\A$ which is of the form $\TT(\Q)$, where $\Q$ a DG quiver. 
\end{defn}

Given a non unital $\Ain$category $\A$, we have a natural strict non unital $\Ain$functor 
\begin{align}\label{beta}
\beta_{\A}:\TT(|\A|)&\to \A\\
(\mathfrak{t};f_1,...,f_n)&\mapsto \mathfrak{t}_{m_{\A}}(f_n,...,f_1).
\end{align}
Where $\mathfrak{t}_{m_{\A}}(f_n,...,f_1)$ is the "combination" of maps $m^n_{\A}$ obtained by setting 
$$m^n_{\TT(\A)}=m^n_{\A}$$
for every positive integer $n$.\ See the following:
\begin{figure}[htbp]
\centerline{\includegraphics[width=.9\textwidth, height=.3\textheight, keepaspectratio]{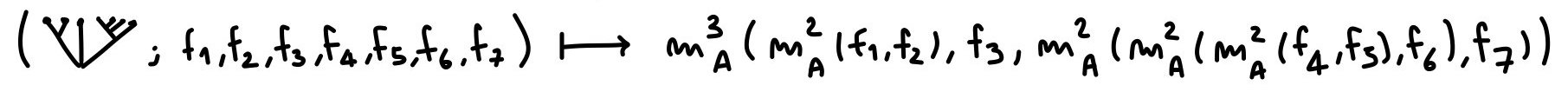}}
%\caption{This is an image from a text that uses color to teach music.}
%\label{fig}
\end{figure}
\\
On the other hand, we can consider the forgetful functor $|\mbox{-}|$ (see formula (\ref{foffofo})).\\
We have a natural functor between DG quivers
\begin{align}\label{unitiono}
\alpha_{\Q}:\Q&\to |\TT(\Q)|\\
f&\mapsto (\bullet;f).
\end{align}
One can prove (with a direct calculation) that $\alpha$ and $\beta$ are the unit and counit of an adjunction.\\
We have:
\begin{align}\label{adjunct}
\TT\dashv |\mbox{-}|:\aCat^{\tiny\mbox{nu}}_{\tiny\mbox{strict}}\to{\DgQui}.
\end{align}
See \cite[2.5 Corollary]{LM1}.\\
\\
Let $\A$ be a non unital $\Ain$category.\
\begin{defn}[$\Ain$ideal] 
An $\Ain$ideal $I$ of $\A$ is a subquiver of $|\A|$ (see Definition \ref{subsquinzi}) such that:
\begin{align}\label{ideale}
m_{\mathcal{A}}^n(f_{n-1},...,f_{l+1},f,f_{l-1}, ...,f_0) &\subset \Hom_{I}(x_0,x_n),
\end{align}
for every sequence of objects $x_0,...,x_{l-1},x,y,x_{l+2},...,x_n\in\A$ and morphisms $f_i\in\Hom_{\A}(x_i,x_{i+1})$, 
and $f\in\Hom_{I}(x,y)$.% where $x_l=x$ and $x_{l+1}=y$.
\end{defn}

\begin{exmp}
Given a DG category $\mathcal{C}$, every DG ideal of $\mathcal{C}$ is an $\Ain$ideal of $\mathcal{C}$.
\end{exmp}

\begin{defn}[$\Ain$quotient]
Let $I$ be an $\Ain$ideal of $\A$, we can define the non unital $\Ain$category $\A/I$ as follows:
\begin{itemize}
\item[1.] $\A/I$ has the same objects of $\A$.
\item[2.] Given $x,y\in\A/I$ the hom-space is given by the quotient of (graded) $R$-modules
\begin{align*}
\Hom_{\A/I}(x,y)&:=\Hom_{\A}(x,y)/I(x,y).
\end{align*}
\item[3.] The $\Ain$structure is given, for every $n\ge 0$, by:
\begin{align*}
m^n_{\A/I}([f_n],...,[f_1])&:=[m^n_{\A}(f_n,...,f_1)].
\end{align*}
Where $[f_i]\in\Hom_{\A/I}(x_i,x_{i+1})$ and $f_i\in\A$ is a representative of the class $[f_i]$.
\end{itemize}
\end{defn}
\begin{rem}
We note that the quotient above is well defined.\ Given an element $\tilde{f}_i\in[f_i]$, by (\ref{ideale}), we have 
$$m^n_{\A}(f_n,...,f_i,...,f_1)=m^n_{\A}(f_n,...,f'_i,...,f_1)+I$$ 
It implies: 
\begin{align*}
m^n_{\A/I}([f_n],...,[f_i],...,[f_1])&=m^n_{\A/I}([f_n],...,[f'_i],...,[f_1]).
\end{align*} 
\end{rem}
For every $\Ain$ideal $I$ of $\A$ we have a strict functor (see Definition \ref{astricts}) $q:\A\to\A/I$ (surjective on the morphisms) given by
\begin{align}\label{quoti}
q^0:\mbox{Obj}(\A)&\to\mbox{Obj}(\A/I)\\
a& \mapsto a\\
q^1:\Hom_{\A}(x,y)&\to\Hom_{\A/I}(x,y)\\
f&\mapsto [f].
\end{align}
Let $\A$ be an $\Ain$category, we have the associated DG quiver $|\A|$.\\ 
A \emph{graded subquiver}\footnote{Note that we do not require that $d_{R}(R)\subset R$, so it is different from Definition \ref{subsquinzi}} $R$ of $|\A|$ is a quiver whose objects are the same of $\A$ and, fixed two objects $x,y\in R$ $\Hom_R(x,y)$ is a dg $R$-submodule of 
$\Hom_{\A}(x,y)$.\\ 
We denote by $(R)$ the graded subquiver of $|\TT(\A)|$ whose objects are the same of $|\A|$ and whose morphisms are generated (cf Example \ref{genera}) by
\begin{align*}
m^n_{\TT(\A)}(f_n,...,r,...,f_1).
\end{align*}
for every $f_i\in\TT(\A)$ and $r\in R$.
\begin{defnthm}[]\label{relaz}
We call \emph{a system of relations} a graded subquiver $R$ of $|\A|$ such that, given a morphism $f\in |\A|$ then $d_{|\A|}(f)\subset (R)$.\ Given a system of relations $R$ the graded subquiver $(R)$ is an $\Ain$ideal of $\TT(\A)$.
\end{defnthm}

Given a DG quiver $\Q$, by construction, the free $\Ain$category $\TT(\Q)$ is a non unital $\Ain$category.\ It can be made an augmented strictly unital $\Ain$category as follows:\

\begin{cons}[Strictly unital free $\Ain$category]\label{UNITTTT}
Let $\Q$ be a DG quiver.
\begin{itemize}
\item[1.] First we consider the discrete DG quiver $\mathsf{I}_{\Q}=|\mbox{disc}(\Q)|$ (see Definition \ref{disco}).\
\item[2.] We take the DG quiver $\Q_{+}:=\Q+\mathsf{I}_{\Q}$ (cf Definition \ref{sumqui}).
\item[3.] The $\Ain$ideal $I_u$ of $\TT(\Q)$ is defined by the system of relations $R_u$ spanned by:
\begin{align*}
\big(\  \mathfrak{t}\vee(\bullet;1)- \mathfrak{t} \ , \ (\bullet;1)\vee \mathfrak{t}- \mathfrak{t} \ ,\ (\mathfrak{T}_n; f_n,...,1_x,...,f_1 )  \ \big),
\end{align*}
for any morphism $\mathfrak{t}\in\TT(\Q)$ and $\mathfrak{T}_n$, and sequence of morphisms of $\A$.\ We recall that $\mathfrak{T}_n$ is the tree with $n$-leaves without nodes, see section \ref{gloglossary}.\
\item[4.] The strictly unital $\Ain$category $\TT(\Q)_+$ (cf subsection \ref{unit}) is defined as the quotient $\TT(\Q_+)/ I_u$.
\end{itemize}
\end{cons}

\begin{rem}\label{primacostr}
We note that to define a strict $\Ain$functor $\sF$ with source a free $\Ain$category $\TT(\Q)$ is sufficient to know $\sF^1|_{\Q}$.\ This follows by the adjunction (\ref{adjunct}):
\begin{align*}
\Hom_{\tiny\aCat}(\TT(\Q),\A)&\simeq\Hom_{\tiny\mbox{DGQuiv}}(\Q,|\A|)\\
\sF&\mapsto\sF^1_{|_\Q}
\end{align*}
where $\A$ is a non unital $\Ain$category.\\
Suppose now that $\A$ is a strictly unital $\Ain$category.\ Every strictly unital $\Ain$functor 
$\sF:\TT(\Q)_{+}\to \A$ is uniquely determined by $\sF^1|_{\Q}$ since the unit $1_{+}$ of $\TT(\Q)_{+}$ must be preserved by $\sF$.\\
On the other hand, given a non unital strict $\Ain$functor $\psi:\TT(\Q)\to \A$, we can extended uniquely $\psi$ to a strictly unital $\Ain$functor:
\begin{align}\label{stronk2}
\psi_{+}:\TT(\Q)_{+}&\to \A
\end{align}
As follows:
\begin{itemize}
\item[1.] $\psi_{+}|_{\TT(\Q)}=\psi$,
\item[2.] $(\bullet;1_{x})\mapsto 1_{x}^{\A}$, for every object $x\in\TT(\Q)$.
\end{itemize}
\end{rem}

\begin{exmp}\label{ex1}
Let $\Q$ be a DG quiver.\ The graded subquiver $R$ of $\TT(\Q)$ whose morphisms are those of the form 
\begin{align*}
R:=\big(\ \MCS\ ;\  f_1,f_2,f_3\ \big)-\big(\ \MCD \ ;\ f_1,f_2,f_3\ \big)
\end{align*}
for every $f_1,f_2,f_3\in\Q$, is a system of relations (cf. Example \ref{diffex} $PT_3)$).\
On the other hand, the following graded subquiver
\begin{align*}
R_{\mathsf{As}}=\Big(\ R \ , \big(\ \MTr \ ;\ f_4,f_5,f_6\ \big),\big(\ \MQu \ ;\ f_7,..., f_{10}\big) , ... \Big)
\end{align*}
where $f_j\in\Q$, is a system of relations .
\end{exmp}

Using Example \ref{ex1} we can define the \emph{non unital free DG categories} and \emph{free DG categories}:

\begin{cons}\label{Drinniborguz}
Let $\Q$ be a DG quiver.\ We can associate to $\Q$,
\begin{itemize}
\item[i.] the non unital free DG category, denoted by $\TT_{\tiny{\mbox{DG}}}(\Q)$, is given by the quotient
\begin{align*}
\TT(\Q)/\mathsf{As}.
\end{align*}
Where $\mathsf{As}$ is the $\Ain$ideal generated by the system of relations $R_{\mathsf{As}}$ (see Example \ref{ex1}).\\ 
Clearly $\TT_{\tiny{\mbox{DG}}}(\Q)$ is associative, we denote by $f*g$ the composition 
$$m^2_{\TT_{\tiny{\mbox{DG}}}(\Q)}(f,g).$$ 
\item[ii.] The free DG category, denoted by $\TT_{\tiny{\mbox{DG}}}(\Q)_{+}$, is given by the quotient
\begin{align*}
\TT_{\tiny{\mbox{DG}}}(\Q+\sf{I}_{\Q})/ I^{\tiny\mbox{DG}}_u.
\end{align*}
Here $\sf{I}_{\Q}$ is the discrete category (see Definition \ref{disco}) and $I^{\tiny\mbox{DG}}_u$ is the DG ideal whose morphisms are of the form: 
$$(\ f*1-f \ ,\ 1*f-f \ )$$
for every $f\in \Q$.
\end{itemize}
%Cleary, if we take $\Q=\Q_S$, as in (\ref{homoclo}), it is possible to take the (non unital) free DG category associated to $\Q$, even in the case of an $\Ain$category.\ 
\end{cons}

\begin{exmp}\label{ex2}
Given $\A$ an $\Ain$category, we have a morphism of graded $R$-modules:
\begin{align*}
\delta_{x,y}:\displaystyle\bigoplus_{\substack{n\ge 1\\ x=x_1,...,x_{n+1}=y}}\Hom_{|\A|}(x_1,x_2)\otimes...\otimes \Hom_{|\A|}(x_n,x_{n+1})\to \Hom_{\TT(|\A|)}(x,y)
\end{align*}
defined as: 
\begin{align*}
\delta(f_n\otimes...\otimes f_1):=\big(\alpha_{|\A|}\cdot m^n_{\A} - m^n_{\TT(|\A|)}\cdot \alpha_{|\A|}^{\otimes n}\big) (f_n\otimes...\otimes f_1).
\end{align*}
Here $\alpha_{|\A|}$ is the unit of the adjunction (\ref{unitiono}).\\
We can define a graded subquiver $R_{\A}$ of $\TT(|\A|)$ taking $\Hom_{R_{\A}}(x,y):=\mbox{Im}(\delta_{x,y})$.\\
In other words the morphisms of $R_{\A}$ are spanned by 
\begin{align}
\alpha_{|\A|}\big(m^n_{\A}(f_1,...,f_n)\big)-\big( \mathfrak{T}_n; \alpha_{|\A|}(f_1),...,\alpha_{|\A|}(f_n)\big)
\end{align}
for every $n>1$ and $f_1,...,f_n\in\A$.\\
It important to say that $R_{\A}$ is a system of relations and $\A$ is equivalent (see Definition \ref{equiv}) to $\TT(|\A|)/R_{\A}$ via two non unital strict $\Ain$functors \cite[3.2 Proposition]{LM2}.
\end{exmp}

\subsection{$\Ain$functors and natural transformations from free $\Ain$categories.}
This subsection characterizes the non unital $\Ain$functors with source a free $\Ain$category.\ 
The results of this subsection can be found in a more general setting in \cite[\S 2]{LM1}.\ 
 
\begin{thm}[\mbox{\cite[2.3 Proposition]{LM1}}]\label{2.3PropLM1}
A non unital $\Ain$functor $\F:\TT(\Q)\to \A$ is uniquely determined by the following datum:
\begin{align}
(F^1,F^2,...,F^n,...).
\end{align}
Where 
\begin{align}
F^1:\Q\to |\A|
\end{align}
is a morphism of DG quivers.\ We denote by $F^0:\mbox{Ob}(\Q)\to\mbox{Ob}(\A)$ is the morphism of sets.\\
For every $n>1$ and $n+1$-objects in $\Q$
\begin{align}
F^n:\Hom_{\TT(\Q)}(x_0,x_1)\otimes...\otimes \Hom_{\TT(\Q)}(x_{k-1},x_k)&\to \Hom_{\A}(f_0(x_0),f_0(x_k))[1-n]
\end{align}
is a morphism of graded (non DG) complexes.
\end{thm}

\begin{proof}[Sketch of proof]
We fix a non unital $\Ain$functor $\F:\TT(\Q)\to\A$.\ Let $f$ be a morphism in $\Q$, we set 
\begin{align*}
F^1(f):=\F^1\big( (\bullet; f)\big).
\end{align*}
This is a morphisms of DG quivers since:
\begin{align*}
d_{|\A|}\big(F^1(f)\big)&=m^1_{\A}(\F^1\big( (\bullet; f)\big))\\
&=\F^1\big(m_{\TT(\Q)}^1 (\bullet; f)\big)\\
&=\F^1\big((\bullet; m_{\Q}^1(f))\big)\\
&=F^1\big( m_{\Q}^1(f)\big).
\end{align*}
Moreover, given $m$-morphisms $g_1,...,g_m$ in $|\TT(\Q)|$, then 
\begin{align*}
F^m(g_1,...,g_m):=\F^m(g_1,...,g_m).
\end{align*}
defines a morphism of graded quivers.\\
On the other hand, given the datum
\begin{align*}
(F^1,F^2,...,F^n,...).
\end{align*}
A planar rooted tree with one leaf (namely i.e. for every $(\bullet; f)$) we set:
\begin{align}\label{roote}
\F^1\big( (\bullet; f)\big):=F^1\big( f\big).
\end{align}
Every tree is generated by the 1-leaf ones (cf. Example \ref{genera}) so (\ref{roote}) is sufficient to determined $\F^1$.
\end{proof}

\subsection{$\Ain$functors from a quotient of a free $\Ain$category}\label{quotran}
In this section we are interested in non unital $\Ain$functors and natural transformations with source a quotient of a free $\Ain$category.\ The more interested readers can take a look at \cite[\S2]{LM1}.\\ 
We start with a Lemma.
\begin{lem}\label{factone}
Let $\A$, $\B$ and $\C$ three non unital $\Ain$categories and $\F:\A\to\B$ a strict non unital $\Ain$functor.\
If $\F$ factorizes as
\begin{align*}
\xymatrix@R=0.7em{
\A\ar[rr]^{\F}\ar[rd]_{\G}&&\C\\
&\B\ar[ru]_{\F'}&
}
\end{align*}
Where $\G$ is a strict non unital $\Ain$functor surjective on the morphisms.\ 
Then $\F'$ is a strict non unital $\Ain$functor.
\end{lem}

\begin{proof}
We have to prove that $\F'^n(f_n,...,f_1)=0$, for any sequence of morphisms $f_n,...,f_1\in\B$.\
Since $\G$ is surjective on the morphisms then $f_i=\G^1(a_i)$, where $a_i$ is a morphism of $\A$ for every $i$ such that $1\ge i\ge n$.\ We have $\F'^n(f_n,...,f_1)=\F'^n\big(\G^1(a_n),...,\G^1(a_1)\big)$.\ By the formula of composition of $\Ain$functors (see Definition \ref{composizioncella}), we have $\F'^n(f_n,...,f_1)=\F^n(a_n,...,a_1)$.\ This is zero since $\F$ is strict and we are done.
\end{proof}

We fix a non unital $\Ain$category $\A$ and a DG quiver $\Q$.\\
Let $R$ be a system of relations (see Definition-Theorem \ref{relaz}) of $\TT(\Q)$.\\ 
We denote by $I:=(R)$ the $\Ain$ideal generated by $R$ and $q$ the strict non unital $\Ain$functor defined in (\ref{quoti}).\\
\\
Let $\F:\TT(\Q)\to\B$ be a non unital $\Ain$functor.\\ 
Clearly $\F$ factorizes (uniquely) as
\begin{align}\label{feeeeer}
\xymatrix@R=0.7em{
\TT(\Q)\ar[rr]^{\F}\ar[rd]_-q&&\A\\
&\TT(\Q)/I\ar[ru]_-{\F'}&
}
\end{align}
if and only if $\F^n(...,I,...)=0$ for every $n>0$.\\
The following Theorem provides an easier criteria to verify.

\begin{thm}[\mbox{\cite[Proposition 2.5]{LM2}}]\label{2.5PropLM2}
The non unital $\Ain$functor $\F$ factorizes as in (\ref{feeeeer}) if and only if the following two conditions are satisfied:
\begin{itemize}
\item[1.] For any morphism $r\in R$ we have $\F^1(r)=0$.
\item[2.] For any morphism $i\in I$, any positive integer $n>1$, and $f_k\in\TT(\Q)$, we have:
$$\F^n(f_{n-1},...,f_j,i,f_{j-1},...,f_1)=0.$$
\end{itemize}
\end{thm}
It means that, if $\mathsf{F}$ is a strict $\Ain$functors, $\mathsf{F}(I)=0$ if and only if $\mathsf{F}(R)=0$.

\newpage

\section{Limits and colimits in the categories of $\Ain$categories}\label{cocompletone}

In this section we prove that the category $\aCat^{\tiny\mbox{nu}}_{\tiny\mbox{strict}}$ (and $\aCat_{\tiny\mbox{strict}}$) is complete and cocomplete.\ In particular, the cocompleteness of $\aCat_{\tiny\mbox{strict}}$ will be fundamental (in Section \ref{semifreeee}) to prove the existence of a semi-free resolution for $\Ain$categories.\\

It is well known that the categories of DG quivers and DG categories are complete and cocomplete.\ But in the case of $\Ain$categories the situation is much more ambiguous.\
First the category $\aCat^{\tiny\mbox{nu}}$ is not complete since it does not admit equalizers (see [COS; \S 1.5]\footnote{The proof is done for the category of strictly unital and cohomological unital $\Ain$categories but it works also for the non unital ones}) but it admits products and coproducts.\
Let me describe explicitly the product $\A\times\B$ of two non-unital $\Ain$categories $\A$ and $\B$:
\begin{itemize}
\item[1.] $\mbox{Ob}(\A\times\B):=\mbox{Ob($\A$)}\times\mbox{Ob($\B$)}$.
\item[2.] $\Hom_{\A\times\B}\big((a,b),(a',b')\big):=\Hom_{\A}(a,a')\oplus\Hom_{\B}(b,b')$.
\item[3.] Fixed $n$-objects, the $\Ain$structure is given by: 
\begin{align*}
m^n:\Hom_{\A\oplus\B}\big( (a_{n-1},b_{n-1}),(a_{n},b_{n}) \big)\otimes ...\otimes\Hom_{\A\oplus\B}\big( (a_{1},b_{1}),(a_{2},b_{2}) \big)&\to\Hom_{\A\oplus\B}\big( (a_{1},b_{1}),(a_{n},b_{n}) \big)[2-n]\\
(f_n,g_n)\otimes....\otimes(f_1,g_1)&\mapsto \big( m_{\A}^n(f_n,...,f_1),m_{\B}^n(g_n,...,g_1)\big).
\end{align*}
\end{itemize}
The coproduct of $\A$ and $\B$ is given by
\begin{itemize}
\item[1.] $\mbox{Ob}(\A\coprod\B):=\mbox{Ob}(\A)\coprod \mbox{Ob}(\B)$.
\item[2.] The hom-spaces are defined as follows:
\begin{align*}
\Hom_{\A\coprod\B}(x,y):=
&\begin{cases}
\Hom_{\A}(x,y)    & \text{if } x,y \in \A \\
  \Hom_{\B}(x,y)    & \text{if } x,y \in \B \\
0        & \text{otherwise}.\\
  \end{cases}
\end{align*}
\item[3.] The $\Ain$structure is the one induced by $\A$ and $\B$.
\end{itemize}
By easy calculations, if $\A$ and $\B$ are strictly unital, then even their product and coproduct is so.

\begin{rem}
Clearly, if $A$ and $B$ are two $\Ain$algebras, then $A\times B$ is the product in the category of $\Ain$algebras.\ 
But the coproduct described above is not contained in the category of $\Ain$algebras (since it has two objects).\ In [Orn3] it is proven that there no exists a coproduct $A\coprod B$ in the category of $\Ain$algebras.\ It has to deal with the problem of a "good" notion of tensor for $\Ain$algebras/categories.\ What we have is the existence of homotopy coproduct of $\Ain$algebras, in [Orn3] it is given an explicit description making use of semi-free resolutions of Section \ref{semifreeee}.
\end{rem}

\begin{thm}\label{equizz}
The categories $\aCat^{\tiny\mbox{nu}}_{\tiny\mbox{strict}}$ and $\aCat_{\tiny\mbox{strict}}$ have equalizers.
\end{thm}

\begin{proof}
Let us describe the equalizers in $\aCat^{\tiny\mbox{(nu)}}_{\tiny\mbox{strict}}$.\ 
Given $\A$, $\B$ two $\Ain$categories and $\mathsf{F},\mathsf{G}:\A\to\B$ two strict $\Ain$functors, then the equalizers\
\[
\xymatrix{
\mathscr{E}\ar[r]&\A\ar@<0.8ex>[r]^{\mathsf{F}}\ar@<-0.1ex>[r]_{\mathsf{G}}&\B
}
\]
is given as follows.\ The morphisms are 
\begin{align*}
\Hom_{\mathscr{E}}\big( (a,b),(a',b') \big):=\mathcal{f} (f,g)\in\Hom_{\A\times\B}\big( (a,b),(a',b') \big)\mbox{ such that } \mathsf{F}^1(f)=\mathsf{G}^1(g)\big) \mathcal{g}.
\end{align*}
The $\Ain$structure on $\mathscr{E}$ is the one of $\A\times\B$.\ It is well defined since, by Remark $\ref{SFE}$,
\begin{align*}
\mathsf{F}^1\big(m_{\A}^n(f_n,...,f_1)\big)=m^n_{\B}(\mathsf{F}^1(f_n),...,\mathsf{F}^1(f_1))=m^n_{\B}(\mathsf{G}^1(f_n),...,\mathsf{G}^1(f_1))=\mathsf{G}^1\big(m_{\A}^n(f_n,...,f_1)\big).
\end{align*}
\end{proof}

%\begin{rem}
%PULLBACK SU FIBRAZIONI!! KELLER?
%\end{rem}

\begin{thm}\label{coeq}
The categories $\aCat^{\tiny\mbox{nu}}_{\tiny\mbox{strict}}$ and $\aCat_{\tiny\mbox{strict}}$ have coequalizers.
\end{thm}

In order to do that we need two Lemmas.

\begin{lem}
\label{reflex}
Let $\A$ and $\B$ be two $\Ain$categories, and a diagram
\begin{align}\label{tongo}
\xymatrix{
\A\ar@<0.8ex>[r]^{\mathsf{F}}\ar@<-0.1ex>[r]_{\mathsf{G}}&\B
}
\end{align}
in $\aCat^{\tiny\mbox{(nu)}}_{\tiny\mbox{strict}}$.\
If there exists a strict $\Ain$functor $r:\B\to\A$ such that, $G\cdot r=F\cdot r=\mbox{Id}_{\B}$,
then (\ref{tongo}) has a coequalizer. 
\end{lem}

\begin{proof}
First we note that the image of $(F-G)$ is an $\Ain$ideal of $\B$.\ To see this we calculate: 
\begin{align*}
m^n_{\B}(f_n,..,(F-G)(f),...,f_1)&=m^n_{\B}(f_n,..,F(f),...,f_1)-m^n_{\B}(f_n,..,G(f),...,f_1).\\
&=m^n_{\B}\big(F\cdot r(f_n),..,F(f),...,F\cdot r( f_1))-m^n_{\B}(G\cdot r(f_n),..,G(f),...,G\cdot r(f_1)\big)\\
&=F\big(m^n_{\A}\big( r(f_n),..,f,..., r( f_1)\big)-G\big(m^n_{\A}\big( r(f_n),..,f,..., r( f_1)\big)\\
&=(F-G)\big(m^n_{\A}\big( r(f_n),..,f,..., r( f_1)\big).
\end{align*}
It easy to see that the quotient $\B/(F-G)$ is the coequalizer in $\aCat^{\tiny\mbox{(nu)}}_{\tiny\mbox{strict}}$.\\
Clearly if $\A$, $\B$, $F$ and $G$ are strictly unital then even the coequalizer is so.
\end{proof}

\begin{exmp}\label{frotot}
Let $\A$ be an $\Ain$category.\ The diagram 
\begin{align*}
\xymatrix@C=3.4em{
\TT\big(|\TT(|\A|)|\big)\ar@<1ex>[r]^-{\TT|\beta_{\A}|}\ar@<-0.5ex>[r]_-{\beta_{\tiny\TT(|\A|)}}& \TT(|\A|)
}
\end{align*}
satisfies the hypothesis of Lemma \ref{reflex}.\ 
The strict $\Ain$functor $\TT\cdot\alpha_{|\A|}$ plays the role of $r$.\ 
This is clear since $\TT\dashv |\mbox{-}|$ and $\alpha$, $\beta$ (defined in \ref{unitiono} and \ref{beta}) are respectively the unit and the counit of the the adjunction.\\
In this case the coequalizer is given by 
\begin{align*}
\xymatrix@C=2.7em{
\TT\big(|\TT(|\A|)|\big)\ar@<1ex>[r]^-{\TT|\beta_{\A}|}\ar@<-0.5ex>[r]_-{\beta_{\tiny\TT(|\A|)}}& \TT(|\A|) \ar@{->>}[r]& \TT(|\A|)/R_{\A}
}
\end{align*}
Where $R_{\A}$ is the system of relations of Example \ref{ex2}.\ So $\A$ (with the counit $\beta_{\A}:\TT(|\A|)\to\A$) is a coequalizer.
\end{exmp}
Actually Lemma \ref{reflex} is sufficient to prove that $\aCat^{\tiny\mbox{(nu)}}_{\tiny\mbox{strict}}$ has coequalizers.\ 
In order to do that we recall the following result due to [Duskin; pp 77-78].

\begin{lem}
Given a diagram in any category:
\begin{align}\label{dusk}
\xymatrix{
A_1\ar@<0.8ex>[r]^{h_2}\ar@<-0.1ex>[r]_{h_1}\ar@<0.8ex>[d]^{\alpha_2}\ar@<-0.2ex>[d]_{\alpha_1}&B_1\ar[r]^{h_3}\ar@<0.8ex>[d]^{\beta_2}\ar@<-0.2ex>[d]_{\beta_1}&C_1\ar@<0.8ex>[d]^{\gamma_2}\ar@<-0.2ex>[d]_{\gamma_1}\\
A_2\ar@<0.8ex>[r]^{g_2}\ar@<-0.1ex>[r]_{g_1}\ar[d]_{\alpha_3}&B_2\ar[r]^{g_3}\ar[d]_{\beta_3}&C_2 \ar[d]_{\gamma_3} \\
A_3\ar@<0.8ex>[r]^{f_2}\ar@<-0.1ex>[r]_{f_1}&B_3\ar[r]^{f_3}&C_3\\
}
\end{align}
Such that: 
\begin{itemize}
\item[1.] $f_i\cdot\alpha_3=\beta_3\cdot g_i$, $f_3\cdot\beta_3=\gamma_3\cdot g_3$, $g_i\cdot\alpha_i=\beta_i\cdot h_i$ and $g_3\cdot\beta_i=\gamma_i\cdot h_3$ for every $i=1,3$.
\item[2.] The top two rows are coequalizers.
\item[3.] The two left columns are coequalizers.
\end{itemize}
The following are equivalent:
\begin{itemize}
\item[i.] The bottom row is a coequalizer.
\item[ii.] The right-hand columns is a coequalizer.
\item[(iii.] The down-right square is a pushout.)
\end{itemize}
\end{lem}

\begin{proof}[Proof of \ref{coeq}]
We want to deduce the existence of the coequalizer of
\begin{align*}
\xymatrix{
\A\ar@<0.8ex>[r]^{\mathsf{F}}\ar@<-0.1ex>[r]_{\mathsf{G}}&\B
}
\end{align*}
using the previous Lemma.\ So we want to build the $3\times 3$ diagram of the form (\ref{dusk}).\ 
We start by taking $E$ the coequalizers of
\begin{align}\label{diagrammino}
\xymatrix{
|\A|\ar@<0.8ex>[r]^{|\mathsf{F}|}\ar@<-0.1ex>[r]_{|\mathsf{G}|}&|\B|\ar[r]^q&E
}
\end{align}
in the category of DG Quivers.\\
Now we consider the diagram:
\begin{align}\label{diagrammone}
\xymatrix@=4em{
\TT\big(|\TT(|\A|)|\big)\ar@<0.8ex>[r]^{h_2}\ar@<-0.1ex>[r]_{h_1}\ar@<0.8ex>[d]^{\alpha_2}\ar@<-0.2ex>[d]_{\alpha_1}&\TT\big(|\TT(|\B|)|\big)\ar[r]^{\TT(q' )} \ar@<0.8ex>[d]^{\beta_2}\ar@<-0.2ex>[d]_{\beta_1}&\TT(E')\ar@<0.8ex>[d]^{\gamma_2}\ar@<-0.2ex>[d]_{\gamma_1}\\
\TT(|\A|)\ar@<0.8ex>[r]^{g_2}\ar@<-0.1ex>[r]_{g_1}\ar[d]_{\beta_{\A}}&\TT(|\B|) \ar[d]_{\beta_{\B}}\ar[r]^{\TT(q)}&\TT(E)\ar[d]_{\gamma_3}\\
\A\ar@<0.8ex>[r]^{\mathsf{F}}\ar@<-0.1ex>[r]_{\mathsf{G}}&\B\ar[r]^p&\C
}
\end{align}
Where:
\begin{itemize}
\item[] $h_2= \TT\big(|\TT(|\mathsf{F}|)|\big)$, $h_1= \TT\big(|\TT(|\mathsf{G}|)|\big)$.
\item[] $g_2= \TT(|\mathsf{F}|)$, $g_1= \TT(|\mathsf{G}|)$.
\item[] $\alpha_1=\TT|\beta_{\A}|$, $\alpha_2=\beta_{\tiny\TT(|\A|)}$ and $\beta_1=\TT|\beta_{\B}|$, $\beta_2=\beta_{\tiny\TT(|\B|)}$
\end{itemize}
Since $E$ is the coequalizer of (\ref{diagrammino}) then $\TT(E)$ is the coequalizer of the second row.\ 
In the same vein $\TT(E')$ is the coequalizer of the first row, where $E'$ is the coequalizer: 
\begin{align}
\xymatrix{
|\TT(|\A|)|\ar@<0.8ex>[r]^{|\TT(\mathsf{F})|}\ar@<-0.1ex>[r]_{|\TT(\mathsf{G})|}&|\TT(|\B|)|\ar[r]^-{q'} &E'.
}
\end{align}
On the other hand by Example \ref{frotot} we know that the first, and the second column, are coequalizers.\\ 
The existence (and uniqueness) of $\gamma_2$ and $\gamma_1$ is given by the universal property of the (first row) coequalizer taking respectively $\TT(q)\cdot \beta_2$ and $\TT(q)\cdot \beta_1$.\ The right down vertical arrow $\gamma_3$ is the coequalizer provided by Lemma \ref{reflex} (note that the hypothesis are satistied) and $p:\B\to\C$ is given by the universal property of the (middle column) coequalizer.\\
By the implication ii.$\Rightarrow$ i. of Lemma \ref{dusk} we have that (\ref{tongo}) has a coequalizer which is given by the bottom row of diagram (\ref{diagrammone}).\\
It is not difficult to prove that, if $\A$, $\B$ and $\mathsf{F}$, $\mathsf{G}$ are strictly unital then $\C$ is strictly unital.
\end{proof}

The strategy we used to prove Theorem \ref{coeq} is not new (see for example \cite[Proposition 2.11]{Wolff}).\ More recently was used by \cite[\S1.2.6]{PS} to proved that the category of weakly unital DG Categories has coequalizers.\ 
It is a (kind of a) consequence of (one of the formulation of) Beck (Crude or Duskin) Monadicity Theorem.\ 
%We suggest the reading of the paper \cite{PS} to whom is not familiar with these techniques.\ 
In \cite[\S 1.2.4]{PS} one can found also the explicit (standard) construction of the coequalizers in the category of DG quivers if needed that we omitted.\\ 
%To be honest 
Even in the case of the category of DG categories the construction of coequalizers is not trivial at all (see \cite{MO}).\ 
One can deduce the existence of coequalizers in $\DgCat$ from Proposition 7.2-7.4 and Lemma 6.6 of \cite[Chapter II]{EKMM}.\ 
Alternatively, %as was noted in \cite[Lemma 2.8]{HL}, 
we can describe explicitly the coequalizers in $\DgCat$ using again the same strategy of Theorem \ref{coeq}.\
%Taking the adjunction 
%\begin{align}
%\xymatrix{
%|\mbox{-}|: \DgCat\ar@<0.8ex>[r]&\ar@<0.3ex>[l]\mbox{GrCat}: F
%}
%\end{align}
%where GrCat denotes the category of graded-categories.\

\begin{thm}\label{Cocompletenesss}
The categories $\aCat^{\tiny\mbox{nu}}_{\tiny\mbox{strict}}$ and $\aCat_{\tiny\mbox{strict}}$ are complete and cocomplete.
\end{thm}

\begin{proof}
It follows immediately from the existence of product + Theorem \ref{equizz} and the existence of coproducts + Theorem \ref{coeq}.\
Let us note that the final object is the zero algebra $0$ and the initial object is the empty set $\emptyset$.
\end{proof}

\newpage

\section{Semi-free resolutions of $\Ain$categories}\label{hprojec}

In this section we give the notion of semi-free $\Ain$category and semi-free resolution of an $\Ain$category.\ 
The main result of this section is that every strictly unital $\Ain$category has a semi-free resolution.\\
\\
We start recalling a few of definitions in the framework of DG categories.\
The first two are due to Drinfeld (cf.\cite[B.4]{Dri})

\begin{defn}[Semi-free over a DG category]\label{sfreeDg}
Let $\mathcal{A}$ and $\mathcal{K}$ be two DG categories and $\mathsf{E}:\mathcal{K}\to\mathcal{A}$ a DG functor.\
We say that $\mathcal{A}$ is \emph{semi-free} over $\mathcal{K}$ if $\mathcal{A}$ can be represented as the union of an increasing sequence of DG subcategories 
\begin{align*}
\mathcal{A}_0 \subset \mathcal{A}_1 \subset \mathcal{A}_3\subset ... \subset \mathcal{A}
\end{align*}
with the following properties:
\begin{itemize}
\item[1.] Every $\mathcal{A}_i$ has the same objects of $\mathcal{A}$.
\item[2.] The DG functor $\mathsf{E}$ maps isomorphically $\mathcal{K}$ onto $\mathcal{A}_0$.
\item[3.] For every $i>0$, $\mathcal{A}_i$ is freely generated, as a graded $R$-category, 
over $\mathcal{A}_{i-1}$ by a family of homogeneous morphisms $f_{\alpha}$, such that $d f_{\alpha}\in\mathcal{A}_{i-1}$.
\end{itemize}
\end{defn}

\begin{defn}[Semi-free DG category]\label{DGGGG}
A DG category $\mathcal{A}$ is \emph{semi-free} if it is semi-free (according to Definition \ref{sfreeDg}) 
over the discrete category $\mathsf{I}_{\mathcal{A}}$ (see Definition \ref{disco}).
\end{defn}

On the other hand, a definition of \emph{semi-free} $\Ain$category is not a straightforward generalization of Definition \ref{sfreeDg} and \ref{DGGGG}.\ 
To do that we start with the notion of \emph{relatively-free} unital $\Ain$categories due to Lyubashenko and Manzyuk (see \cite[\S 1.7]{LM2}.\ 

\begin{defn}[Relatively free $\Ain$categories]\label{reflfree}
Let $\mathsf{E}:\mathscr{K}\to\A$ be a strict $\Ain$functor between two unital $\Ain$categories.\\
We say that $\A$ is \emph{relatively-free} over $\mathscr{K}$ if $\A$ can be represented as the union of an increasing sequence of its $\Ain$subcategories $\A_j$ and DG quivers $\Q_j$ such that:
\begin{align}\label{sequz}
|{\A}_0| \subset {\Q}_1\subset |{\A}_1|\subset {\Q}_2\subset |{\A}_2|\subset \Q_3\subset... \subset |{\A}|.
\end{align}
All graded quivers have the same objects of $\A$.\ 
Moreover we have:
\begin{itemize}
\item[1.] $\mathsf{E}^0$ is an isomorphism, $\mathsf{E}^1$ is an embedding such that $\mathsf{E}^1(\mathscr{K})=\A_0$.
\item[2.] For every $n>0$ we have an isomorphism of graded (non DG) quivers:
\begin{align}\label{nonzono}
\Q_n\simeq |\A_{n-1}|+ \B_n
\end{align}
for a certain graded quiver $\B_n$ (cf. Definition \ref{sumqui}). 
\item[3.] For every $n>0$ we have an isomorphism
\begin{align*}
\A_n\simeq\TT(\Q_n)/(R_{n})
\end{align*}
where $R_{n}$ is the system of relations (contained in $\TT(\Q_{n})$) given in Example \ref{ex2}, with 
\begin{align*}
i:|\A_{n-1}|\hookrightarrow\TT(\Q_n).
\end{align*}
\end{itemize}
\end{defn}

\begin{defn}[Semi-free $\Ain$categories]\label{sfreeAC}
Let $\mathsf{E}:\mathscr{K}\to\A$ be a strict $\Ain$functor.\ 
We say that $\A$ is \emph{semi-free} over $\mathscr{K}$, if $\A$ is relatively free over $\mathscr{K}$ and the sequence (\ref{sequz})
satisfies the extra hypoteses:
\begin{itemize}
\item[4.] The quotient of DG quivers $\Q_n/|\A_{n-1}|$ has zero differential and the $R$-module $\Hom_{\Q_n/|\A_{n-1}|}(x,y)$ is free, for every $x,y\in\A$.
\end{itemize}
\end{defn}

\begin{rem}
We point out that item 4. of Definition \ref{sfreeAC} corresponds to item 3. of Definition \ref{sfreeDg}.\ 
Suppose that $\A$ is \emph{semi-free} over $\mathscr{K}$.\ Item 4. means that: 
for every $n>0$ and $x,y\in\A$, the hom-space $\Hom_{\B_n}(x,y)$ of the graded quiver $\B_n$ given in (\ref{nonzono}):
has a basis and, given $f\in \Hom_{\B_n}(x,y)$, the differential, induced by (\ref{nonzono}), 
is such that $d_{\Q_n}(f)\subset \Hom_{|\A_{n-1}|}(x,y)$.
\end{rem}

The following definition comes naturally. 
\begin{defn}[Semi-free $\Ain$categories]\label{semifrollo}
A strictly unital $\Ain$category $\A$ is \emph{semi-free} if it is semi-free (in the sense of Definition \ref{sfreeAC}) over the discrete category $\mathsf{I}_{\A}$.
\end{defn}

\begin{defn}[Semi-free resolutions]
We say that a strictly unital $\Ain$category $\A$ (resp. DG category) has a semi-free resolution if, 
there exists a semi-free $\Ain$category $\tilde{\A}$, and a strict $\Ain$functor $\sF:\tilde{\A}\to\A$ (resp. DG category) such that $\sF$ is quasi-equivalence.
\end{defn}

\begin{defn}[H-projective and h-flat resolutions]
We say that a $\star$category $\A$ has a h-projective (resp. h-flat) resolution if, 
there exists a h-projective (resp. h-flat) $\star$category $\tilde{\A}$, and a strict $\star$functor $\sF:\tilde{\A}\to\A$ such that $\sF$ is quasi-equivalence.\\
%\\
Here $\star$ is strictly unital $\Ain$, cohomological unital $\Ain$, unital $\Ain$, non unital $\Ain$ or DG.
\end{defn}

 \subsection{Semi-free resolutions of $\Ain$categories}\label{semifreeee}

It is well known that if ${\A}$ is a DG category, then $\A$ has a semi-free resolution \cite[Lemma B.5]{Dri}.\ 
In this subsection we prove the same result is the framework of $\Ain$categories.\ We 

\begin{thm}\label{semifree}
Every strictly unital $\Ain$category $\A$ admits a semi-free resolution.\\
It means that there exists a strictly unital semi-free $\Ain$category $\A^{\tiny\mbox{sf}}$ 
with a (surjective on the morphisms) strict quasi-equivalence $\Psi:\A^{\tiny\mbox{sf}}\to\A$.
\end{thm}

\begin{proof}[Proof of Theorem \ref{semifree}]
We need to find a sequence of $\Ain$categories:
\begin{align}
\tilde{\A}_0\subset\tilde{\A}_1\subset\tilde{\A}_2\subset...\subset\tilde{\A}_{n}\subset...
\end{align}
satisfying items 1, 2, 3 of Definition \ref{sfreeAC} and item 4 of Definition \ref{semifrollo}.\
Where $\tilde{\A}_0$ is the discrete DG category $\mbox{disc}(\A)$ (see Definition \ref{disco}).\\
We have a strict strictly unital $\Ain$functor, defined as the identity on the objects, and on the morphisms as:
\begin{align*}
\Psi_0:\tilde{\A}_0&\to \A \\
1^{\tilde{\A}_0}_x&\mapsto 1_x.
\end{align*}
Now we want to define the strictly unital $\Ain$category $\tilde{\A}_1$.\\
We take the set $S_1=\mathcal{f}S_{x,y}\mathcal{g}_{x,y\in\tiny\mbox{Ob}(\A)}$ of (sets of) morphisms, where  
\begin{align}\label{homoclo}
S_{x,y}:=&\big\{\mbox{$f_s\in\Hom_{\A}(x,y)$ is homogeneous and closed}\big\}
\end{align}
if $x\not=y$, otherwise
\begin{align}\label{homoclone}
S_{x,x}:=&\big\{\mbox{$f_s\in\Hom_{\A}(x,x)$ is homogeneous such that it is closed or $m_{\A}^1(f_s)=1^{\A}_{x}$}\big\}.
\end{align}
Now we take the graded quiver $\Q_{S_1}$ whose objects are the same of $\A$ and 
whose morphisms are freely generated by the set $S_1$ (namely the set $S_{x,y}$ is a basis for the graded $R$-module $\Hom_{\Q_{S_1}}(x,y)$).\\
We take the graded quiver $\Q_1:=\Q_{S_1}+|\tilde{\A}_0|$, which can be made a DG quiver taking, as differential,
$d_{\Q_{S_1}}(f_s)=1^{\tilde{\A}_0}_{x}$ if $m_{\A}^1(f_s)=1_x^{\A}$ and $d_{\Q_{S_1}}(f_s)=0$ if $m_{\A}^1(f_s)=0$.\\
We have an inclusion of DG quivers $|\tilde{\A}_0|\to \Q_1$, via the adjunction, 
we have a strict non unital $\Ain$functor:
\begin{align}\label{stronk}
\psi:\TT(\Q_1)\to \A.
\end{align}
We take the strictly unital $\Ain$category $\TT(\Q_1)_+$ (see Construction \ref{UNITTTT}), 
we denote by $1_{x}^{\tilde{\A}_1}$ the new unit of the object $x\in\A$.\
We can extend (uniquely) the functor $\psi$ to a strictly unital strict $\Ain$functor $\psi_+:\TT(\Q_1)_{+}\to\A$ (see Remark \ref{primacostr}).\\
Now we can define the strictly unital $\Ain$ category $\TT(\Q_1)_+/ I^1_+$, where $I^1_+$ is the $\Ain$ ideal generated by 
$$R^1_+:=\big((\bullet;1_{x}^{\tilde{\A}_1})-(\bullet; 1_{x}^{\tilde{\A}_0})\mbox{, $\forall x\in\A$} \big).$$
The strict $\Ain$functor $\psi_{+}$ factorizes over $\TT(\Q_1)_+/ I^1_+$.\ 
So we have a strict strictly unital $\Ain$functor:
\begin{align}\label{froggolotto}
\Psi_{S_1}:\TT(\Q_1)_+/ I^1_+\to{\A}.
\end{align}
Clearly the functor $(\ref{froggolotto})$ is surjective on the closed morphisms of ${\A}$.\\
We can choose a subset $\tilde{S}_1\subset S_1$ making (\ref{froggolotto}) surjective on the closed morphisms.\\
\\
We define $\tilde{\A}_1:=\TT(\Q_{\tilde{S}_1}+|\tilde{\A}_0|)_+/ I^1_+$ and $\Psi_1=\Psi_{\tilde{S}_1}$.\\
Clearly the choice of the basis $\tilde{S}_1$ (so the one of $\tilde{\A}_1$) is not unique.
\\
It is defined, as the identity on the objects and, on the morphisms, as
\begin{align*}
\Psi^1_1:\tilde{\A}_1&\to\A\\
\big(\bullet; f\big)&\mapsto f
\end{align*}
and $\Psi_1^{n>1}=0$.\\
We have an inclusion $\tilde{\A}_0\subset\tilde{\A}_1$ of strictly unital $\Ain$categories.\\
\\
Now we want to define $\tilde{\A}_2$.\\

We take the set $S_2=\mathcal{f}S_{x,y}\mathcal{g}_{x,y\in\tiny\mbox{Ob}(\A)}$ of (sets of) morphisms, where $S_{x,y}$ is given by the pairs 
\begin{align*}
S_{x,y}:=&\Big\{\big(a,(\mathfrak{t};f_1,...,f_n)\big)\hspace{0,2cm}\mbox{s.t. $a\in\Hom_{\A}(x,y)$ and $(\mathfrak{t};f_1,...,f_n)\in\Hom_{\tilde{\A}_1}(x,y)$ is homogeneous closed}\\ 
&\mbox{ such that $m_{\A}^1(a)=\Psi_1\big((\mathfrak{t};f_1,...,f_n)\big)$}\Big\}.
\end{align*}
We denote by $\Q_{S_2}$ the graded (non DG) quiver generated by $S_2$.\\%, note that the objects are the same of $\A$.\\
\\
Now we take the (graded) quiver $\Q_2:=\Q_{S_2}+|\tilde{\A}_1|$.\\
We can make $\Q_2$ a DG quiver as follows:\ the differential in $\Q_2$ agrees with the differential of $|\tilde{\A}_1|$ 
(namely $d_{\Q_2}|_{|\tilde{\A}_1|}=d_{|\tilde{\A}_1|}$) and, on the morphisms of $\Q_{S_2}$, is defined as follows:
\begin{align*}
d_{\Q_{2}}\Big(\big(a,(\mathfrak{t};f_1,...,f_n)\big) \Big)&=(\mathfrak{t};f_1,...,f_n)\in|\tilde{\A}_1|\subset\Q_2.
\end{align*}
We have a strictly unital $\Ain$category $\TT(\Q_2)_{+}$, and an inclusion of DG quivers:
\begin{align}\label{inco}
\xymatrix{
|\tilde{\A}_1|\subset|\TT(\Q_2)_{+}|
}
\end{align}
Note that (\ref{inco}) is not an inclusion of $\Ain$categories despite both actually have an $\Ain$structure.\\
We take the $\Ain$ideal $I_2$ of $\TT(\Q_2)$ generated (via the $m^n_{\TT(\Q_2)}$) by the system of relations $R_2$ spanned by the morphisms of the form:
\begin{align}\label{R2}
R_2:=\left(m^n_{\TT(\Q_2)}\big( (\bullet;f_1),...,(\bullet;f_n) \big)-\big(\bullet; m^n_{\tilde{\A}_1}(f_1,...,f_n )\big) \right)&
\end{align}
for every $n\ge1$ and $f_1,...,f_n\in\tilde{\A}_1$.\ See Example \ref{ex2}.\\
We take the non unital $\Ain$category $\TT(\Q_2)/I_2$.\ We have a functor:
\begin{align*}
G_2:\Q_2&\to|\A|\\
\big(\mathfrak{t};f_1,...,f_n\big)&\mapsto \Psi_1((\mathfrak{t};f_1,...,f_n)\big))\\
\big(a,(\mathfrak{t};f_1,...,f_n)\big)&\mapsto a.
\end{align*}
%%%%%%%%%%%%%%%%%%%%%%%%%%%%%%%%%%%%%%%%%%%%%%%%%%%%%%%%%%%%%%%%%%%%%%%%%%%%%%%
It is a map of DG quivers since, $|\Psi_1|:|\tilde{\A}_1|\to|\A|$ is a map of DG quivers, and:
\begin{align*}
G_2(d_{\Q_2}\big(a,(\mathfrak{t};f_1,...,f_n)\big))&=G_2(\mathfrak{t};f_1,...,f_n)\\
&=|\Psi_1|(\mathfrak{t};f_1,...,f_n)\\
&= m_{\A}^1(a)\\
&= m_{\A}^1\big( G_2\big(a,(\mathfrak{t};f_1,...,f_n)\big) \big).
\end{align*}
Via the adjunction (\ref{adjunct}) we have a non unital $\Ain$functor:
\begin{align*}
\beta_{\A}\cdot \TT(G_2):\TT(\Q_2)&\to\A.
\end{align*}
Clearly we have that $\big(\beta_{\A}\cdot \TT(G_2)\big)(R_2)=0$, 
then $\beta_{\A}\cdot\TT(G_2)$ factorizes as follows:
\begin{align}\label{dingfo}
\xymatrix@=1em{
\TT(\Q_2)\ar[rr]^{\beta_{\A}\cdot\TT(G_2)}\ar[dr]_q&&\A \\
&\TT(\Q_2)/I_2\ar@{-->}[ur]&
}
\end{align}
%The strict $\Ain$functor\footnote{this is strict by Lemma \ref{factone} since $\beta_{\A}\cdot\TT(G_2)$ is strict and $q$ is strict and surjective on the morphisms.}:
The dotted strict $\Ain$functor in (\ref{dingfo}) is denoted by
\begin{align}\label{rombolo}
\tilde{\Psi}_2:\TT(\Q_2)/I_2\to&\A
\end{align}
Explicitly 
\begin{align*}
\tilde{\Psi}_2:\TT(\Q_2)/I_2\to&\A\\
(\mathfrak{t};\textbf{f}_n,...,\textbf{f}_1) \mapsto& \mathfrak{t}_{m^{\A}}(G_2(\textbf{f}_n),...,G_2(\textbf{f}_1)).
\end{align*}
Now we take the strictly unital $\Ain$category $\big(\TT(\Q_2)/I_2\big)_+$.\\ 
We denote by $1^{\tilde{\A}_2}_x$ the unit added in this step.\ We can extend $\tilde{\Psi}_2$ to a strict $\Ain$functor:
\begin{align*}
\hat{\Psi}_2: \big(\TT(\Q_2)/I_2\big)_+&\to\A \\
1^{\tilde{\A}_2}_x&\mapsto 1_x.
\end{align*}
We denote by $\tilde{\A}_{2,S}$ the quotient $\big(\TT(\Q_2)/I_2\big)_+/I^2_+$.\ 
Where, as before 
$$I^2_+:=(R^2_+)=\left((\bullet;1_{x}^{\tilde{\A}_2})-(\bullet; 1_{x}^{\tilde{\A}_0})\mbox{, $\forall x\in\A$} \right).$$
Again using Theorem \ref{2.5PropLM2} we have that $\hat{\Psi}_2$ factorize over $I^2_{+}$ since 
\begin{align*}
\hat{\Psi}_2\big( (\bullet;1^{\tilde{\A}_0}_x) \big)=\Psi_1\big((\bullet;1^{\tilde{\A}_0}_x)\big)=1_x.
\end{align*}
We call $\tilde{\Psi}_2$ the strict strictly unital $\Ain$functor 
\begin{align}\label{strokere}
\tilde{\Psi}_2:\tilde{\A}_{2,S} \to\A
\end{align}
The strictly unital $\Ain$category $\tilde{\A}_{2,S}$ and $\tilde{\Psi}_2$ are such that:
\begin{itemize}
\item[1.] $\tilde{\Psi}_2$ is surjective on the morphisms.
\item [2.] Every cycle $\mathfrak{t}\in\tilde{\A}_1$ whose image in $\Psi_1$ is a boundary is a boundary in $\tilde{\A}_{2,S}$.
\end{itemize}
To get $\tilde{\A}_2$ it suffices to take a subset $\tilde{S}_2$ of $S_2$ such that (\ref{strokere}) satisfies 1. and 2.\\
Then, we can define $\tilde{\A}_2:=\tilde{\A}_{2,\tilde{S}}$ and $\Psi_{2}:=\tilde{\Psi}_2$.\\
We have an inclusion of $\Ain$categories:
\begin{align*}
\tilde{\A}_0\subset\tilde{\A}_1\subset\tilde{\A}_2.
\end{align*}
Moreover, by definition, we have
\begin{align*}
\restr{\Psi_2}{\tilde{\A}_1}=\Psi_{1}:\tilde{\A}_1\to&\A.
\end{align*}
Now we can iterate the process to get $\tilde{\A}_n$\\ 
First we define the quiver $\Q_{S_n}$ whose objects are the same of $\A$.\\
The morphisms are generated by the set $S_n$
\begin{align*}
S_n:=&\Big\{\big(a,(\mathfrak{t};f_1,...,f_n)\big)\hspace{0,2cm}\mbox{where $a\in\A$ and $(\mathfrak{t};f_1,...,f_n)\in\tilde{\A}_{n-1}$ is homogeneous closed}\\ 
&\mbox{ such that $m_{\A}^1(a)=\Psi_{n-1}\big((\mathfrak{t};f_1,...,f_n)\big)$}\Big\}.
\end{align*}
As before we can define the DG quiver $\Q_{n}:=|\tilde{\A}_{n-1}|+\Q_{S_n}$.\\
The differential in $\Q_{n}$ is defined on the morphisms of $\Q_{S_n}$ as:
\begin{align*}
d_{\Q_{n-1}}\big(a,(\mathfrak{t};f_1,...,f_n)\big)&=(\mathfrak{t};f_1,...,f_n)\in|\tilde{\A}_{n-1}|\subset\Q_{n}.
\end{align*}
In $\TT(\Q_{n})$ we take the $\Ain$ideal $I_n$ which is generated by the system of relations $R_n$.\ 
Note that $R_n$ is spanned by the morphisms of the form:
\begin{align}\label{relazionerzo}
R_n:=\left(m^n_{\TT(\Q_{n})}\big((\bullet; f_1),...,(\bullet; f_n) \big)-\big(\bullet; m^n_{\tilde{\A}_{n-1}}(f_1,...,f_n )\big) \right).
\end{align}
For every $n\ge1$ and $f_1,...,f_n\in\tilde{\A}_{n-1}$.\\
Proceeding as before, we get a strict strictly unital $\Ain$functor and an $\Ain$category $\tilde{\A}_n$ 
(depending on a subset $\tilde{S}_n\subset S_n$) such that:
\begin{itemize}
\item[1.] $\tilde{\Psi}_n:\tilde{\A}_n\to \A$ is surjective on the morphisms.
\item [2.] Every cycle $\mathfrak{t}\in\tilde{\A}_{n-1}$ whose image in $\Psi_{n-1}$ is a boundary is a boundary in $\tilde{\A}_{n}$.
\end{itemize}
We have an inclusion of strictly unital $\Ain$categories:
\begin{align}\label{incluss}
\tilde{\A}_0\subset\tilde{\A}_1\subset\tilde{\A}_2\subset...\subset\tilde{\A}_{n-1}\subset\tilde{\A}_n.
\end{align}
We have:
\begin{align*}
\restr{\Psi_n}{\tilde{\A}_{n-1}}=\Psi_{n-1}:\tilde{\A}_{n-1}\to&\A.
\end{align*}
Taking the colimit $n\to+\infty$ in $\aCat_{\tiny\mbox{strict}}$, which we know is cocomplete (see Theorem \ref{Cocompletenesss}) we get a strictly unital $\Ain$category ${\A^{\tiny\mbox{sf}}}$ and a strict strictly unital $\Ain$functor:
\begin{align}
{\Psi}:{\A^{\tiny\mbox{sf}}}&\to\A.
\end{align}
which is surjective on the morphisms and a quasi-equivalence of $\Ain$categories.
\end{proof}

\begin{rem}\label{Drinni}
Let $\A$ be a DG category, considering $\A$ as an $\Ain$category we can take the semi-free resolution $\A^{\tiny\mbox{sf}}$ provided by Theorem \ref{semifree}.\ 
We note that $\A^{\tiny\mbox{sf}}$ is not a semi-free DG category (not even a DG category).\
We can easily modify the proof of Theorem \ref{semifree} in order to prove that every DG category has a semi-free resolution.\ 
It suffices to take $\TT_{\tiny{\mbox{DG}}}(\mbox{-})_{+}$ instead of $\TT(\mbox{-})_{+}$ (see Construction \ref{Drinniborguz}).\\
On the other hand, if $\A$ is an $\Ain$category then it is still possible to take the free DG category associated to $\Q$ (where $\Q_S$ is the DG quiver as in (\ref{homoclo})).\ But the induced strict $\Ain$functor $\TT(\Q_S)_{+}\to \mathcal{\A}$ factorizes over the ideal $\mathsf{As}$ only if $\mathcal{\A}$ is associative (namely if $\mathcal{\A}$ is a DG category).\\ 
In Corollary \ref{cor1} we prove that, if $\mathcal{A}$ is a DG category, then the resolution as $\Ain$category and as DG category are weakly equivalent.\ 
\end{rem}
 
\begin{thm}\label{augmo}
Given a strictly unital $\Ain$ (resp. DG)category $\A$, its semi-free resolution ${\A^{\tiny\mbox{sf}}}$ has nice unit (according to Definition \ref{snu}).
\end{thm}

\begin{proof}
We have the inclusions:
\begin{align*}
\tilde{\A}_0\subset\tilde{\A}_1\subset\tilde{\A}_2\subset...\subset \tilde{\A}_{n-1}\subset \tilde{\A}_n\subset...\subset{\A}^{\tiny\mbox{sf}}.
\end{align*}
Fixing an object $x\in\A$ we have the following inclusions of (graded) $R$-modules:
\begin{align*}
0\subset \Hom_{\tilde{\A}_0}(x,x)\subset \Hom_{\tilde{\A}_1}(x,x)\subset \Hom_{\tilde{\A}_2}(x,x)\subset...\subset\Hom_{\A^{\tiny\mbox{sf}}}(x,x).
\end{align*}
Where $\Hom_{\tilde{\A}_0}(x,x):=\Hom_{\tiny\mbox{disc}(\A)}(x,x)\simeq R$.\\
In particular we have the short exact sequence:
\begin{align*}
\xymatrix{
0\to R\cdot 1_x\ar@{^(->}[r]& \Hom_{\tilde{\A}_1}(x,x)\ar@{->>}[r]& \Hom_{\tilde{\A}_1}(x,x)/\Hom_{\tilde{\A}_0}(x,x)\ar[r]&0.
}
\end{align*}
This is splitting since $\Hom_{\tilde{\A}_1}(x,x)/\Hom_{\tilde{\A}_0}(x,x)$ is free.\\
So the inclusion $R\cdot 1_x\to \Hom_{\tilde{\A}_1}(x,x)$ has a retraction.\ On the other hand, all of the quotients (of graded $R$-modules) are free by construction.\ It means that
\begin{align*}
\xymatrix{
0\to R\cdot 1_x\ar@{^(->}[r]& \Hom_{{\A}^{\tiny\mbox{sf}}}(x,x)\ar@{->>}[r]& \Hom_{{\A}^{\tiny\mbox{sf}}}(x,x)/R\cdot 1_x\ar[r]&0.
}
\end{align*}
splits and we are done.
\end{proof}

\newpage

\section{Homotopy Theory of $\aCat$}\label{results}

In this last section we investigate the homotopy category of the (strictly unital) $\Ain$categories.\ 
The key tools are the semi-free resolutions of $\Ain$categories.\ 
The main point is that the semi-free $\Ain$categories behave like the cofibrant objects in model categories.\ 

We introduce the notion of \emph{category with cofibrant morphisms} which is a generalization of h-projective category.\
In particular we have two important properties: the semi-free $\Ain$categories are categories with cofibrant morphisms (Theorem \ref{h-proj}), 
and the semi-free $\Ain$categories have (a weak) lifting property with respect to quasi-equivalences surjective on the morphisms (Theorem \ref{lemlift}).

\subsection{Properties of Semi-free $\Ain$categories}\label{cofibrenzio}
 
We start with two Definitions well known to the experts, we recall that: a (unbounded) DG $R$-module $(M,d_M)$ is \emph{h-projective} (resp. \emph{h-flat}) if, 
for every acyclic complex $N$, the DG $R$-module $\mbox{Hom}(M,N)$ (resp. ${M}\otimes N$) is acyclic.\ 
Note that we use the Drinfeld's notation \emph{h-proj} (standing for \emph{homotopically projective}) 
instead the Spaltenstein's one \emph{k-proj} (see \cite[\S1]{Spa}).

\begin{defn}[Homotopically projective DG categories]
A \emph{h-projective} DG category $\mathcal{A}$ is a DG category such that, for every pair of objects $ x,y\in\mathcal{A}$, 
the DG $R$-module $\big(\Hom_{\mathcal{A}}(x,y), d_{\mathcal{A}}\big)$ is h-projective. 
\end{defn}

\begin{defn}[Homotopically flat DG categories]
A \emph{h-flat} DG category $\mathcal{A}$ is a DG category such that, for every pair of objects $ x,y\in\mathcal{A}$, 
the DG $R$-module $\big(\Hom_{\mathcal{A}}(x,y), d_{\mathcal{A}}\big)$ is a h-flat.
\end{defn}

The definitions of $h$-projective and $h$-flat DG categories can be easily extended in the framework of $\Ain$categories (unital, strictly unital or cohomological unital).\ 
Namely an $\Ain$category $\A$ is \emph{h-projective} (resp. \emph{h-flat}) if the DG module $\big(\Hom_{\A}(x,y), m^1_{\A}\big)$ is h-projective (resp. h-flat), for every pair of objects $ x,y\in\A$.

\begin{rem}\label{cofibrelli}
The category of unbounded DG $R$-modules has a well known model structure \cite[Definition 2.3.3]{Hov}.\ 
The weak equivalences are the quasi-isomorphisms, the fibrations are the morphisms degrewise surjective and the cofibrant objects are the h-projective modules degreewise projectives \cite[Theorem 9.6.1 (ii') iff (v')]{AFH}.\
\end{rem}

\begin{rem}
Note that the tensor product of two cofibrant DG $R$-modules is cofibrant.\ On the other hand if $\C$ and $\C'$ are cofibrant DG categories (in Tabuada's model structure) then $\C\otimes\C'$ is not a cofibrant DG category.
\end{rem}

\begin{defn}[Quivers/Categories with cofibrant morphisms]
Let $\C$ be a DG quiver of an $\Ain$(or DG)category, we say that $\C$ \emph{has cofibrant morphisms} if $\big(\Hom_{\C}(x,y),m^1_{\C}\big)$ is a cofibrant object in the category of DG $R$-modules (according to Remark \ref{cofibrelli}).
\end{defn}

\begin{thm}\label{h-proj}
Every semi-free $\Ain$category has cofibrant morphisms.
\end{thm}

To prove Theorem \ref{h-proj}, we need a few Lemmas.\ 
The first is a particular case of \cite[9.3.4 Lemma (2)]{AFH} when $R$ is a commutative ring (not DG).

\begin{lem}\label{lucif}
If $M$ is a DG $R$-module, such that $M=\displaystyle\cup_{u>0}M^u$, for a sequence of DG $R$-submodules $M^{u-1}\subset M^u$, such that $C^u:=M^u/M^{u-1}$ is cofibrant, then $M$ is cofibrant.
\end{lem}

\begin{lem}\label{freenzo}
If $\Q$ is a DG quiver with cofibrant morphisms then $\TT(\Q)$ is an $\Ain$category with cofibrant morphisms.
\end{lem}

\begin{proof}[Proof of Lemma \ref{freenzo}]
Without loss of generality, we can suppose that $\Q$ has only one object (i.e. $\Q$ is a cofibrant DG $R$-module).\\ 
Fixed $m>1$ we consider the graded $R$-module: 
\begin{align}\label{sequenzonello}
U_{m,l}=PT^0_1(\Q) \oplus \displaystyle\bigoplus^{m-1}_{\tilde{m}=2}\displaystyle\bigoplus^{\tilde{m}-1}_{\tilde{l}=1}PT^{\tilde{l}}_{\tilde{m}}(\Q) \oplus \displaystyle\bigoplus^l_{\tilde{l}=1}PT^{\tilde{l}}_m(\Q)
\end{align} 
Where $U_{1,0}=PT^0_1(\Q)$.\\ 
As we pointed out in \ref{notello}, given $T\in PT^l_m(\Q)$ and $l>1$, then $\mTT^1(T)\not\subset PT^l_m(\Q)$.\\ 
On the other hand, if $T\in U_{m,l}$ then $\mTT^1(T)\subset U_{m,l}$.\ So $U_{m,l}$ is a DG $R$-module.\\

Clearly if $(m,l)<(m',l')$ then $U_{m,l}\subset U_{m',l'}$.\ We have the following sequence of DG $R$-modules:
\begin{align}\label{sequenzone}
0\subset U_{1,0} \subset U_{2,1} \subset U_{3,1} \subset U_{3,2} \subset ... \subset U_{m,l} \subset ... \subset \TT(\Q). 
\end{align}
Where $m>1$ and $1\le l \le m-1$.\\ 
Now we want to prove that the sequence (\ref{sequenzone}) fulfills the hypothesis of Lemma \ref{lucif}.\\
First we note that $U_{1,0}\simeq\Q$ so it is cofibrant.\
We fix $m>1$:
\begin{itemize}
\item[i.] If $1\le l < m-1$, then as graded $R$-modules, we have an isomorphism:
\begin{align}
U_{m,l+1}/U_{m,l}\simeq PT^{l+1}_m(\Q).
\end{align}
As DG $R$-module $U_{m,l+1}/U_{m,l}$ is
\begin{align}
U_{m,l+1}/U_{m,l}\simeq \displaystyle\bigoplus_{PT^{l+1}_m}(\Q,d_{\Q})\otimes ...\otimes(\Q,d_{\Q}).
\end{align}
i.e. a finite number of copies of $(\Q,d_{\Q})^{\otimes m}$.
\item[ii.] If $l=m-1$, then as graded $R$-modules, we have an isomorphism:
\begin{align}
U_{m+1,1}/U_{m,m-1}\simeq PT^{1}_{m+1}(\Q).
\end{align}
As DG $R$-module $U_{m+1,1}/U_{m,m-1}$ is equal to $(\Q,d_{\Q})^{\otimes (m+1)}$.
\end{itemize}
So (\ref{sequenzone}) fulfills the hypothesis of Lemma \ref{lucif} and we are done.
\end{proof}

\begin{cor}\label{coquive}
If $\Q$ is a DG quiver with cofibrant morphisms, then the categories $\TT(\Q)_{+}$, $\TT_{DG}(\Q)$ and $\TT_{DG}(\Q)_{+}$ are categories with cofibrant morphisms.
\end{cor}

\begin{proof}
The proof is the same of Theorem \ref{freenzo} changing opportunely the filtration (\ref{sequenzone}).\\ 
For example, to prove that $\TT_{DG}(\Q)$ has cofibrant objects one has to take the filtration:
\begin{align*}
0\subset U_{1,0} \subset U_{2} \subset U_{3} \subset ... \subset U_{m} \subset ... \subset \TT_{DG}(\Q),
\end{align*}
where
%\begin{align*}
$U_m:=PT^0_1(\Q)\oplus \displaystyle\bigoplus^{m}_{\tilde{m}=2}PT^{1}_{\tilde{m}}(\Q).$
%\end{align*}
\end{proof}

We need to introduce a new notation in the same spirit of \ref{notello}.\

\begin{notatu}\label{notellone}
Let $\Q_1$ and $\Q_2$ two graded quivers with the same objects.\
Fixed two objects $x_1,x_2\in\Q_1$, we denoted by
\begin{align*}
PT^l_m(\Q_1^{\otimes j},\Q_2^{\otimes(n-j)})(x,y):=\displaystyle\bigoplus_{\substack{\mathfrak{t}\in PT^l_m \\ x=x_1,...,x_{m+1}=y}} \ 
\displaystyle\bigoplus_{\substack{ {i_k}\in\mathcal{f}1,2\mathcal{g}\\ \sharp \mathcal{f} i_k =1\mathcal{g} = j} }
\Hom_{\Q_{i_1}}(x_1,x_{2}) \otimes ...\otimes \Hom_{\Q_{i_m}}(x_{m},x_{m+1}).
\end{align*}
Namely, if $\Q_1,\Q_2$ are $R$-module the 
Where ${i_k}\in\mathcal{f}1,2\mathcal{g}$, for $1\ge k\ge n$, and $\sharp \mathcal{f}i_k \mbox{ s.t. } i_k =1\mathcal{g} = j$.\\
We denote:
\begin{align*}
PT^l_m(\Q_1^{\otimes j},\Q_2^{\otimes(m-j)}):=\displaystyle\bigoplus_{x_1,x_{n+1}\in\Q_1} PT^l_m(\Q_1^{\otimes j},\Q_2^{\otimes(m-j)})(x_1,x_{n+1}).
\end{align*}
\end{notatu}

\begin{proof}[Proof of Theorem \ref{h-proj}]
Without loss of generality we can suppose that $\A$ has only one object.\ 
We have a sequence of $\Ain$algebras (as in Theorem \ref{semifree}).\
\begin{align*}
\tilde{\A}_0\subset\tilde{\A}_1\subset\tilde{\A}_2\subset...\subset \tilde{\A}_{n-1}\subset \tilde{\A}_n\subset...\subset{\A}^{\tiny\mbox{sf}}.
\end{align*}
We prove by induction that, for every positive integer $n$, the DG $R$-module $(\tilde{\A}_{n},m^1_{\tilde{\A}_n})$ is cofibrant.\\
\\
First we prove that $\tilde{\A}_1$ has cofibrant morphisms:\ If $\A$ is not cohomologically trivial then $\A$ has cofibrant morphisms since $\tilde{\A}_1\simeq\TT(\Q)_{+}$ and $\Q$ is a free $R$-module, see (\ref{homoclone}).\\
If $\A$ is cohomologically trivial then, there exists $h\in\A$ such that $m^1_{\A}(h)=1$.\
We have:
\begin{align*}
\tilde{\A}_1:=\TT(\Q_1)_{+}/(1_{+}-1_{\tilde{\A}_0}).
\end{align*}
Note that $\Q_1:=\Q_S+R\cdot h + R\cdot 1_{\tilde{\A}_0}$ is a DG $R$-quiver such that $d_{\Q_1}(h)=1_{\tilde{\A}_0}$.\\
With an abuse of notation we denote by $h$ the graded $R$-quiver with one object such that $h\simeq R$.\\
We have a filtration of DG $R$-modules: 
\begin{align*}
0\subset U_{1,0,0}\subset U_{1,0,1} \subset U_{2,1,0} \subset U_{2,1,1} \subset U_{2,1,2} \subset U_{3,1,0} \subset ... \subset U_{m,l,j} \subset ... \subset |\tilde{\A}_1|.
\end{align*}
Where $U_{1,0,0}:=PT^0_1(\Q_S)$, $U_{1,0,1} = U_{1,0,0} \oplus PT^0_1(h)$, and
\begin{align*}
U_{m,l,j}& :=U_{1,0,1}\oplus \displaystyle\bigoplus_{\tilde{m}=2}^{m-1}\displaystyle\bigoplus_{\tilde{l}=1}^{\tilde{m}-1}\displaystyle\bigoplus_{\tilde{j}=0}^{\tilde{m}}PT_{\tilde{m}}^{\tilde{l}}\big(\Q_S^{\otimes\tilde{j}},h^{\otimes(\tilde{m}-\tilde{j})}\big) 
\oplus \displaystyle\bigoplus_{\tilde{l}=1}^{l-1}\displaystyle\bigoplus_{\tilde{j}=0}^{{m}}PT_{{m}}^{\tilde{l}}\big(\Q_S^{\otimes\tilde{j}},h^{\otimes(\tilde{m}-\tilde{j})}\big) 
\oplus 
\displaystyle\bigoplus_{\tilde{j}=0}^{j}PT_{m}^{l}\big(\Q_S^{\otimes\tilde{j}},h^{\otimes(\tilde{m}-\tilde{j})}\big),
\end{align*}
for $m>1$, $1\le l\le m-1$ and $0\le j\le m$.\\
Clearly $U_{1,0,0}\simeq (\Q_S,0)$ and $U_{1,0,1}/U_{1,0,0}\simeq (R,0)$ is cofibrant.\
In general, we have:
\begin{align}\label{qui}
U_{m,1,0}/U_{m-1,1,m-1}\simeq \displaystyle\bigoplus_{PT^1_m}(\Q_S,0)^{\otimes m}.
\end{align}
If $l>1$:
\begin{align}\label{quo}
U_{m,l,0}/U_{m-1,l-1,m-1}\simeq \displaystyle\bigoplus_{PT^l_m}(\Q_S,0)^{\otimes m}.
\end{align}
If $j>1$ then
\begin{align}\label{quenzo}
U_{m,l,j}/U_{m,l,j-1}\simeq \displaystyle\bigoplus_{N\le\sharp(PT^l_m)}(R,0)^{\otimes j}\otimes (\Q_S,0)^{\otimes (m-j)}.
\end{align}
The quotients (\ref{qui}), (\ref{quo}) and (\ref{quenzo}) are cofibrant since they are a finite number of free $R$-modules.\
By Lemma \ref{freenzo} we have that $\tilde{\A}_1$ has cofibrant morphisms.\\
\\
Now suppose that $(\tilde{\A}_{n-1},m^1_{\tilde{\A}_{n-1}})$ is cofibrant, we prove that $(\tilde{\A}_{n},m^1_{\tilde{\A}_n})$ is cofibrant.\\
\\
As before, we want to find a filtration for $|\tilde{\A}_n|$ fulfilling the hypotheses of Lemma \ref{lucif}.\\
First we define a filtration for the DG $R$-module $|\TT(\Q_n)|$, where $\Q_n:=\Q_{S_n}+\tA$.\ We have:
\begin{itemize}
\item For $m=1$, $U_{1,0,0}:=PT^0_1(\tA)$ and $U_{1,0,1} := U_{1,0,0} \oplus PT^0_1(\Qs)$.
\item For $m>1$, $1\le l\le m-1$ and $0\le j\le m$, we have:
\begin{align*}
U_{m,l,j}& :=U_{1,0,1}\oplus \displaystyle\bigoplus_{\tilde{m}=2}^{m-1}\displaystyle\bigoplus_{\tilde{l}=1}^{\tilde{m}-1}\displaystyle\bigoplus_{\tilde{j}=0}^{\tilde{m}}PT_{\tilde{m}}^{\tilde{l}}\big(\Qs^{\otimes\tilde{j}},\tA^{\otimes(\tilde{m}-\tilde{j})}\big) \\
&\oplus \displaystyle\bigoplus_{\tilde{l}=1}^{l-1}\displaystyle\bigoplus_{\tilde{j}=0}^{{m}}PT_{{m}}^{\tilde{l}}\big(\Qs^{\otimes\tilde{j}},\tA^{\otimes(\tilde{m}-\tilde{j})}\big) \\ 
& \oplus 
\displaystyle\bigoplus_{\tilde{j}=0}^{j}PT_{m}^{l}\big(\Qs^{\otimes\tilde{j}},\tA^{\otimes(\tilde{m}-\tilde{j})}\big) 
\end{align*}
\end{itemize}
According to \ref{notellone}, the $U_{m,l,j}$ defined above are graded $R$-modules, they can be made DG taking the differential of ${\TT(\Q_n)}$.\\
It is clear that if $(m,l,j)<(m',l',j')$ then $U_{m,l,j}\subset U_{m',l',j'}$, so we have the following sequence of DG $R$-modules:
\begin{align*}
0\subset U_{1,0,0}\subset U_{1,0,1} \subset U_{2,1,0} \subset U_{2,1,1} \subset U_{2,1,2} \subset U_{3,1,0} \subset ... \subset U_{m,l,j} \subset ... \subset |\TT(\Q_n)|.
\end{align*}
As a DG $R$-module, we have:
$|\tilde{\A}_n|\simeq |\TT(\Q_n)|/I_n$, where $I_n$ is the DG $R$-module given by the system of relations $R_n$, see (\ref{relazionerzo}).\\
We have an induced filtration:
\begin{align}\label{sequewnza}
0\subset \overline{U}_{1,0,1}\subset \overline{U}_{2,1,2} \subset \overline{U}_{3,1,1} &\subset ...\subset \overline{U}_{m-1,m-2,m-1}\subset \overline{U}_{m,1,1} \subset ...\\
&...\subset \overline{U}_{m,l-1,m}\subset \overline{U}_{m,l,1}\subset ...\subset \overline{U}_{m,l,j-1}\subset \overline{U}_{m,l,j} \subset ... \subset |\tilde{\A}_n|.
\end{align}
Note that, $\overline{U}_{1,0,0}=U_{1,0,0}\simeq|\tilde{\A}_{n-1}|$, $\overline{U}_{1,0,1}=U_{1,0,1}$ and $\overline{U}_{m,l,0}=\overline{U}_{m,l-1,m}$, if $l> 1$.\\
We have the following isomorphisms of DG $R$-modules:
\begin{itemize}
\item If $j>1$ then:
\begin{align*}
\overline{U}_{m,l,j}/\overline{U}_{m,l,j-1}\simeq \displaystyle\bigoplus_{N \le \sharp(PT^l_m)}(\Q,0)^{\otimes j}\otimes (\tilde{\A}_{n-1},m^1_{\tilde{\A}_{n-1}})^{\otimes (m-j)}.
\end{align*}
\item If $j=1$ then: 
\begin{itemize}
\item if $l=1$:
\begin{align*}
\overline{U}_{m,1,1}/\overline{U}_{m-1,m-2,m-1}\simeq \displaystyle\bigoplus_{N\le \sharp(PT^1_m)}(\Q,0)\otimes (\tilde{\A}_{n-1},m^1_{\tilde{\A}_{n-1}})^{\otimes (m-1)},
\end{align*}
\item
if $l>1$:
\begin{align*}
\overline{U}_{m,l,1}/\overline{U}_{m,l-1,m}\simeq \displaystyle\bigoplus_{N\le\sharp(PT^l_m)}(\Q,0)\otimes (\tilde{\A}_{n-1},m^1_{\tilde{\A}_{n-1}})^{\otimes (m-1)}.
\end{align*}
\end{itemize}
\end{itemize}
Since $\overline{U}_{1,0,1}$ is cofibrant, and every quotient in (\ref{sequewnza}) is cofibrant by Lemma \ref{lucif} we are done.
\end{proof}

\begin{rem}\label{DGSEM}
We can use the same strategy of Theorem \ref{h-proj} to prove that semi-free DG categories have cofibrant morphisms.\ 
One has to change opportunely the filtration as in Corollary \ref{coquive}.
\end{rem}

We recall that an $\Ain$category $\A$ is semi-free (cf. Definition \ref{semifrollo}) if it is relatively free over the discrete category $\mathsf{I}_{\A}$.\
It means that we need a (strict) strictly unital $\Ain$functor
\begin{align*}
\mathsf{E}:\mathsf{I}_{\A}\to\A.
\end{align*}
If $\A$ is a non unital $\Ain$category the definition of \emph{semi-free} does not make sense.\\
On the other hand the definition of $\Ain$category with cofibrant morphisms makes sense even in the case of non unital $\Ain$category.\
We have the following result: 

\begin{thm}\label{cofmor}
Given a non unital $\Ain$category $\A$ there exists a non unital $\Ain$category $\A^{\tiny\mbox{cm}}$ with cofibrant morphisms and a strict $\Ain$functor $\Psi:{\A}^{\tiny\mbox{cm}}\to\A$.\ 
The category $\A^{\tiny\mbox{cm}}$ has the same objects of $\A$ and $\Psi$ is a quasi-equivalence, surjective on the morphisms.
\end{thm}

\begin{proof}
To construct the $\Ain$category $\A^{\tiny\mbox{cm}}$ we proceed as in the proof of Theorem \ref{freenzo}.\\ 
We need to find a sequence of (non unital) $\Ain$category: 
\begin{align}
\tilde{\A}_0\subset\tilde{\A}_1\subset\tilde{\A}_2\subset...\subset\tilde{\A}_{n}\subset...  
\end{align}
We define $\tilde{\A}_0=0$.\ 
The non unital $\Ain$category $\tilde{\A}_1$ is the free $\Ain$category $\TT(\Q_{\tilde{S}_1})$.\ 
Here $\Q_{S_1}$ is the free graded quiver generated by the set 
$S_1=\mathcal{f}S_{x,y}\mathcal{g}_{x,y\in\tiny\mbox{Ob}(\A)}$ of (sets of) morphisms, where  
\begin{align}\label{homoclo}
S_{x,y}:=&\big\{\mbox{$f_s\in\Hom_{\A}(x,y)$ is homogeneous and closed}\big\}
\end{align}
for every $x,y\in\A$.\ 
$\tilde{S}_1$ is a subset of $S_1$ making the induced $\Ain$functor $\Psi_1:\TT(\Q_{\tilde{S}_1})\to \A$ surjective on the closed morphisms of $\A$.\\
Clearly if we want a functorial construction we need to take $\tilde{S}_1=S_1$ (see Remark \ref{functororne}).\\ 
The non unital $\Ain$category is defined as $\A^{\tiny\mbox{cm}}_n:=\TT(\Q_n)/I_n$, see (\ref{relazionerzo}).\\ 
Taking the colimit (in the category of non unital $\Ain$categories with strict $\Ain$functors) of the sequence above, we get a non unital $\Ain$category $\A^{\tiny\mbox{cm}}$ and a quasi-equivalence $\Psi:\A^{\tiny\mbox{cm}}\to\A$.\
To prove that $\A^{\tiny\mbox{cm}}$ has cofibrant morphisms one can easily use the same the proof of Theorem \ref{h-proj} \emph{mutatis mutandis}.
\end{proof}

\begin{rem}\label{gorto}
In the same vein of Remark \ref{Drinni}, it is easy to prove that, given a non unital DG category $\C$ we can find a non unital DG category with cofibrant morphisms with the same properties of the one in Theorem \ref{cofmor}.\ To prove it is suffices to take $\TT_{DG}$ instead of $\TT$ in the proof of Theorem \ref{cofmor}.\ 
\end{rem}

\begin{rem}\label{crachen}
We recall that every {h-projective DG module} is {h-flat}.\ It implies:  
\emph{(semi-free $\Ain$(resp. DG)category} $\Rightarrow)$ \emph{$\Ain$(resp. DG)category with cofibrant morphisms} $\Rightarrow$ \emph{h-projective $\Ain$(resp. DG)category} $\Rightarrow$ \emph{h-flat $\Ain$(resp. DG)category}.
%To prove namely \cite[]{Toe}.
\end{rem}

\newpage

\subsection{Lifting property}

We proved that semi-free $\Ain$categories have cofibrant morphisms, now we prove that they have a lifting property \`a la Drinfeld \cite[Lemma 13.6]{Dri}.\ 

\begin{lem}\label{lemlift}
Let $\A$ be a semi-free $\Ain$category and $\B$, $\C$ two strictly unital $\Ain$categories.\\
Given $\mathsf{F}:\A\to\B$ a strict strictly unital $\Ain$functor and $\mathsf{G}:\C\to\B$ a strict quasi-equivalence, surjective on the morphisms, there exists a (non unique) strict $\Ain$functor $\tilde{\mathsf{F}}:\A\to\C$ making the diagram
\[
\xymatrix{
&{\C}\ar@{->>}[d]_{\simeq}^{\mathsf{G}}\\
\A\ar@{-->}[ur]^{\tilde{\sF}}\ar[r]^{\sF}&\B
}
\]
commutative in $\aCat_{\tiny\mbox{strict}}$.
\end{lem}

\begin{rem}\label{remarzo}
As Drinfeld pointed out in proof of \cite[Lemma 13.6]{Dri}, if $g^{\bullet}:X^{\bullet}\to Y^{\bullet}$ is a surjective quasi-isomorphism of complexes then, if $g^{\bullet}(x)=d_{Y^{\bullet}}(y)$ where $x\in X^{\bullet}$, $y\in Y^{\bullet}$ and $d_{\tiny X^{\bullet}}(x)=0$ there is an $x'\in X^{\bullet}$ such that $g^{\bullet}(x')=y$ and $x=d_{X^{\bullet}}(x')$.
\end{rem}

\begin{proof}[Proof of Lemma \ref{lemlift}]
First we can suppose that $\A$ admits a finite filtration:
\begin{align*}
\tilde{\A}_0\subset\tilde{\A}_1\subset\tilde{\A}_2\subset...\subset\tilde{\A}_{n-1}\subset\tilde{\A}_n=\A.
\end{align*}
We will prove it by induction.\\
Step 1: we recall that $\tilde{\A}_1:=\TT(\Q_1)_+/I_0$.\ 
Where $\Q_1$ are the closed.\\
So we have a strict non unital $\Ain$functor $p:\TT(\Q_1+\mathsf{I}_{\A})\to\tilde{\A}_1$, and a diagram
\[
\xymatrix{
&&{\C}\ar@{->>}[d]_{\simeq}^{\mathsf{G}}\\
\TT(\Q_1+\mathsf{I}_{\A})\ar[r]_-{p}&\tilde{\A}_1\ar[r]_{\sF|_{\tilde{\A}_1}}&\B
}
\]
Since $\TT(\Q_1+\mathsf{I}_{\A})$ is free then, by Theorem \ref{2.3PropLM1}, $\mathsf{F}|_{\tilde{\A}_1}\cdot p$ is given by 
$\big( (\mathsf{F}|_{\tilde{\A}_1}\cdot p)^1,0,...\big)$.\\ 
Where $(\mathsf{F}\cdot p)^1$ is defined on the DG quiver $\Q_1+\mathsf{I}_{\A}$.\\ 
To define a strict non unital $\Ain$functor $\TT(\Q_1+\mathsf{I}_{\A})\to\C$ it suffices to define a morphisms of DG quivers 
$$\tilde{\mathsf{G}}:\Q_1+\mathsf{I}_{\A}\to|\C|.$$ 
Given $f:x\to y\in\Q_1+\mathsf{I}_{\A}$ then $(\mathsf{F}\cdot p)_1(f)\in\Hom_{\B}\big((\mathsf{F}\cdot p)^0(x), (\mathsf{F}\cdot p)^0(y)\big)$.\\
By the axiom of choice we can take a morphism $g\in\C$ such that $\mathsf{G}(g)=(\mathsf{F}\cdot p)^1(f)$.\\
On the other hand, since $\mathsf{G}$ is a (strict) strictly unital $\Ain$functor then $\mathsf{G}(1_x)=1_{\mathsf{G}_0(x)}$.\\
We can define a morphism of quivers.
\begin{align*}
\tilde{\mathsf{G}}(f):\Q_1+\mathsf{I}_{\A}&\to |\B|\\
f+1_x&\mapsto g+1_{\mathsf{G}_0(x)}.
\end{align*}
Since $\B$ is strictly unital then the strict non unital $\Ain$functor $\beta_{\C}\cdot\TT(\tilde{\mathsf{G}})$ factorizes over the ideals $I_0$ and $I_+$, where $I_+$ is the ideal identifying the units of $\tilde{\A}_0$ and $\tilde{\A}_1$.\ 
So we get a strict strictly unital $\Ain$functor $\tilde{\sF}|_{\tilde{\A}_1}$ fitting:
\[
\xymatrix{
&{\C}\ar@{->>}[d]_{\simeq}^{\mathsf{G}}\\
\tilde{\A}_1\ar@{-->}[ur]^{\exists \tilde{\sF}|_{\tilde{\A}_1}}\ar[r]_{\sF|_{\tilde{\A}_1}}&\B.
}
\]
We suppose that for $k$ we have a strict $\Ain$functor:
\begin{align*}
\tilde{\sF}|_{\tilde{\A}_{k}}:\tilde{\A}_k\to\C \mbox{   such that   } \sG\cdot\tilde{\sF}|_{\tilde{\A}_{k}} ={\sF}|_{\tilde{\A}_{k}}.
\end{align*}
Step $k+1$:\ we have the commutative diagram:
\begin{align}\label{ronzo}
\xymatrix@=3em{
\A_k\ar@{^(->}[d]\ar[r]^-{\tilde{\sF}|_{\tilde{\A}_{k}}}&{\C}\ar@{->>}[d]_{\simeq}^-{\mathsf{G}}\\
\A_{k+1}\ar[r]_-{\mathsf{F}|_{\tilde{\A}_{k+1}}}&\B
}
\end{align}
We want to define the strict $\Ain$functor:
\begin{align*}
\tilde{\sF}|_{\tilde{\A}_{k+1}}:\tilde{\A}_{k+1}\to \C.
\end{align*}
Since it is free over $\tilde{\A}_{k}$ then we can extend the diagram (\ref{ronzo}) as follows:
\begin{align*}
\xymatrix@=3em{
&\A_k\ar@{^(->}[d]\ar[r]^-{\tilde{\sF}|_{\tilde{\A}_{k}}}&{\C}\ar@{->>}[d]_{\simeq}^-{\mathsf{G}}\\
\TT\big(\Q_{k+1}+|\tilde{\A}_k|+\mathsf{I}_{\A}\big)\ar[r]_-p&\A_{k+1}\ar[r]_-{\mathsf{F}|_{\tilde{\A}_{k+1}}}&\B
}
\end{align*}
As before we define the morphism on the DG quivers:
\begin{align*}
\tilde{\mathsf{G}}_{k+1}:\TT\big(\Q_{k+1}+|\tilde{\A}_k|+\mathsf{I}_{\A}\big)&\to |\C|.
\end{align*}
The only thing to define is $\tilde{\sG}_{k+1}\big( (a,\mathfrak{t})\big)$ where $(a;\mathfrak{t})\in \Q_k$.\\
We have $(\sF|_{\tilde{\A}_{k+1}}\cdot p)|_{\Q_{k+1}+|\tilde{\A}_k|+\mathsf{I}_{\A}}\big( (\bullet; (a;\mathfrak{t}) )\big)\in\B$ such that 
$$d_{\B}\left( (\sF|_{\Q_{k+1}+|\tilde{\A}_k|+\tiny \mathsf{I}_{\A}}\cdot p)\big( (\bullet; (a;\mathfrak{t}))\big) \right)= {\sF}|_{\tilde{\A}_{k}}( \mathfrak{t})=\sG\cdot\tilde{\sF}|_{\tilde{\A}_{k}}(\mathfrak{t}).$$
We can use Remark \ref{remarzo} to define $\tilde{\sG}_{k+1}\big( (a,\mathfrak{t})\big)$.\ 
We can procede as before to get a strict strictly unital $\Ain$functor $\tilde{\sF}|_{\tilde{\A}_n}:{\tilde{\A}_n}\to \C$.\\
So we get a strict strictly unital $\Ain$functor $\tilde{\sF}$ fitting:
\begin{align*}
\xymatrix@=3em{
&{\C}\ar@{->>}[d]_{\simeq}^-{\mathsf{G}}\\
\tilde{\A}\ar@{-->}[ur]^{\tilde{\sF}}\ar[r]_-{\mathsf{F}}&\B
}
\end{align*}
and we are done.
\end{proof}

\begin{thm}\label{Invertone}
Let $\A$, $\B$ and $\C$ be three strictly non unital $\Ain$categories.\\
Given $\mathsf{F}:\A^{\tiny\mbox{cm}}\to\B$ and $\mathsf{G}:\C\to\B$ two strict $\Ain$functors, such that $\mathsf{G}$ is a quasi-equivalence surjective on the morphisms.\
There exists a (non unique) strict $\Ain$functor $\tilde{\mathsf{F}}:\A^{\tiny\mbox{cm}}\to\C$ making the diagram
\[
\xymatrix{
&{\C}\ar@{->>}[d]_{\simeq}^{\mathsf{G}}\\
\A^{\tiny\mbox{cm}}\ar@{-->}[ur]^{\tilde{\mathsf{F}}}\ar[r]^{\mathsf{F}}&\B
}
\]
commutative in $\aCat^{\tiny\mbox{nu}}$.\ Here $\A^{\tiny\mbox{cm}}$ is the one in Theorem \ref{cofmor}.
\end{thm}

\begin{proof}
The proof is the same of Theorem \ref{lemlift}.
\end{proof}

\begin{rem}
Note that not all the categories with cofibrant morphisms have the lifting property of Theorem \ref{Invertone}, only the one of the form $\A^{\tiny\mbox{cm}}$.\
\end{rem}

\begin{rem}\label{functororne}
The categories found in Theorem \ref{semifree} and in Theorem \ref{cofmor} are not functorial in $\aCat_{\tiny\mbox{strict}}$ 
since they depend on the choice of the sets $\tilde{S}_n$, for every $n\ge1$.\ It is possible to make them functorial taking $\tilde{S}_n=S_n$, for every $n>0$.\ 
In this way, taking a strict $\Ain$functor $\mathsf{F}:\A\to\B$, the "lifted" functor $\tilde{\sF\cdot\Psi_{\A}}$ fitting the diagram:
\begin{align}
\xymatrix{
\A^{\tiny\mbox{sf}/\mbox{cm}}\ar[d]^{\Psi_{\A}}\ar@{-->}[r]^{\tilde{\sF\cdot\Psi_{\A}}}&\ar[d]^{\Psi_{\B}}\B^{\tiny\mbox{sf}/\mbox{cm}}\\
\A\ar[r]^{\sF}&\B\\
}
\end{align}
given by Theorem \ref{lemlift} (resp. Theorem \ref{Invertone}), is unique.\ We denote by $\sF^{\tiny\mbox{sf}/\mbox{cm}}$ the unique $\Ain$functor $\tilde{\sF\cdot\Psi_{\A}}$.\
This is the same strategy that Canonaco, Neeman and Stellari used in \cite[\S 3.2]{CNS} to make Drinfeld's construction \cite[Lemma B.5]{Dri} functorial.\ 
\end{rem}

We conclude with some easy corollaries: 
\begin{cor}
Every (non unital) $\Ain$(or DG)category has a (functorial in $\aCat_{\tiny\mbox{strict}}$) h-projective resolution.\
\end{cor}

\begin{proof}
This is an immediate consequence of Theorem \ref{cofmor}, \ref{cofmor}, and Remark \ref{gorto}, \ref{crachen}, \ref{functororne}.
\end{proof}

\begin{rem}
If $\A$ is cohomological unital $\Ain$ (resp. DG)category then the category $\A^{\tiny\mbox{cm}}$ must be cohomological unital since they are a quasi-equivalent.\
On the hand $\A^{\tiny\mbox{cm}}$ is h-projective, it follows that $\A^{\tiny\mbox{cm}}$ is unital (see \cite[Lemma 3.13]{COS2}).\ 
\end{rem}

\begin{cor}\label{cor1}
Let $\A$ be a DG category and $\tilde{\A}$ its DG semi-free resolution (cf. \cite[Lemma 13.5]{Dri}).\
There exist two quasi-equivalences $\Phi:\A^{\tiny\mbox{sf}}\to\tilde{\A}$ and $\Phi':\tilde{\A}\to\A^{\tiny\mbox{sf}}$ such that 
$$\Phi\cdot\Phi'\approx\mbox{Id}_{\tilde{\A}}$$ and $$\Phi'\cdot\Phi\approx\mbox{Id}_{{\A}^{\tiny\mbox{sf}}}.$$
\end{cor}
 
\begin{proof}
It follows directly from Homological Perturbation Theory (see subsection \ref{HPT}) noting that $\A^{\tiny\mbox{sf}}$ and $\tilde{\A}$ are both h-projective 
(see Theorem \ref{h-proj} and Remark \ref{gorto}).
\end{proof}

\begin{rem}\label{TOnzo}
One can prove that a semi-free DG category is h-projective by using the model structure of DGCat \cite[Proposition 2.3(3)]{Toe}, 
since semi-free DG categories are cofibrant in such a model structure.\
\end{rem}

\begin{rem}
In \cite{Orn4} we will prove that the category $\aCat_{\tiny\mbox{strict}}$ (resp. $\aCat^{\tiny\mbox{nu}}_{\tiny\mbox{strict}}$) has a model structure, 
whose standard cofibrations are exactly $\A^{\tiny\sem}$ (resp. $\aCat^{\tiny\mbox{nu}}_{\tiny\mbox{strict}}$).\ 
Note that in the case of non unital $\Ain$algebras the existence of a model structure follows (more or less) from \cite[2.2.1. Theorem]{Hin}.\\ 
It important to say that, if the base ring is a field then the situation change completely.\ 
Let us denote by $\mbox{Alg$_{\infty}$}$ the category of non unital $\Ain$algebras over a field.\ 
Lefevre-Hasegawa proved that $\mbox{Alg$_{\infty}$}$ has a model structure without limits (\cite[Theorem 1.3.3.1]{LH}).\ In such a model structure the weak equivalences are the quasi-isomorphisms, the fibrations (resp. cofibrations) are the $\Ain$functors $f$, such that $f_1$ is an epimorphism (resp. monomorphism).\ 
Note that, in this case, every $\Ain$algebra is cofibrant and fibrant.\   
To conclude we say that the category of $\Ain$categories (with any kind of unit), over a field, has a structure of fibrant object in the category of Relative Categories (\cite{Orn5}).

%On the other hand, in \cite{Orn5}, adopting the same techniques, it is proven that the category of non-unital (unital/ strictly unital) 
%$\Ain$categories is a fibrant category (see \cite[\S6]{Pas}.\\  

%One can ask if, at least the category $\aCat_{\tiny\mbox{strict}}$ has a model structure, since we have an explicit description of equalizers and products (see section \ref{cocompletone}).\
%This will be done in \cite{Orn4}.\ 

\end{rem}

\newpage

\subsection{Homotopy theory of the category of $\Ain$categories}\label{hyhyhhyhy}
We are finally ready to describe the homotopy of the category of $\Ain$categories in a very explicit way.\\

%\begin{rem}
%Note how to prove that the GZ localization exists Note section \ref{highercats}, in particular the proof of Theorem \ref{mspaces} the inclusion $[1]\to\overline{[1]}$.
%\end{rem}

First, we summarize a few of results of \cite{COS2}.\\ 
We have an adjunction:
\begin{align*}
\xymatrix{
U^n:\aCat^{\tiny\mbox{cu}}\ar@<-0.5ex>[r]&\ar@<-0.5ex>[l]\DgCat^{\tiny\mbox{cu}}:i^n
}
\end{align*}
We denote the unit and the counit of the adjunction $U^n\dashv i^n$ respectively by $\alpha^n$ and $\beta^n$.\\
Fixed a cohomological unital $\Ain$category $\A$, the $\Ain$functor $\alpha^n_{\A}:\A\to U^n(\A)$ has an 
inverse $\epsilon^n_{\A}:U^n(\A)\to\A$ up to $\approx$, see \cite[Proposition 2.1 + Remark 2.2]{COS2}.\\ 
\\
On the other hand, we have a functor $U:\aCat\to\DgCat$ providing an adjunction of categories
\begin{align*}
\xymatrix{
U:\aCat\ar@<-0.5ex>[r]&\ar@<-0.5ex>[l]\DgCat:i
}
\end{align*}
We denote by $\alpha$ and $\beta$, respectively the unit and the counit of the adjunction $U\dashv i$.\\
Fixed a strictly unital $\Ain$category $\A$, the (strictly unital) $\Ain$functor $\alpha_{\A}:\A\to U(\A)$ has an 
inverse $\epsilon_{\A}:U(\A)\to\A$ up to $\approx$, if $\A$ has nice unit (see \cite[Proposition 3.3]{COS2}.\\
\\
The constructions of $U^n$ and $U$ can be found in \cite[\S2.7]{Orn1}.\\
\\
Moreover we recall \cite[Proposition 4.10]{COS2}:

\begin{lem}\label{canzo}
Given two DG categories $\A$ and $\B$, with $\A$ h-projective, we have:
\begin{align}
\mbox{Ho}(\DgCat)(\A,\B)\simeq\aCat^{\tiny\mbox{u}}(\A,\B)/\approx.
\end{align}
\end{lem}

We prove the following:

\begin{lem}\label{crunzoz}
Given two $\Ain$categories $\A^{\sem}$, $\B^{\sem}$, we have:
\begin{align*}
\mbox{Ho}(\aCat)(\A^{\sem},\B^{\sem})\simeq\aCat^{\tiny\mbox{u}}(\A^{\sem},\B^{\sem})/\approx.
\end{align*}
\end{lem}

\begin{proof}
We have the following sequence of isomorphisms of sets:
\begin{align*}
\mbox{Ho}(\aCat)(\A^{\sem},\B^{\sem})&\simeq^{\dagger_1}\mbox{Ho}(\DgCat)\big(U(\A^{\sem}),U(\B^{\sem})\big)\\
&\simeq^{\dagger_2}\mbox{Ho}(\DgCat)\big((U(\A^{\sem}))^{\tiny\mbox{sf}_{\tiny\mbox{DG}}},U(\B^{\tiny\mbox{sf}})\big)\\
&\simeq^{\dagger_3}\aCat^{\tiny\mbox{u}}\big( (U(\A^{\sem}))^{\tiny\mbox{sf}_{\tiny\mbox{DG}}},U(\B^{\sem})\big)/\approx\\
&\simeq^{\dagger_4}\aCat^{\tiny\mbox{u}}( \A^{\sem},\B^{\sem})/\approx
\end{align*}
The five isomorphisms are:
$\dagger_1$ is \cite[Theorem B]{COS2}, $\dagger_2$ is clear since $(U(\A^{\sem}))^{\tiny\mbox{sf}_{\tiny\mbox{DG}}}$ is the DG semi-free resolution of $(U(\A^{\sem}))$, $\dagger_3$ is Lemma \ref{canzo}.\ 
To conclude $\dagger_4$ comes from the composition of quasi-equivalences
\begin{align*}
\xymatrix{
(U(\A^{\sem}))^{\tiny\mbox{sf}_{\tiny\mbox{DG}}}\ar[r]&U(\A^{\sem})\ar[r]^{}&\A^{\sem}
}
\end{align*}
and Theorem \ref{DC}, noting that $(U(\A^{\sem}))^{\tiny\mbox{sf}_{\tiny\mbox{DG}}}$ and $\A^{\sem}$ are h-projective (see Remark \ref{DGSEM} or \ref{TOnzo} and Theorem \ref{h-proj}).
\end{proof}

%\begin{thm}
%Let $\A^{\tiny\mbox{sf}}$ and $\B$ be two strictly unital $\Ain$categories with $\A^{\tiny\mbox{sf}}$ semi-free and two strictly unital $\Ain$functors $\F,\G:\A^{\tiny\mbox{sf}}\to\B$.\
%If $[\F]=[\G]$ in $\mbox{Ho}(\aCat)$ then $\F\approx\G$.
%\end{thm}

%\begin{proof}
%Since $[\F]=[\G]$ in $\mbox{Ho}(\aCat)(\A^{\tiny\mbox{sf}},\B)$ then we have 
%\[
%\xymatrix{
%&\B\ar[d]^{u}_{\simeq}&\\
%\A^{\tiny\mbox{sf}}\ar[r]^{\delta}\ar[ur]^{\F}\ar[dr]_{\G}&\C&\ar[l]^u_{\simeq}\B\ar@{=}[ul]\ar@{=}[dl]\\
%&\B\ar[u]_{u}^{\simeq}&
%}
%\]
%where $u$ is a quasi-equivalence.\ We have $\delta=u\cdot\F=u\cdot\G$.\\
%Now we take the semi-free resolutions of $\B$ and $\C$, getting
%\[
%\xymatrix{
%\B^{\tiny\mbox{sf}}\ar@{->>}[d]^{\Psi^{\B}}_{\simeq}\ar@{-->}[r]^{\exists\tilde{u}}&\C^{\tiny\mbox{sf}}\ar@{->>}[d]^{\Psi^{\C}}_{\simeq}\\
%\B\ar[r]^u_{\simeq}&\C\\
%&\A^{\tiny\mbox{sf}}\ar[u]^{\delta}\ar[ul]^{\F}\ar@/^4.5pc/[uul]^{\tilde{\F}}
%}
%\]
%We have $\tilde{u}\cdot\tilde{\F}=\tilde{u}\cdot\tilde{\G}$, it implies $\tilde{\F}\approx\tilde{\G}$ and $\Psi^{\B}\cdot\tilde{\F}\approx\Psi^{\B}\cdot\tilde{\G}$ so $\F\approx\G$.
%\end{proof}

Let us prove the following easy Lemma:

\begin{lem}\label{donderbong}
The sets $\aCat^{\tiny\mbox{cu}}(\A^{\tiny\mbox{sf}},\B)/\approx$, $\aCat^{\tiny\mbox{u}}(\A^{\tiny\mbox{sf}},\B)/\approx$ and $\aCat(\A^{\tiny\mbox{sf}},\B)/\approx$ are isomorphic.\ 
\end{lem}
\begin{proof}
It suffices to prove that given a cohomological unital $\Ain$functor $\F:\A^{\sem}\to\B$, there exists $\F'$ such that $\F\approx\F'$.\
We can use \cite[Lemma 4.1]{COS2} (or implicitly \cite[Lemma 3.7]{Orn1}) + Lemma \ref{sim} taking into account that $\A^{\tiny\mbox{sf}}$ has a nice unit (Corollary \ref{augmo}).
\end{proof}

We proved that the resolutions $(\mbox{-})^{\sem}$ and $(\mbox{-})^{\tiny\mbox{cm}}$ are functorial in the categories with strict $\Ain$functors.\
Making use of (\ref{Dianzo3}) we can make these resolutions functorial in the category $\aCat^{\star}/\approx$.\\
On the other hand, denoting by $\aCat^{\diamond}_{\tiny\mbox{cm}}$ the category of $\Ain$categories with cofibrant morphisms (whose functors are the ones in $\diamond=\mathcal{f}\mbox{cu, u, nu}\mathcal{g}$) and by $\aCat_{\sem}$ the category of semi-free $\Ain$categories, we have the inclusion functors 
\begin{align*}
\aCat^{\diamond}_{\tiny\mbox{cm}}\hookrightarrow \aCat^{\diamond}
\end{align*}
and
\begin{align*}
\aCat_{\sem}\hookrightarrow \aCat.
\end{align*}
They induce the functors 
\begin{align}\label{tryru1}
\aCat^{\diamond}_{\tiny\mbox{cm}}/\approx \hookrightarrow \aCat^{\diamond}/\approx
\end{align}
and 
\begin{align}\label{tryru2}
\aCat_{\sem}/\approx\hookrightarrow \aCat/\approx.
\end{align}

\begin{thm}\label{funtocat}
We have a functor
\begin{align*}
(\mbox{-})^{\tiny\mbox{cm}}_{\approx}:\aCat^{\diamond}/\approx &\to \aCat^{\diamond}_{\tiny\mbox{cm}}/\approx\\
\A&\mapsto \A^{\tiny\mbox{cm}}\\
\F&\mapsto \F^{\tiny\mbox{cm}}_{\approx}.
\end{align*}
and a functor 
\begin{align*}
(\mbox{-})^{\sem}_{\approx}:\aCat/\approx &\to \aCat_{\sem}/\approx\\
\A&\mapsto \A^{\sem}\\
\F&\mapsto \F^{\sem}_{\approx}.
\end{align*}
which are the left inverses of (\ref{tryru1}) and (\ref{tryru2}) respectively.
\end{thm}

\begin{proof}
Let $\F:\A\to\B$ an $\Ain$functor, we have a diagram
\begin{align*}%\label{Dianzo4}
\xymatrix@=5em{
\big({U^{n}(\A^{\sem})}\big)^{\tiny\mbox{cm}}\ar[r]_{U^n(\Psi_{\A})^{\tiny\mbox{cm}}}\ar@<0.8ex>[d]^{\Psi_{U^n(\A^{\sem})}}&\big({U^{n}(\A)}\big)^{\tiny\mbox{cm}}\ar[r]_{({U^n(\F)} )^{\tiny\mbox{cm}}}\ar[d]^{\Psi_{U^n(\A)}}&\big({U^n(\B)}\big)^{\tiny\mbox{cm}}\ar[d]^{\Psi_{U^n(\B)}}\ar@<0.8ex>[r]^{(({U^n(\Psi_{\B})})^{\tiny\mbox{cm}})^{-1}}& \ar@<0.8ex>[l]^{({U^n(\Psi_{\B})})^{\tiny\mbox{cm}}} \big({U^n(\B^{\tiny\mbox{cm}})}\big)^{\tiny\mbox{cm}}\ar@<0.8ex>[d]^{\Psi_{U^n(\B^{\tiny\mbox{cm}})}}\\
U^{n}(\A^{\tiny\mbox{cm}}) \ar@<0.8ex>[u]^{\big(\Psi_{U^n(\A^{\tiny\mbox{cm}})}\big)^{-1}} \ar@<0.8ex>[d]^{\epsilon_{\A^{\tiny\mbox{cm}}}}\ar[r]_{U^n(\Psi_{\A})}& U^n(\A)\ar@<0.8ex>[d]^{\epsilon_{\A}}\ar[r]^{U^n(\F)} &U^n(\B) \ar@<0.8ex>[d]^{\epsilon_{\B}}& \ar[l]^{U^n(\Psi_{\B})} U^n(\B^{\tiny\mbox{cm}})\ar@<0.8ex>[d]^{\epsilon_{\B^{\tiny\mbox{cm}}}} \ar@<0.8ex>[u]^{\big(\Psi_{U^n(\B^{\tiny\mbox{cm}})}\big)^{-1}}\\
\A^{\sem}\ar[r]_{\Psi_{\A}} \ar@<0.8ex>[u]^{\alpha_{\A^{\sem}}}&\ar@<0.8ex>[u]^{\alpha_{\A}}\A \ar[r]_{\F}&\B \ar@<0.8ex>[u]^{\alpha_{\B}}&\B^{\tiny\mbox{cm}} \ar[l]^{\Psi_{\B}} \ar@<0.8ex>[u]^{\alpha_{\B^{\tiny\mbox{cm}}}}
}
\end{align*}
We defined 
\begin{align*}
\F^{\tiny\mbox{cm}}_{\approx}:=\epsilon_{\B^{\tiny\mbox{cm}}}\cdot \F' \cdot \alpha_{\A^{\tiny\mbox{cm}}}
\end{align*}
Where 
\begin{align*}
\F':=\Psi_{U^n(\B^{\tiny\mbox{cm}})}\cdot( U^n(\Psi_{\B}^{\tiny\mbox{cm}})^{-1}\cdot (U^n(\F))^{\tiny\mbox{cm}}\cdot (U^n(\Psi_{\A}))^{\tiny\mbox{cm}}\cdot (\Psi_{U^n(\A^{\tiny\mbox{cm}})})^{-1}.
\end{align*}
We have a functor 
\begin{align*}
(\mbox{-})^{\tiny\mbox{cm}}_{\approx}:\aCat^{\diamond}/\approx &\to \aCat^{\diamond}_{\tiny\mbox{cm}}/\approx\\
\A&\mapsto \A^{\tiny\mbox{cm}}\\
\F&\mapsto \F^{\tiny\mbox{cm}}_{\approx}.
\end{align*}
Where $\diamond=\mathcal{f}\mbox{cu, u, nu}\mathcal{g}$.\ In the same vein, taking $U$ instead of $U^n$, we have a functor 
\begin{align*}
(\mbox{-})^{\sem}_{\approx}:\aCat/\approx &\to \aCat_{\sem}/\approx\\
\A&\mapsto \A^{\sem}\\
\F&\mapsto \F^{\sem}_{\approx}.
\end{align*}
It is easy to see that the compositions
\begin{align*}
\xymatrix{
\aCat^{\diamond}_{\tiny\mbox{cm}}/\approx\ar@^{^(->}[r]&  \aCat^{\diamond}/\approx\ar@^{->>}[r] &\aCat^{\diamond}_{\tiny\mbox{cm}}/\approx.
}
\end{align*}
and 
\begin{align*}
\xymatrix{
\aCat_{\sem}/\approx\ar@^{^(->}[r]& \aCat/\approx\ar@^{->>}[r] &\aCat_{\sem}/\approx.
}
\end{align*}
are the identities.
Note that the functors $(\mbox{-})^{\sem}_{\approx}$ and $(\mbox{-})^{\tiny\mbox{cm}}_{\approx}$ are not equivalence in general.
\end{proof}

\begin{rem}
In \cite[Remark 1.8.]{Tan} the author provides a cofibrant resolution for a unital $\Ain$category\footnote{Note that the Tanaka calls "cofibrant" the h-projective $\Ain$categories} $\A$, taking the following composition: 
\begin{align*}
Y(\A)^{\tiny\mbox{cof}} \to Y(\A) \to \A .
\end{align*}
Here $Y$ is the $\Ain$ Yoneda functor (which has an inverse, up to $\approx$), see \cite[\S 14.7]{BLM}.\\ 
As the author pointed out \cite[Remark 3.38]{Tan} this construction is non functorial, the technical reason is that, 
given an $\Ain$functor $\F:\A\to\B$, the pushforward $Y(\A)\to Y(\B)$ is a non strict $\Ain$ functor!
\end{rem}

On the other hand, given two $\Ain$categories $\A$ and $\B$, we have a map:
\begin{align*}
\delta:\aCat(\A,\B)/\approx&\to\Ho(\aCat)(\A,\B)\\
\F&\mapsto [\F]
\end{align*}
since, if $\F\approx\G$ then $[\F]=[\G]$ (see Theorem \ref{mspaces}).\ 
We note that, since $\approx$ is compatible with the multiplication, we can extend $\delta$ to a functor of (1-)categories:
\begin{align*}
\delta:\aCat/\approx&\to\Ho(\aCat)\\
\F&\mapsto [\F]
\end{align*}
It follows easily that:
\begin{lem}\label{trrror}
Given a quasi-equivalence $u:\B\to\A$ with an inverse $v:\A\to\B$, up to $\approx$, 
then $[v]=[u]^{-1}$ in $\Ho(\aCat)(\A,\B)$.
\end{lem}
By Lemma \ref{crunzoz} we have that $\aCat(\A^{\sem},\B^{\sem})/\approx$ and $\Ho(\aCat)(\A^{\sem},\B^{\sem})$ are isomorphic.\\
Unfortunately, it is very difficult to unwind this isomorphism since it involve the model structure of $\DgCat$.\ We prove directly the following
\begin{lem}\label{rorororoor} 
Given two semi-free $\Ain$categories $\A^{\sem}$ and $\B^{\sem}$, the map
\begin{align*}
\phi:\aCat(\A^{\sem},\B^{\sem})/\approx &\to \mbox{Ho}(\aCat)(\A^{\sem},\B^{\sem})\\
\F&\mapsto [\F]
\end{align*}
is an isomorphism.\ 
\end{lem}
\begin{proof}
We want to find an inverse $\Phi$ to $\phi$.\
We recall that a morphism $\hat{\F}$ in $\mbox{Ho}(\aCat)(\A^{\sem},\B^{\sem})$ is a roof of the form: 
\begin{align*}
\xymatrix@C=0.45cm{
\A^{\sem}\ar[r]_{\F_1}&\B_1&\ar[l]^-{{\F}_{2}}_-{\simeq}\B_2\ar[r]&...\ar[r]&\B_n&\ar[l]^-{\F_n}_-{\simeq}\B^{\sem}.
}
\end{align*}
More generally, 
\begin{align}
\hat{\F}:=(\F_n)^{\dagger_{n}}* ...* (\F_1)^{\dagger_{1}}
\end{align}
where $\dagger_j\in\mathcal{f}\mbox{ 1, -1}\mathcal{g}$.\ $\F_j:\B_{j-1}\to\B_j$ is an $\Ain$functor and, if $\dagger_j={-1}$, then $\F_j$ is a quasi-equivalence.\\
We can define 
\begin{align}\label{formulbello}
\Phi(\hat{\F}):= (\Psi_{\B^{\sem}})^{-1} \big((\F_n)^{\sem}_{\approx}\big)^{\dagger_{n}}\cdot ...\cdot \big((\F_1)^{\sem}_{\approx}\big)^{\dagger_{1}} \cdot\Psi_{\B^{\sem}}.
\end{align}
Note that all the $(\F_j)^{\sem}_{\approx}$ are $\Ain$functors between semi-free $\Ain$categories.\ 
In particular, if $\dagger_{j}=-1$ then $(\F_j)^{\sem}_{\approx}$ has an inverse up to $\approx$.\ 
So $\big((\F_j)^{\sem}_{\approx}\big)^{-1}$ in (\ref{formulbello}) denotes the inverse of $(\F_j)^{\sem}_{\approx}$ in $\aCat/\approx$.\\
Since $(\mbox{-})^{\sem}_{\approx}:\aCat\to \aCat_{\sem}/\approx$ is a functor it is easy to prove that $\Psi$ respect the relations \cite[]{GZ}.\\
Namely 
\begin{align*}
\Phi \big( (\F_n)^{\dagger_{n}}*...*(\F_j)*&(\F_{j-1})* ...* (\F_1)^{\dagger_{1}}\big)=\\
&=(\Psi_{\B^{\sem}})^{-1} \big((\F_n)^{\sem}_{\approx}\big)^{\dagger_{n}}\cdot ...\cdot (\F_j)^{\sem}_{\approx}\cdot (\F_{j-1})^{\sem}_{\approx} \cdot ... \cdot \big((\F_1)^{\sem}_{\approx}\big)^{\dagger_{1}} \cdot\Psi_{\B^{\sem}}\\
&=(\Psi_{\B^{\sem}})^{-1} \big((\F_n)^{\sem}_{\approx}\big)^{\dagger_{n}}\cdot ...\cdot \big((\F_j\cdot \F_{j-1})^{\sem}_{\approx}\big) \cdot ...\cdot \big((\F_1)^{\sem}_{\approx}\big)^{\dagger_{1}} \cdot\Psi_{\B^{\sem}}\\
&=\Psi \big( (\F_n)^{\dagger_{n}}*...*(\F_j\cdot \F_{j-1})* ...* (\F_1)^{\dagger_{1}}\big)\\
\end{align*}
and
\begin{align*}
\Phi \big( (\F_n)^{\dagger_{n}}*...*(\F_j)^{\dagger_j}*&(\F_{j-1})^{-\dagger_j}* ...* (\F_1)^{\dagger_{1}}\big)=\\
&=(\Psi_{\B^{\sem}})^{-1} \big((\F_n)^{\sem}_{\approx}\big)^{\dagger_{n}}\cdot ...\cdot \big((\F_j)^{\sem}_{\approx}\big)^{\dagger_j}\cdot\big((\F_{j-1})^{\sem}_{\approx}\big)^{-\dagger_j} \cdot ... \cdot \big((\F_1)^{\sem}_{\approx}\big)^{\dagger_{1}} \cdot\Psi_{\B^{\sem}}\\
&=(\Psi_{\B^{\sem}})^{-1} \big((\F_n)^{\sem}_{\approx}\big)^{\dagger_{n}}\cdot ...\cdot \Id \cdot ...\cdot \big((\F_1)^{\sem}_{\approx}\big)^{\dagger_{1}} \cdot\Psi_{\B^{\sem}}\\
&=\Psi \big( (\F_n)^{\dagger_{n}}*... * (\F_1)^{\dagger_{1}}\big)\\
\end{align*}
Clearly $\Phi\cdot \phi=\Id$ and $\phi\cdot\Phi=\Id$.
\end{proof}

We have a commutative diagram
\begin{align}\label{mteo}
\xymatrix{
\aCat^{u}(\A^{\sem},\B^{\sem})/\approx \ar[d] \ar[r]^{\epsilon}& \aCat^{u}(\A^{\sem},\B)/\approx \ar[d]^{\delta} \\
\ar[u]_{1:1}^{\phi} \mbox{Ho}(\aCat)(\A^{\sem},\B^{\sem})\ar[r]_{\phi'}^{1:1}&\ar[l]\mbox{Ho}(\aCat)(\A^{\sem},\B)
}
\end{align}
Where $\epsilon(\F):=\Psi_{\B}\F$, and $\delta(\F)=[\F]$ it is well defined by Theorem \ref{mspaces}. 

\begin{lem}\label{dodi}
The map $\epsilon$ is bijective.
\end{lem}

The surjectivity of $\epsilon$ follows from the more general Lemma:

\begin{lem}\label{lordod}
Given two $\Ain$functor $\F:\A\to\B$ and $\G:\C\to\B$ which is a quasi-equivalence.\ 
There exists a (non unique) $\Ain$functor $\tilde{\F}:\A^{\tiny\mbox{sf}}\to\C$ making the diagram
\[
\xymatrix{
&{\C}\ar[d]_{\simeq}^{\G}\\
\A^{\tiny\mbox{sf}}\ar@{-->}[ur]^{\tilde{\F}}\ar[r]^{\F}&\B
}
\]
commutative in $\aCat/\approx$
\end{lem}

\begin{proof}[Proof of Lemma \ref{lordod}]
We consider the two following diagrams in $\aCat^{u}$:
\begin{align}\label{Dianzo1}
\xymatrix@C=4.5em{
\big({U^{n}(\A^{\sem})}\big)^{\tiny\mbox{cm}}\ar[r]_{({U^n(\F)} )^{\tiny\mbox{cm}}}\ar[d]_{\simeq}^{\Psi_{U^n(\A^{\sem})}}&\big({U^n(\B^{\sem})}\big)^{\tiny\mbox{cm}}\ar[d]_{\simeq}^{\Psi_{U^n(\B)}}& \ar[l]_{\simeq}^{({U^n(\G)})^{\tiny\mbox{cm}}} \big({U^n(\C^{\sem})}\big)^{\tiny\mbox{cm}}\ar[d]_{\simeq}^{\Psi_{U^n(\C)}}\\
U^{n}(\A^{\sem})\ar[r]_{U^n(\F)}&U^n(\B) & \ar[l]_{\simeq}^{U^n(\G)} U^n(\C)\\
}
\end{align}
and
\begin{align}\label{Dianzo2}
\xymatrix@C=4.5em{
U^{n}(\A^{\sem})\ar[r]_{U^n(\F)}&U^n(\B) & \ar[l]_{\simeq}^{U^n(\G)} U^n(\C)\\
\A^{\sem}\ar[r]_{\F} \ar@<0.8ex>[u]_{\simeq}^{\alpha^n_{\A^{\sem}}}&\B \ar@<0.8ex>[u]_{\simeq}^{\alpha^n_{\B}}&\B^{\sem} \ar[l]_{\simeq}^{\G} \ar@<0.8ex>[u]_{\simeq}^{\alpha^n_{\C}}
}
\end{align}
Note that all the vertical arrows in (\ref{Dianzo2}), $\Psi_{{U^n({\A^{\sem}})}}$ and $(U^n(\G))^{\tiny\mbox{cm}}$ in (\ref{Dianzo2}) have an inverse up to $\approx$.\
So we have the diagram in $\aCat^{u}/\approx$:
\begin{align}\label{Dianzo3}
\xymatrix@=5em{
\big({U^{n}(\A^{\sem})}\big)^{\tiny\mbox{cm}}\ar[r]_{({U^n(\F)} )^{\tiny\mbox{cm}}}\ar@<0.8ex>[d]^{\Psi_{U^n(\A^{\sem})}}&\big({U^n(\B)}\big)^{\tiny\mbox{cm}}\ar[d]^{\Psi_{U^n(\B)}} \ar@<0.8ex>[r]^{(({U^n(\G)})^{\tiny\mbox{cm}})^{-1}}& \ar@<0.8ex>[l]^{({U^n(\G)})^{\tiny\mbox{cm}}} \big({U^n(\C)}\big)^{\tiny\mbox{cm}}\ar[d]^{\Psi_{U^n(\C)}}\\
U^{n}(\A^{\sem}) \ar@<0.8ex>[u]^{\big(\Psi_{U^n(\A^{\sem})}\big)^{-1}} \ar@<0.8ex>[d]^{\epsilon_{\A^{\sem}}}\ar[r]_{U^n(\F)}&U^n(\B) \ar@<0.8ex>[d]^{\epsilon_{\B}}& \ar[l]^{U^n(\G)} U^n(\B^{\sem})\ar@<0.8ex>[d]^{\epsilon_{\C}} \\
\A^{\sem}\ar[r]_{\F} \ar@<0.8ex>[u]^{\alpha_{\A^{\sem}}}&\B \ar@<0.8ex>[u]^{\alpha_{\B}}&\C \ar[l]^{\G} \ar@<0.8ex>[u]^{\alpha_{\C}}
}
\end{align}
By diagram (\ref{Dianzo3}), we have an $\Ain$functor 
\begin{align*}
\mathscr{H}':=\Psi_{U^n(\C)}\cdot \big( (U^n(\G)^{\tiny\mbox{cm}})\big)^{-1}\cdot (U^n(\F))^{\tiny\mbox{cm}} \cdot \big( \Psi_{U^n(\A^{\sem})} \big)^{-1}.
\end{align*}
Take $\mathscr{H}:=\epsilon_{\C}\cdot\mathscr{H}'\cdot\alpha_{\A^{\sem}}$ we have:
\begin{align*}
\G\cdot\mathscr{H}&=\G\cdot\big( \epsilon_{\C}\cdot\mathscr{H}'\cdot\alpha_{\A^{\sem}} \big)\\
&=\big(\G\cdot \epsilon_{\C} \big)\cdot \big( \big(  \Psi_{U^n(\C)}\cdot \big( (U^n(\G)^{\tiny\mbox{cm}})\big)^{-1}\cdot (U^n(\F))^{\tiny\mbox{cm}} \cdot \big( \Psi_{U^n(\A^{\sem})} \big)^{-1} \big)  \cdot\alpha_{\A^{\sem}} \big)\\
&\approx\big( \epsilon_{\B} \cdot U^n(\G) \big)\cdot \big( \big(  \Psi_{U^n(\C)}\cdot \big( (U^n(\G)^{\tiny\mbox{cm}})\big)^{-1}\cdot (U^n(\F))^{\tiny\mbox{cm}} \cdot \big( \Psi_{U^n(\A^{\sem})} \big)^{-1}   \big)  \cdot\alpha_{\A^{\sem}} \big)\\
%%%%%%%%%%%%%%%
&\approx \epsilon_{\B} \cdot \big( U^n(\G) \cdot   \Psi_{U^n(\C)}\cdot \big( (U^n(\G)^{\tiny\mbox{cm}})\big)^{-1}\cdot (U^n(\F))^{\tiny\mbox{cm}} \cdot \big( \Psi_{U^n(\A^{\sem})} \big)^{-1}   \big)  \cdot\alpha_{\A^{\sem}}\\
%%%%%%%%%%%%
&\approx \epsilon_{\B} \cdot \big( \Psi_{U^n(\B)} \cdot (U^n(\F))^{\tiny\mbox{cm}} \cdot \big( \Psi_{U^n(\A^{\sem})} \big)^{-1}   \big)  \cdot\alpha_{\A^{\sem}}\\
&\approx \epsilon_{\B} \cdot U^n(\F)  \cdot\alpha_{\A^{\sem}}\\
&\approx \F.
\end{align*}
\end{proof}

\begin{proof}[Proof of Theorem \ref{dodi}]
It remains to prove that $\epsilon$ is injective we prove that $r\cdot\epsilon(\F)=\F$, where $r:=\phi\cdot\phi'\cdot\delta$
We have 
\begin{align*}
r\cdot\epsilon(\F)&:=\phi\cdot\phi'\cdot\delta(\Psi_{\B}\F)\\
&= (\phi\cdot\phi')\big( [\Psi_{\B}\F] \big)\\
&= \phi(\big( [\Psi_{\B}]^{-1}[\Psi_{\B}\F] \big) )\\
&=\phi([\F])\\
&=\F
\end{align*}
and we are done.
\end{proof}

We have the following natural isomorphism:

\begin{thm}\label{artoo}
Given $\A,\B$ two strictly unital $\Ain$categories we have:
\begin{align}
\mbox{Ho}(\aCat)(\A,\B)\simeq\mbox{Ho}(\aCat)(\A^{\tiny\mbox{sf}},\B)\simeq\aCat^{\star}(\A^{\tiny\mbox{sf}},\B)/_{\approx}.
\end{align}
Where $\star=\{\mbox{ , u, cu}\}$.\ Namely $\aCat^{\tiny\mbox{cu}}(\A^{\tiny\mbox{sf}},\B)$ denotes the set of cohomological unital $\Ain$functors, $\aCat^{\tiny\mbox{u}}(\A^{\tiny\mbox{sf}},\B)$ denotes the set of unital $\Ain$functors and $\aCat(\A^{\tiny\mbox{sf}},\B)$ denotes the set of strictly unital $\Ain$functors.
\end{thm}

\begin{proof}
Given a strictly unital $\Ain$category $\A$ we have a semi-free resolution $\A^{\tiny\mbox{sf}}$ (see Theorem \ref{semifree}).\\
The quasi-equivalence $\Psi:\A^{\tiny\mbox{sf}}\to\A$ becomes an isomorphism in the homotopy category.\
So $$\mbox{Ho}(\aCat)(\A,\B)\simeq\mbox{Ho}(\aCat)(\A^{\tiny\mbox{sf}},\B).$$
By diagram (\ref{mteo}) and Lemma \ref{dodi} we have:
$$\mbox{Ho}(\aCat)(\A^{\tiny\mbox{sf}},\B)\simeq \aCat^{\tiny\mbox{u}}/\approx(\A^{\tiny\mbox{sf}},\B).$$
\end{proof}

\begin{rem}
Note that the proof of Lemma \ref{dodi} actually proves that, if $\A$ and $\B$ are $\Ain$categories such that $\A$ is h-projective then:
$$\Ho(\aCat)(\A,\B)\simeq\aCat^{\tiny\mbox{u}}(\A,\B)/\approx.$$
Moreover, if $\A$ has nice unit, then 
$$\Ho(\aCat)(\A,\B)\simeq\aCat(\A,\B)/\approx.$$
\end{rem}

\begin{rem}
Note that our Theorem \ref{artoo} does not provide an equivalence of categories.\ 
The functor 
$$\aCat/\approx\to\Ho(\aCat)$$
is not indeed an equivalence of categories, see \cite[Remark 4.12]{COS2}.\\
%To see that one can take a cohomologically trivial $\Ain$category 
The category $\Ho(\aCat)$ is equivalent to the category $\aCat^{\tiny\mbox{u}}_{\tiny\mbox{h-proj}}/\approx$ (see \cite[Theorem B]{COS2}).\ 
Where $\aCat^{\tiny\mbox{u}}_{\tiny\mbox{h-proj}}/\approx$ is the category whose objects, are the h-projective (strictly unital) $\Ain$categories, and whose morphisms 
are the cohomological unital $\Ain$functors quotiented by the relation $\approx$.\\  
Moreover, fixed two $\Ain$categories $\A$ and $\B$ the set 
$$\aCat(\A^{\sem},\B)/{\approx}$$
can be equipped with an $\Ain$structure (see Section \ref{lonzolenzo}): 
\begin{itemize}
\item[1.] the objects are the (strictly unital) $\Ain$functors quotiented by $\approx$, 
\item[2.] The morphisms are the strictly unital prenatural transformations.
\end{itemize}
On the other hand even the set 
$$\aCat^{\tiny\mbox{u}}(\A^{\sem},\B)/{\approx}$$
has an $\Ain$structure:
\begin{itemize}
\item[1.] the objects are the cohomological unital $\Ain$functors quotiented by $\approx$, 
\item[2.] The morphisms are the prenatural transformations.
\end{itemize}
Now we consider $\aCat(\A^{\sem},\B)$ and $\aCat^{\tiny\mbox{u}}(\A^{\sem},\B)/{\approx}$ with their $\Ain$structures.\
Since $\A^{\tiny\sem}$ has nice unit, it is clear that $\aCat(\A^{\sem},\B)$ and $\aCat^{\tiny\mbox{u}}(\A^{\sem},\B)/{\approx}$ have the same objects (see Lemma \ref{donderbong}), 
moreover and they are quasi-equivalent has $\Ain$categories (see \cite[Lemma 4.2]{COS2} or \cite[Lemma 3.8]{Orn1}, see also \cite[Remark 5.3]{COS2}).
\end{rem}

\end{document}